\pdfoutput = 1

\documentclass[11pt]{article}
\usepackage{fullpage}

\usepackage[utf8]{inputenc} 
\usepackage[T1]{fontenc}    
\usepackage{hyperref}       
\usepackage{url}            
\usepackage{booktabs}       
\usepackage{amsfonts}       
\usepackage{nicefrac}       
\usepackage{microtype}      
\usepackage{xcolor}         

\usepackage[pdftex]{graphicx}
\usepackage{subcaption}

\usepackage{amsfonts}       
\usepackage{amssymb}
\usepackage{amsmath}
\usepackage{amsthm}
\usepackage{mathtools}
\usepackage{multicol}
\usepackage{multirow}

\usepackage{color}

\usepackage{array}
\usepackage{breakcites}
\usepackage[shortlabels]{enumitem}

\usepackage{bbm}
\usepackage{stackengine}
\usepackage{algorithmic}
\usepackage{algorithm}

\usepackage[symbol]{footmisc} 

\def\M{\mathcal{M}}

\def\L{\mathcal{L}}
\def\D{\mathrm{D}}

\def\bzero{{\mathbf 0}}
\def\bone{{\mathbf 1}}

\def\ba{{\mathbf a}}
\def\bb{{\mathbf b}}
\def\bc{{\mathbf c}}

\def\bp{{\mathbf p}}
\def\bq{{\mathbf q}}

\def\bx{{\mathbf x}}
\def\by{{\mathbf y}}

\def\bA{\mathbf A}
\def\bB{\mathbf B}
\def\bC{\mathbf C}
\def\bD{\mathbf D}

\def\bI{{\mathbf I}}

\def\bL{\mathbf L}

\def\bP{{\mathbf P}}
\def\bQ{{\mathbf Q}}
\def\bR{{\mathbf R}}
\def\bS{\mathbf S}
\def\bT{\mathbf T}
\def\bU{{\mathbf U}}
\def\bV{{\mathbf V}}

\def\bX{{\mathbf X}}
\def\bY{{\mathbf Y}}
\def\bZ{{\mathbf Z}}

\def\sR{{\mathbb R}}
\def\sS{{\mathbb S}}

\def\gD{{\mathcal{D}}}

\def\gV{{\mathcal{V}}}
\def\gX{{\mathcal{X}}}
\def\gY{{\mathcal{Y}}}

\def\grad{{\mathrm{grad}}}
\def\hess{{\mathrm{Hess}}}

\newcommand{\trace}{\mathrm{tr}}

\def\bPi{{\boldsymbol \Pi}}
\def\bLambda{{\mathbf \Lambda}}
\def\bGamma{{\mathbf \Gamma}}
\def\bTheta{{\mathbf \Theta}}
\def\bPsi{{\mathbf \Psi}}

\DeclareMathOperator*{\argmin}{arg\,min}

\def\bgamma{{\boldsymbol \gamma}}
\def\bDelta{{\mathbf \Delta}}
\def\bomega{{\boldsymbol \omega}}
\def\bxi{{\boldsymbol \xi}}
\def\sone{{\mathbbm{1}}}

\definecolor{blue-violet}{rgb}{0.00,0.75,0.90}

\newlength{\tempdima}
\newcommand{\rowname}[1]
{\rotatebox{90}{\makebox[\tempdima][c]{{#1}}}}

\theoremstyle{plain}
\newtheorem{theorem}{Theorem}[section]
\newtheorem{proposition}[theorem]{Proposition}

\theoremstyle{definition}

\newtheorem{assumption}[theorem]{Assumption}
\theoremstyle{remark}

\usepackage[textsize=tiny]{todonotes}

\newcommand{\algname}{RMOT}
\newcommand{\manname}{block SPD coupling manifold}
\newcommand{\Manname}{Block SPD coupling manifold} 

\definecolor{orange}{rgb}{1,0.5,0}

\newcommand{\revision}[1]{{\color{black} #1}}

\title{Riemannian block SPD coupling manifold \\ and its application to optimal transport}

\author{Andi Han\footnote{Discipline of Business Analytics, University of Sydney. Email: \{andi.han, junbin.gao\}@sydney.edu.au. \protect\label{usyd}} \quad Bamdev Mishra\footnote{Microsoft India. Email: \{bamdevm, pratik.jawanpuria\}@microsoft.com. \protect\label{microsoft}} \quad Pratik Jawanpuria\footref{microsoft} \quad Junbin Gao\footref{usyd}}

\date{}

\begin{document}

\maketitle

\begin{abstract}
In this work, we study the optimal transport (OT) problem between symmetric positive definite (SPD) matrix-valued measures. We formulate the above as a generalized optimal transport problem where the cost, the marginals, and the coupling are represented as block matrices and each component block is a SPD matrix. The summation of row blocks and column blocks in the coupling matrix are constrained by the given block-SPD marginals. We endow the set of such block-coupling matrices with a novel Riemannian manifold structure. This allows to exploit the versatile Riemannian optimization framework to solve generic SPD matrix-valued OT problems. We illustrate the usefulness of the proposed approach in several applications.

\end{abstract}

\section{Introduction}

Optimal transport (OT) offers a systematic approach to compare probability distributions by finding a transport plan (coupling) that minimizes the cost of transporting mass from one distribution to another. It has been successfully applied in a wide range of fields, such as computer graphics \cite{solomon2015convolutional,solomon2014earth}, graph representation learning \cite{chen2020graph,petric2019got}, text classification \cite{yurochkin2019hierarchical}, domain adaptation \cite{courty2016optimal,courty2014domain,nath20a}, cross-lingual translation~\cite{melis18a,jawanpuria20a} and prototype selection~\cite{gurumoorthy21a} to name a few. OT based distances have also used as loss functions in both discriminative and generative learning settings~\cite{frogner15a,arjovsky17a,genevay18a,jawanpuria21a}



Despite the popularity of OT, existing OT formulations are mostly limited to scalar-valued distributions. 
\revision{
On the other hand, many applications involve symmetric positive definite (SPD) matrix-valued distributions. In diffusion tensor imaging \cite{le2001diffusion}, the local diffusion of water molecules in human brain are encoded in fields of SPD matrices \cite{assaf2008diffusion}. In image processing, region information of an image can be effectively captured through several SPD covariance descriptors \cite{tuzel2006region}. For the application of image set/video classification, each set of images/frames can be represented by its covariance matrix, which has shown great promise in modelling the intra-set variations \cite{huang2015log,harandi2014manifold}. In addition, fields of SPD matrices are also important in computer graphics for anisotropic diffusion \cite{weickert1998anisotropic}, remeshing \cite{alliez2003anisotropic} and texture synthesis \cite{galerne2010random}, just to name a few. In all such cases, being able to compare fields represented by SPD matrices is crucial.
}
This, however, requires a nontrivial generalization of existing (scalar-valued) optimal transport framework with careful construction of cost and transport plan. 


In the quantum mechanics setting, existing works~\cite{jiang2012distances,carlen2014analog,chen2017matrix,chen2018vector} have explored geodesic formulation of the Wasserstein distance between vector and matrix-valued densities. In \cite{ning2013thesis,ning2014matrix}, the Monge-Kantorovich optimal mass transport problem has been studied for comparing matrix-valued power spectra measures. Recently, \cite{peyre2016quantum} proposed to solve an unbalanced optimal transport problem for SPD-valued distributions of unequal masses. 






In this paper, we propose a general framework for solving the balanced OT problem between SPD-valued distributions, where the cost and the coupling are represented as block SPD matrices. We discuss a Riemannian manifold structure for the set of such block coupling matrices, and we are able to use the Riemannian optimization framework \cite{absil2009optimization, boumal2020intromanifolds} to solve various generalized OT problems. Specifically, our {contributions} \revision{are} as follows. 
\begin{enumerate}
    \item We introduce the general SPD matrix-valued balanced OT problem for SPD matrix-valued marginals and study its metric properties for a specific setting. 
    
    \item We propose a novel manifold structure for the set of block matrix coupling matrices, which generalizes the manifold structures studied in \cite{douik2019manifold,shi2021coupling,mishra2021manifold,mishra2019riemannian}. We discuss optimization-related ingredients like Riemannian metric, Riemannian gradient, Hessian, and retraction.
    
    \item We extend our SPD-valued balanced OT formulation to block SPD Wasserstein barycenter and Gromov-Wasserstein OT.
    
    \item We empirically illustrate the benefit of the proposed framework in domain adaptation, tensor-valued shape interpolation, and displacement interpolation between tensor fields.
\end{enumerate}

{\bf Organizations.} We start with a brief review of Riemannian optimization and SPD matrix-valued optimal transport problem in Section \ref{preliminary_sect}. In Section \ref{matrix_generalized_ot_sect}, we introduce the generalized SPD matrix-valued OT problem and define the proposed {\manname}. Section \ref{manifold_geometry_sect} discusses the Riemannian structure of the proposed manifold and derives the necessary optimization-related ingredients. Section \ref{sec:applications} presents two additional OT related applications of the proposed {\Manname}. In Section \ref{Experiment_sect}, we empirically evaluate the proposed approach in various applications. Section~\ref{sec:conclusion} concludes the paper. In the appendix sections, we provide the proofs and present additional experiments. 


\section{Preliminaries}
\label{preliminary_sect}

\subsection{Riemannian optimization}
A matrix manifold $\M$ is a smooth subset of the ambient vector space $\gV$ with local bijectivity to the Euclidean space. A Riemannian manifold is a manifold endowed with a Riemannian metric (a smooth, symmetric positive definite inner product structure $\langle \cdot, \cdot\rangle_x$) on every tangent space $T_x\M$. The induced norm on the tangent space is thus $\| u\|_x = \sqrt{\langle u,u \rangle_x}$. 

The orthogonal projection operation for an embedded matrix manifold ${\rm P}_{x}: \gV \xrightarrow{} T_x\M$ is a projection that is orthogonal with respect to the Riemannian metric $\langle \cdot, \cdot \rangle_x$. Retraction is a smooth map from tangent space to the manifold
That is, for any $x \in \M$, retraction $R_x: T_x\M \xrightarrow{} \M$ such that 1) $R_x(0) = x$ and 2) $\D R_x(0)[u] = u$, where $\D f(x)[u]$ is the derivative of a function at $x$ along direction $u$. 

The Riemannian gradient of a function $F:\M \xrightarrow{} \sR$ at $x$, denoted as $\grad F(x)$, generalizes the notion of the Euclidean gradient $\nabla F(x)$. It is defined as the unique tangent vector satisfying $\langle \grad F(x) , u \rangle_x = \D F(x)[u] = \langle \nabla F(x) ,u\rangle_2$ for any $u \in T_x\M$, where $\langle \cdot, \cdot\rangle_2$ denotes the Euclidean inner product. To minimize the function, Riemannian gradient descent \cite{absil2009optimization} and other first-order solvers apply retraction to update the iterates along the direction of negative Riemannian gradient while staying on the manifold, i.e., $x_{t+1} = R_{x_t}(-\eta \, \grad F(x_t))$, where $\eta$ is the step size. 
Similarly, the Riemannian Hessian $\hess F(x) : T_x\M \xrightarrow{} T_x\M$ is defined as the covariant derivative of Riemannian gradient. Popular second-order methods, such as trust regions and cubic regularized Newton's methods have been adapted to Riemannian optimization \cite{absil2007trust,agarwal2018adaptive}. 


\subsection{Scalar-valued optimal transport}
\label{ot_preliminary_sect}

Consider two discrete measures supported on $\sR^d$, $\mu = \sum_{i = 1}^m p_i \delta_{\bx_i}$, $\nu = \sum_{j = 1}^n q_j \delta_{\by_j}$, where $\bx_i, \by_j \in \sR^d$ and $\delta_\bx$ is the Dirac at $\bx$. The weights $\bp \in \Sigma_m, \bq \in \Sigma_n$ are in probability simplex \revision{where $\Sigma_k \coloneqq \{ \bp \in \sR^k : p_i \geq 0, \sum_i p_i = 1 \}$}.
The $2$-Wasserstein distance between $\mu, \nu$ is given by solving the Monge-Kantorovich optimal transport problem: 
\begin{equation}
\label{W22}
    {\mathrm W}_2^2(\bp,\bq) = \min_{\bgamma \in \Pi(\bp,\bq)} \sum_{i,j} \|\bx_i-\by_j\|^2 \gamma_{i,j},
\end{equation}
where $\Pi(\bp,\bq) \coloneqq \{\bgamma \in \sR^{m \times n} : \bgamma \geq 0, \bgamma \bone = \bp, \bgamma^\top \bone = \bq  \}$ is the space of joint distribution between the source and
the target marginals. An optimal solution of (\ref{W22}) is referred to as an optimal transport plan (or coupling).  
Recently, \cite{cuturi2013sinkhorn} proposed the Sinkhorn-Knopp  algorithm~\cite{sinkhorn1964relationship,knight2008sinkhorn} for entropy-regularized OT formulation. In case $\mu$ and $\nu$ are measures (i.e., the setting is not restricted to probability measures), it may happen that they are of unequal masses. OT in this case is termed as unbalanced optimal transport \cite{chizat18a,Liero2018}. 
For a recent survey of OT literature and related machine learning applications, please refer to~\cite{peyre2019computational}. 




\subsection{SPD matrix-valued optimal transport}\label{sec:preliminaries_spdOT}
A SPD matrix-valued measure is a generalization of the (scalar-valued) probability measure (discussed in Section~\ref{ot_preliminary_sect}). Let us consider a SPD matrix-valued measure $M$ and a scalar-valued measure $\mu$  defined on a space $\mathcal{X}$. Let $A$ be a measurable subset of $\mathcal{X}$. Then, while $\mu(A)$ is a non-negative scalar, the ``mass'' $M(A)\in\sS_{+}^d$, where $\sS_{+}^d$ denotes the set of $d\times d$ positive semi-definite matrices. SPD matrix-valued measures have been employed in applications such as diffusion tensor imaging \cite{le2001diffusion}, image set classification~\cite{huang2015log,harandi2014manifold}, 
anisotropic diffusion \cite{weickert1998anisotropic}, and brain imaging~\cite{assaf2008diffusion}, to name a few. 

Recent works~\cite{carlen2014analog,chen2017matrix,ryu18a,peyre2016quantum} have explored optimal transport formulations for SPD matrix-valued measures. While the works \cite{carlen2014analog,chen2017matrix,ryu18a} discuss dynamical (geodesic) OT framework, \cite{peyre2016quantum} studies the ``static'' OT formulation that learns a suitable joint coupling between the input SPD matrix-valued measures. However, \cite{peyre2016quantum} explores an unbalanced OT setup for SPD matrix-valued measures and term it as quantum optimal transport (QOT). Thus, the marginals of the (learned) joint coupling in QOT is not equal to the input SPD matrix-valued measures. As in case of unbalanced (scalar-valued) OT \cite{chizat18a,Liero2018}, the discrepancy between marginals of the joint and the input measures in QOT is penalized via the Kulback-Leibler divergence (for SPD matrix-valued measures). 


\section{Block SPD optimal transport} 
\label{matrix_generalized_ot_sect}
In this section, we study a balanced OT formulation for SPD matrix-valued measures. Consider $\bP$ and $\bQ$ to be ($d$-dimensional) SPD matrix-valued input measures. Let $\bP\coloneqq\{[\bP_i]_{m\times 1}:\bP_i \in \sS_{++}^d\}$ and $\bQ\coloneqq\{[\bQ_j]_{n\times 1}: \revision{\bQ_j} \in \sS_{++}^d\}$ and $\bP$ and $\bQ$ have the same total mass. Without loss of generality, we assume $\sum_i \bP_i = \sum_j \bQ_j = \bI$. Here, $[\cdot]_{m \times n}$ denotes a collection of $mn$ matrices organized as a block matrix and $\bI$ represents the identity matrix. 
The cost of transporting a positive definite matrix-valued mass $\bA$ from position $\bx_i$ (in source space) to $\by_j$ (in target space) is parameterized by a (given) positive semi-definite matrix $\bC_{i,j}$ and is computed as $\trace(\bC_{i,j}\bA)$. Under this setting, we propose the block SPD matrix-valued balanced OT problem as 
\begin{equation}\label{reformulated_mw}
    {\rm MW}^2(\bP, \bQ) \coloneqq \min_{\bGamma \in \bPi(m,n,d, \bP, \bQ)  } \sum_{i,j} \trace( \bC_{i,j} \bGamma_{i,j}), 
\end{equation}
where $\bGamma = [\bGamma_{i,j}]_{m \times n}$ is a block-matrix coupling of size $m \times n$ and the set of such couplings are defined as $\bPi(m,n,d,\bP, \bQ) \coloneqq \{ [\bGamma_{i,j}]_{m \times n} : \bGamma_{i,j} \in \sS_{+}^d, \sum_j \bGamma_{i,j} = \bP_i, \sum_i \bGamma_{i,j} = \bQ_j, \forall i \in [m], j \in [n] \}$. Here $\sS_{+}^d$ is used to denote the set of $d\times d$ positive semi-definite matrices and $\trace(\cdot)$ is the matrix trace. The problem is well-defined provided that the corresponding coupling constraint set $\bPi(m,n,d,\bP, \bQ)$ is non-empty. For arbitrary SPD marginals $\bP, \bQ$, there is no guarantee that the set $\bPi(m,n,d,\bP, \bQ)$ defined in (\ref{reformulated_mw}) is not empty \cite{ning2014matrix}. Hence, in this work, we assume that the given marginals $\bP$ and $\bQ$ are such that $\bPi(m,n,d,\bP, \bQ)$ is not empty. In Section \ref{retraction_sinkhorn} later, we discuss a block matrix balancing algorithm which can be used to check whether $\bPi(m,n,d,\bP, \bQ)$ is empty or not for given marginals $\bP$ and $\bQ$.

\subsection{Metric properties of ${\rm MW}(\bP, \bQ)$} 
In the following result, we show that ${\rm MW}(\bP, \bQ)$ is a valid distance metric for a special case of block SPD marginals.
\begin{proposition}
\label{mw_distance_prop}
Suppose the input SPD matrix-valued marginals have the same support size $n$ and the costs $\{\bC_{i,j}\}_{i,j=1}^n$ satisfy
\begin{enumerate}
    \item $\bC_{i,j}=\bC_{j,i}$ and
    \item $\bC_{i,j}\succ\bzero$ for $i\neq j$ and $\bC_{i,j}=\bzero$ for $i=j$,
    \item $\forall (i,j,k)\in [n]^3,$ and $ \bA\succeq \bzero,$ $\sqrt{\trace( \bC_{i,j} \bA)} \leq \sqrt{\trace( \bC_{i,k} \bA)} + \sqrt{\trace(  \bC_{j,k}  \bA)}$.
\end{enumerate}
Then, ${\rm MW}(\bP, \bQ)$ is a metric between the SPD matrix-valued marginals $\bP$ and $\bQ$ defined as $\bP\coloneqq\{[\bP_i]_{m\times 1}:\bP_i = p_i \bI\}$ and $\bQ\coloneqq\{[\bQ_j]_{n\times 1}:\bQ_j = q_i \bI\}$, where $\bp, \bq \in \Sigma_n$ and $\bI$ is the $d\times d$ identity matrix. 
\end{proposition}

We remark that the conditions on $\bC_{i,j}$ in Proposition \ref{mw_distance_prop} generalize the conditions required for ${\rm W}_2(\bp, \bq)$ in \eqref{W22} to be a metric. See for example \cite[Proposition 2.2]{peyre2019computational}. In Appendix \ref{cost_choice_appendix}, we discuss some particular constructions of the cost that satisfy the conditions.

\subsection{Manifold structure for the coupling set $\bPi(m,n,d,\bP,\bQ)$}

We next analyze the coupling constraint set $\bPi(m,n,d,\bP, \bQ)$ and show that it can be endowed with a manifold structure. This allows to exploit the versatile Riemannian optimization framework to solve (\ref{reformulated_mw}) and any more general problem \cite{absil2009optimization}. 

We propose the following manifold structure,
termed as the {\manname}, 
\begin{equation}\label{eq:manifold_proposed}
    \M_{m,n}^d(\bP, \bQ) 
    \coloneqq \{ \bGamma : \bGamma_{i,j} \in \sS_{++}^d, \sum_{j} \bGamma_{i,j} = \bP_i, 
     \sum_i \bGamma_{i,j} =\bQ_j \},
\end{equation}
where $\sum_i \bP_i = \sum_j \bQ_j = \bI$. Particularly, we restrict $\bP_i, \bQ_j,\bGamma_{i,j} \in \sS_{++}^d$, the set of SPD matrices. This ensures that the proposed manifold $\M_{m,n}^d(\bP, \bQ)$ in (\ref{eq:manifold_proposed}) is the interior of the set $\bPi(m,n,d,\bP, \bQ)$.

As discussed earlier $\bPi(m,n,d,\bP,\bQ)$ is not guaranteed to be non-empty for arbitrary choices of block SPD marginals $\bP$ and $\bQ$~\cite{ning2013thesis}. To this end, we assume that the marginals $\bP$ and $\bQ$ that are given ensure feasibility of the set $\bPi(m,n,d,\bP,\bQ)$. In particular, the manifold $ \M_{m,n}^d(\bP, \bQ)$ inherits the following assumption.
\begin{assumption}\label{assumption:existence}
In this work, we consider block-SPD marginals $\bP$ and $\bQ$ such that the set $\M_{m,n}^d(\bP, \bQ)$ is not empty. 
\end{assumption}
It should be noted that Assumption \ref{assumption:existence} is trivially satisfied for diagonal SPD marginals, i.e., when $\bP_i$ and $\bQ_j$ are diagonal. However, non-diagonal SPD marginals may also satisfy Assumption \ref{assumption:existence} for many problem instances. In Section~\ref{Experiment_sect}, we discuss empirical settings where non-diagonal SPD marginals satisfying Assumption~\ref{assumption:existence} are considered. The following proposition implies that we can endow $\M^d_{m,n}(\bP, \bQ)$ with a differentiable structure.  
\begin{proposition}
\label{smooth_set_prop}
Under Assumption \ref{assumption:existence}, the set $\M^d_{m,n}(\bP, \bQ)$ is smooth, i.e., differentiable.
\end{proposition}

It should be emphasized that the proposed manifold $\M_{m,n}^d(\bP, \bQ)$ can be regarded as a generalization to existing manifold structures. For example, when $d =1$ and either $m=1$ or $n=1$, $\M_{m,n}^d(\bP, \bQ)$ reduces to the multinomial manifold of probability simplex  \cite{sun2015heterogeneous}. When $d=1$ and $m, n \neq 1$, it reduces the so-called doubly stochastic manifold \cite{douik2019manifold} with uniform marginals or the more general matrix coupling manifold \cite{shi2021coupling}. When $d > 1$ and either $m =1$ or $n =1$, our proposed manifold simplifies to the simplex manifold of SPD matrices \cite{mishra2019riemannian}.

In the next section, we derive various optimization-related ingredients on $\M_{m,n}^d(\bP, \bQ)$ that allow optimization of an arbitrary differentiable objective function on the manifold. In particular, we propose a Riemannian optimization approach following the general treatment by \cite{absil2009optimization, boumal2020intromanifolds}. It allows employing the proposed approach not only for (\ref{reformulated_mw}) but also for other OT problems as discussed in Section~\ref{sec:applications}.

\begin{algorithm}[t]
 \caption{Riemannian optimization for solving (\ref{eq:manifold_problem}).}
 \label{Riemannian_optimizer}
 \begin{algorithmic}[1]
 \STATE Initialize a feasible $\bGamma_0 \in \M_{m,n}^d$.
 \STATE $\bGamma = \bGamma_0$.
  \WHILE{not converging}
  \STATE Compute Riemannian gradient (and Hessian) at $\bGamma$.
  \STATE Compute the update step $\bxi \in T_{\bGamma} \M_{m,n}^d$.
   \STATE Update $\bGamma \leftarrow R_{\bGamma}(\bxi)$.
  \ENDWHILE
\STATE \textbf{Output:} $\bGamma$.
 \end{algorithmic} 
\end{algorithm}


\section{Riemannian geometry and optimization over $\M_{m,n}^d(\bP, \bQ)$}
\label{Riem_opt_sec} \label{manifold_geometry_sect}

We consider the general optimization problem
\begin{equation}
    \min_{\bGamma \in \M_{m,n}^d(\bP, \bQ)} F(\bGamma), \label{eq:manifold_problem}
\end{equation}
where $F:{\M_{m,n}^d(\bP, \bQ)} \rightarrow \mathbb{R}$ is a differentiable objective function. The proposed manifold ${\M_{m,n}^d(\bP, \bQ)}$ can be endowed with a smooth Riemannian manifold structure \cite{absil2009optimization, boumal2020intromanifolds}. Consequently, (\ref{eq:manifold_problem}) is an optimization problem on a Riemannian manifold. We solve the problem via the Riemannian optimization framework. It provides a principled class of optimization methods and computational tools for manifolds, both first order and second order, as long as the ingredients such as Riemannian metric, orthogonal projection, retraction, and Riemannian gradient (and Hessian) of a function are defined \cite{absil2009optimization, boumal2014manopt, boumal2020intromanifolds}. Conceptually, the Riemannian optimization framework treats (\ref{eq:manifold_problem}) as an ``unconstrained'' optimization problem over the constraint manifold $\M_{m,n}^d$ (omitted marginals $\bP$, $\bQ$ for clarity).

In Algorithm \ref{Riemannian_optimizer}, we outline the skeletal steps involved in optimization over $\M_{m,n}^d$, where the step $\bxi$ can be computed from different Riemannian methods. In Riemannian steepest descent, $\bxi = -\eta \, \grad F(\bGamma)$, where $\grad F(\bGamma)$ is the Riemannian gradient at $\bGamma$. Also, $\bxi$ is given by the ``conjugate'' direction of $\grad F(\bGamma)$ in the Riemannian conjugate gradient method. And, for the Riemannian trust-region method, $\bxi$ computation involves minimizing a second-order approximation of the objective function in a trust-region ball
\cite{absil2009optimization}. Below, we show the computations of these ingredients.

\subsection{Riemannian metric}
The manifold $\M^d_{m,n}$ is a submanifold of the Cartesian product of $m \times n$ SPD manifold of size $d \times d$, which we denote as $\bigtimes_{m, n} \sS_{++}^d$. The dimension of the manifold $\M^d_{m,n}$ is $(m-1)(n-1)d(d+1)/2$. The {tangent space} characterization of $\M_{m,n}^d$ at $\bGamma$ is obtained as
\begin{equation*}
    T_\bGamma \M_{m,n}^d = \{ [\bU_{i,j}]_{m \times n}: \bU_{i,j} \in \sS^d, \sum_{j} \bU_{i,j} = \bzero, \sum_{i} \bU_{i,j} = \bzero\},
\end{equation*}
where $\sS^d$ is the set of $d \times d$ symmetric matrices. The expression for the tangent space is obtained by linearizing the constraints. We endow each SPD manifold with the affine-invariant Riemannian metric \cite{bhatia2009positive}, which induces a {Riemannian metric} for the product manifold $\M_{m,n}^d$ as
\begin{equation}
    \langle \bU, \bV \rangle_\bGamma = \sum_{i,j} \trace(\bGamma_{i,j}^{-1} \bU_{i,j} \bGamma^{-1}_{i,j} \bV_{i,j}), \label{Riem_metric}
\end{equation}
for any $\bU, \bV \in T_{\bGamma} \M_{m,n}^d$. 


\subsection{Orthogonal projection, Riemannian gradient, and Riemannian Hessian}
As an embedded submanifold, the orthogonal projection plays a crucial role in deriving the Riemannian gradient (as orthogonal projection of the Euclidean gradient in the ambient space). 

\begin{proposition}
\label{orthogonal_proj}
The orthogonal projection of any $\bS \in \bigtimes_{m, n} \sS^d$ to $T_\bGamma\M_{m,n}^d$ with respect to the Riemannian metric \eqref{Riem_metric} is given by 
\begin{equation*}
     {\rm P}_{\bGamma}(\bS) = \bU, \text{ with } \bU_{i,j} = \bS_{i,j} + \bGamma_{i,j} (\bLambda_i+ \bTheta_j) \bGamma_{i,j},
\end{equation*}
where auxiliary variables $\bLambda_i, \bTheta_j$ are solved from the system of matrix linear equations:
\begin{equation*}
\begin{cases}
    -\sum_i \bS_{i,j} = \sum_{i} \bGamma_{i,j} (\bLambda_i + \bTheta_j) \bGamma_{i,j}, & \forall j \\
    -\sum_j \bS_{i,j} = \sum_{j} \bGamma_{i,j} (\bLambda_i + \bTheta_j) \bGamma_{i,j}, & \forall i.
\end{cases}
\end{equation*}
\end{proposition}

Subsequently, the Riemannian gradient and Hessian are derived as the orthogonal projection of the gradient and Hessian from the ambient space. 

\begin{proposition}
\label{Riem_grad_hess}
The Riemannian gradient and Hessian of $F: \M_{m\times n}^d \xrightarrow{} \sR$ are derived as 
\begin{align*}
    \grad F(\bGamma) &= {\rm P}_\bGamma([\bGamma_{i,j} \{ \nabla F(\bGamma_{i,j}) \}_{\rm S} \bGamma_{i,j}]_{m\times n}),\\
    \hess F(\bGamma)[\bU] &= {\rm P}_\bGamma([\D \grad F(\bGamma_{i,j})[\bU_{i,j}] - \{ \bU_{i,j} \bGamma_{i,j}^{-1} \grad F(\bGamma_{i,j}) \}_{\rm S}]_{m\times n}),
\end{align*}
where $\bU \in T_\bGamma\M_{m,n}^d$ and $\nabla F(\bGamma_{i,j})$ is the block partial derivative of $F$ with respect to $\bGamma_{i,j}$. Here, $\D \grad F(\bGamma_{i,j})[\bU_{i,j}]$ denotes the directional derivative of the Riemannian gradient $\grad F$ along $\bU$ and $\{ \bA\}_{\rm S} \coloneqq (\bA + \bA^\top)/2$.
\end{proposition}

\subsection{Retraction and block matrix balancing algorithm}
\label{retraction_sinkhorn}

The retraction operation on $\M_{m, n}^d$ is given by a composition of two operations. The first operation is to ensure positive definiteness of the blocks in the coupling matrix. In particular, we use the exponential map associated with the affine-invariant metric on the SPD manifold $\sS_{++}^d$ \cite{bhatia2009positive}. The second operation is to ensure that the summation of the row blocks and column blocks respect the block-SPD marginals. Given an initialized block SPD matrix $[\bA_{i,j}]_{m \times n}$, where $\bA_{i,j} \in \sS_{++}^d$, the goal is to find a `closest' block SPD coupling matrix $\bB \in \M_{m,n}^d$. This is achieved by alternatively normalizing the row and column blocks to the corresponding marginals. The procedure is outlined in Algorithm~\ref{bmb_algorithm}. The solution for the row and column normalization factors $\bR_j, \bL_i$, which are SPD matrices, are computed by solving the Riccati equation $\bT \bX \bT = \bY$ for given $\bX, \bY \in \sS_{++}^d$. Here, $\bT$ admits a unique solution \cite{bhatia2009positive,malago2018wasserstein}. Different from the scalar marginals case where the scaling can be expressed as a diagonal matrix, we need to symmetrically normalize each SPD block matrix. Algorithm~\ref{bmb_algorithm} is a generalization of the RAS algorithm for balancing non-negative matrices \cite{sinkhorn1967diagonal}, which is related to the popular Sinkhorn-Knopp algorithm \cite{sinkhorn1964relationship,knight2008sinkhorn}. We also use Algorithm \ref{bmb_algorithm} to test feasibility of the set $\M_{m, n}^d$ by checking whether Algorithm \ref{bmb_algorithm} outputs a balanced block SPD matrix for a random block SPD matrix $\bA$.

It should be noted that a similar matrix balancing algorithm has been introduced for positive operators \cite{gurvits2004classical,georgiou2015positive}, where the convergence is only established in limited cases. Algorithm~\ref{bmb_algorithm} is different from the quantum Sinkhorn algorithm proposed in \cite{peyre2016quantum} that applies to the {unbalanced} setting. Although we do not provide a theoretical convergence analysis for Algorithm \ref{bmb_algorithm}, we empirically observe quick convergence of this algorithm in various settings (see Appendix \ref{convergence_bmb_retr_appendix}).





\begin{algorithm}[t]
 \caption{Block matrix balancing algorithm}
 \label{bmb_algorithm}
 \begin{algorithmic}[1]
  \STATE \textbf{Input:} $[\bA_{i,j}]_{m \times n}$, where $\bA_{i,j} \in \sS_{++}^d$ and block SPD marginals $\bP$ and $\bQ$.
  \STATE Initialize $\bB = [\bA_{i,j}]_{m \times n}$. 
  \WHILE{not converging}
  \STATE Find $\bR_j \in \sS_{++}^d$ such that $\sum_i \bR_j \bB_{i,j} \bR_j = \bQ_j$, $\forall j$.
  \STATE Update $\bB_{i,j} \xleftarrow{} \bR_j \bB_{i,j} \bR_j$, $\forall j$. 
  \STATE Find $\bL_i \in \sS_{++}^d$ such that $\sum_j \bL_i \bB_{i,j} \bL_i = \bP_i$, $\forall i$.
  \STATE Update $\bB_{i,j} \xleftarrow{} \bL_i \bB_{i,j} \bL_i, \forall i$. 
  \ENDWHILE
  \STATE \textbf{Output:} The balanced matrix $\bB$.
 \end{algorithmic} 
\end{algorithm}

Based on Algorithm \ref{bmb_algorithm}, we define a retraction $R_\bGamma(\bU)$ at $\bGamma \in  \M_{m, n}^d$ for any $\bU \in T_\bGamma \M_{m,n}^d$ as
\begin{equation}
    R_\bGamma(\bU) = {\rm MBalance} \big([\bGamma_{i,j} {\rm exp}(\bGamma_{i,j}^{-1} \bU_{i,j})]_{m \times n} \big), \label{retr_define}
\end{equation}
where {\rm MBalance} calls the matrix balancing procedure in Algorithm~\ref{bmb_algorithm} and $\exp(\cdot)$ denotes the matrix exponential. The retraction proposed in \eqref{retr_define} is valid (i.e., satisfy the two conditions) for diagonal marginals and empirically we also see the retraction is well-defined for arbitrary block-SPD marginals. See Appendix \ref{convergence_bmb_retr_appendix} for more details. 






\subsection{Convergence and computational complexity}
\label{convergence_complexity_sect}

{\bf Convergence of Riemannian optimization.}
Similar to Euclidean optimization, the necessary first-order optimality condition for any differentiable $F$ on $\M_{m,n}^d$ is $\grad F(\bGamma^*) = 0$, i.e., where the Riemannian gradient vanishes. We call such $\bGamma^*$ the stationary point. The Riemannian methods are known to converge to a stationary point \cite{absil2009optimization,boumal2020intromanifolds} under standard assumptions. Additionally, we show the following.

\begin{theorem}
\label{Prop_Riem_opt_optimal}
Suppose the objective function of the problem $\min_{\bGamma \in \bPi(m,n,d,\bP, \bQ)} F(\bGamma)$ is strictly convex and the optimal solution $\bGamma^*$ is positive definite, i.e., it lies in the interior of $\bPi(m,n,d, \bP, \bQ)$. Then, Riemannian optimization (Algorithm~\ref{Riemannian_optimizer}) for (\ref{eq:manifold_problem}) converges to the same global optimal solution $\bGamma^*$. 
\end{theorem}
Theorem \ref{Prop_Riem_opt_optimal} guarantees the quality of the solution obtained by Riemannian optimization for a class of objective functions which includes the SPD matrix-valued OT problem with convex regularization.

{\bf Computational complexity.}
The complexity of each iteration of the Riemannian optimization algorithm is dominated by the computations of retraction, the Riemannian gradient, the Riemannian Hessian. These also make use of the orthogonal projection operation. All these operations cost $O(mnd^3)$. 
Since the number of parameters to be learned \revision{is} $N=mnd^2$ (size of the coupling block SPD matrix $\bGamma$), the above cost is almost linear in $N$.

\section{Applications of {\manname}}
\label{sec:applications}

As discussed earlier, we employ the proposed {\manname} optimization approach to solve the block SPD matrix valued balanced OT problem (\ref{reformulated_mw}). We now present two other OT related applications of the {\manname}: learning Wasserstein barycenters and the Gromov-Wasserstein averaging of distance matrices. 

\subsection{Block SPD Wasserstein barycenter learning}
\label{block_spd_barycenter}
We consider the problem of computing the Wasserstein barycenter of a set of block SPD matrix-valued measures. 
Let $\Delta_n(\sS_{++}^d) \coloneqq \{ \bP = [\bP_i]_{n \times 1} : \bP_i\in\sS_{++}^d, \sum_i \bP_i = \bI \}$ denotes the space of $n\times 1$ block SPD marginals. 
Then, the Wasserstein barycenter $\bar{\bP}$ of a set $\bP^\ell \in \Delta_{n_\ell}(\sS_{++}^d)$ for all $\ell=\{1, \ldots, K \}$ is computed as follows:
\begin{equation}
    \bar{\bP}= \argmin_{{\bP} \in \Delta_{n}(\sS_{++}^d)} \sum_{\ell =1}^K \omega_\ell {\rm MW}^2_\epsilon({\bP}, \bP^\ell) , \label{w_barycenter}
\end{equation}
where the given non-negative weights satisfy $\sum_\ell \omega_\ell =1$. 
It should be noted that we employ a regularized version of the proposed block SPD OT problem (\ref{reformulated_mw}) to ensure the differentiability of the objective function near boundary in~\eqref{w_barycenter}. The regularized block SPD OT problem is defined as 
\begin{equation}\label{eqn:regMWOT}
    {\rm MW}^2_\epsilon(\bP, \bQ) \coloneqq \min_{\bGamma \in \M_{m,n}^d(\bP, \bQ)} \sum_{i,j} \Big( \trace( \bC_{i,j}\bGamma_{i,j})+ \epsilon \, \Omega(\bGamma_{i,j}) \Big), 
\end{equation}
where $\epsilon>0$ is the regularization parameter and  $\Omega(\cdot)$ is a strictly convex regularization  (e.g., entropic regularization) on the block SPD coupling matrices. 


To solve for $\bar{\bP}$ in (\ref{w_barycenter}), we consider Riemannian optimization on $\Delta_n(\sS_{++}^d)$, which has recently been studied in~\cite{mishra2019riemannian}. 
The following result provides an expression for the Euclidean gradient of the objective function in problem~(\ref{w_barycenter}). 
\begin{proposition}
\label{dual_var_solution_prop}
The Euclidean gradient of \eqref{w_barycenter} with respect to ${\bP}_i$, for $i \in [n]$ is
\begin{equation*}
    \sum_{\ell=1}^K \omega_\ell \nabla_{{\bP}_i}{\rm MW}_\epsilon({\bP}, \bP^\ell) = - \sum_{\ell=1}^K \omega_\ell (\bLambda^\ell_i)^*,
\end{equation*}
where $(\bLambda^\ell_i)^*$ is given by evaluating the orthogonal projection ${\rm P}_{(\bGamma^\ell)^*}(\nabla_{(\bGamma^\ell)^*} {\rm MW}_\epsilon)$, where $\nabla_{(\bGamma^\ell_{i,j})^*} {\rm MW}_\epsilon = \bC_{i,j}^\ell + \epsilon \nabla \Omega((\bGamma^\ell_{i,j})^*)$ and $(\bGamma^\ell)^*$ is the optimal coupling for $\bP^\ell$. That is, $(\bLambda^\ell_i)^*$ is the auxiliary variable obtained during the solving of the system of matrix linear equations in Proposition \ref{orthogonal_proj}. 
\end{proposition}
The complete algorithm for computing the barycenter in \eqref{w_barycenter} is outlined in Algorithm \ref{wb_algorithm} (Appendix \ref{geometry_spd_simplex_appendix}). 



\subsection{Block SPD Gromov-Wasserstein discrepancy}
\label{sec:proposedGW}
The Gromov-Wasserstein (GW) distance~\cite{memoli2011gromov} generalizes the optimal transport to the case where the measures are supported on possibly different metric spaces $\gX$ and $\gY$. 
Let $\bD^x\in\sR^{m\times m}$ and $\bD^y\in\sR^{n\times n}$ represent the similarity (or distance) between elements in metric spaces $\gX$ and $\gY$ respectively. Let $\bp\in \Sigma_m$ and $\bq \in \Sigma_n$ be the marginals corresponding to the elements in $\gX$ and $\gY$, respectively. 
Then, the GW discrepancy between the two distance-marginal pairs $(\bD^x, \bp)$ and $(\bD^y, \bq)$ is defined as 
\begin{equation*}
\label{eqn:GW}
{\rm GW} \big( (\bD^x, \bp), ( \bD^y, \bq ) \big)  \coloneqq \min_{\bgamma \in \Pi(\bp, \bq)} \sum_{i,i',j,j'} \L(D^x_{i,i'}, D^y_{j,j'}) \gamma_{i,j} \gamma_{i',j'},
\end{equation*}
where $D_{k,l}$ denotes the $(k,l)$-th element in the matrix $\bD$ and $\L$ is a loss between the distance pairs. Common choices of $\L$ include the $L_2$ distance and the KL divergence.

We now generalize the GW framework to our setting where the marginals are SPD matrix-valued measures.  
Let $(\bD^x, \bP)$ and $(\bD^y, \bQ)$ be two distance-marginal pairs, where the Dirac measures are given by $\sum_i \bP_i \delta_{x_i}$, $\sum_j \bQ_j \delta_{y_j}$ respectively, for $\{x_i\}_{i \in [m]} \subset \gX, \{y_j\}_{j \in [n]} \subset \gY$. The marginals are tensor-valued with $\bP \in \Delta_m(\sS_{++}^d)$, $\bQ \in \Delta_n(\sS_{++}^d)$. 
%
We define the SPD generalized GW discrepancy as
\begin{equation}\label{scalar_gw_ot}
    {\rm MGW} \big( (\bD^x, \bP), ( \bD^y, \bQ ) \big)  \coloneqq \min\limits_{\bGamma \in \M^d_{m \times n}} \sum\limits_{{i,i',} {j, j'}} \L\big( D^x_{i,i'},D^y_{j, j'}  \big) \trace( \bGamma_{i,j} \bGamma_{i',j'}),
\end{equation}
where we use Riemannian optimization (Algorithm \ref{Riemannian_optimizer}) to solve problem~\eqref{scalar_gw_ot}.




%

{\bf Gromov-Wasserstein averaging of distance matrices.} The GW formulation with scalar-valued probability measures has been used for averaging distance matrices~\cite{peyre2016gromov}. Building on \eqref{scalar_gw_ot}, we consider the problem of averaging distance matrices where the marginals are SPD-valued.  
Let $\{(\bD^\ell, \bP^\ell)\}_{\ell = 1}^K$ with $\bP^\ell \in \Delta_{n_\ell}(\sS_{++}^d)$, be a set of distance-marginal pairs on $K$ incomparable domains.
Suppose the barycenter marginals $\bar{\bP} \in \Delta_{n}(\sS_{++}^d)$ are given, the goal is to find the average distance matrix $\bar{\bD}$ by solving 
\begin{equation}\label{gw_barycenter_problem}
\begin{array}{ll}
    \bar{\bD} = \argmin\limits_{\bD \in \sS^n : D_{i,j} \geq 0} \sum_{\ell=1}^K \omega_\ell \, {\rm MGW} \big( (\bD, \bar{\bP}), ( \bD^\ell, \bP^\ell ) \big),
\end{array}    
\end{equation}
where the given weights satisfy $\sum_\ell \omega_\ell = 1$.  Problem \eqref{gw_barycenter_problem} can be solved via a block coordinate descent method, that iteratively updates the couplings $\{ \bGamma^\ell \}_{\ell=1}^K$ and the distance matrix $\bar{\bD}$. The update of the coupling is performed via Algorithm \ref{Riemannian_optimizer}. For the update of the distance matrix, we show when the loss $\L$ is decomposable, including the case of $L_2$ distance or the KL divergence, the optimal $\bar{\bD}$ admits a closed-form solution. This is a generalization of the result \cite[Proposition 3]{peyre2016gromov} to SPD-valued marginals.
\begin{proposition}
\label{average_distance_update}
Suppose the loss $\L$ can be decomposed as $\L(a, b) = f_1(a) + f_2(b) - h_1(a) h_2(b)$ with $f_1'/h_1'$ invertible, then \eqref{gw_barycenter_problem} has a closed form solution given by $\bar{D}_{i,i'} = \big(\frac{f_1'}{h_1'} \big)^{-1} \big(h_{i,i'} \big)$ with
\begin{equation*}
   h_{i,i'} = \Big( \frac{\sum_{\ell=1}^K \omega_\ell \trace \big(\sum_j \bGamma_{i,j}^\ell \sum_{j'} h_2( D^\ell_{j,j'}) \bGamma^\ell_{i', j'} \big) }{\trace(\bar{\bP}_i \bar{\bP}_{i'})} \Big).
\end{equation*}
\end{proposition}

\section{Experiments}
\label{Experiment_sect}
In this section, we show the utility of the proposed framework in a number of applications. For empirical comparisons, we refer to our approaches, block SPD OT~\eqref{reformulated_mw}, the corresponding Wasserstein barycenter~\eqref{w_barycenter}, and block SPD Gromov-Wasserstein OT~\eqref{scalar_gw_ot}~\&~\eqref{gw_barycenter_problem}, collectively as RMOT (Riemannian optimized Matrix Optimal Transport). 
For all the experiments, we use the Riemannian steepest descent method using the Manopt toolbox \cite{boumal2014manopt} for implementing Algorithm~\ref{Riemannian_optimizer}. The codes are available at \url{https://github.com/andyjm3/BlockSPDOT}.

\subsection{Domain adaptation}
\label{domain_adapt_sect}
We apply our OT framework to the application of unsupervised domain adaptation where the goal is to align the distribution of the source with the target for subsequent tasks. 

Suppose we are given the source $\bp \in \Sigma_m$ and target marginals $\bq \in \Sigma_n$, along with samples $\{ \bX_i\}_{i=1}^m, \{ \bY_j\}_{j = 1}^n$ from the source and target distributions. The samples are matrix-valued, i.e., $\bX_i ,\bY_j \in \sR^{d \times s}$. We define the cost as $\bC_{i,j} = (\bX_i - \bY_j)(\bX_i -\bY_j)^\top$. It should be noted that $\trace{(\bC_{i,j})}=\|\bX_i - \bY_j\|_{\rm F}^2$ is the cost function under the $2$-Wasserstein OT setting~(\ref{W22}). 

For domain adaptation, we first learn an optimal coupling between the source and target samples by solving the proposed OT problem \eqref{reformulated_mw} with marginals $\bP, \bQ$ constructed as $\bP \coloneqq \{[\bP_i]_{m\times 1}:\bP_i = p_i \bI\}$ and $\bQ\coloneqq\{[\bQ_j]_{n\times 1}:\bQ_j = \revision{q_j} \bI\}$. Finally, the source samples are projected to the target domain via barycentric projection. Once the optimal couplings $[\bGamma_{i,j}^*]_{m\times n}$, the barycentric projection of a source sample $\bX_i$ is computed as
\begin{equation}\label{eq:barycenter_SPDOT}
    \hat{\bX}_i = \argmin\limits_{\bX_i \in \sR^{d\times s}} \sum_{i,j} \trace ( (\bX_i - \bY_j)(\bX_i - \bY_j)^\top \bGamma_{i,j}^* ) = \bP_i^{-1}(\sum_j \bGamma_{i,j}^* \bY_j).
\end{equation}
The above approach also works for structured samples. For instance, when the samples are SPD, i.e., $\bX_i, \bY_j \in \sS_{++}^d$, the projected source sample $\hat{\bX}_i$ is now the solution to the matrix Lyapunov equation: $\{ \bP_i \hat{\bX}_i \}_{\rm S} = \{\sum_j \bGamma^*_{i,j} \bY_j\}_{\rm S}$. Here, $\{ \bA\}_{\rm S} = (\bA + \bA^\top)/2$.

For the scalar-valued OT case, discussed in Section \ref{ot_preliminary_sect}, the barycentric projection of a source sample $\bX_i$ is computed as
\begin{equation}\label{eq:barycenter_scalar}
    \hat{\bX}_i = \argmin\limits_{\bX_i \in \sR^{d\times s}} \sum_{i,j} \| \bX_i - \bY_j\|_F^2 \bgamma_{i,j}^* = p_i^{-1}(\sum_j \gamma_{i,j}^* \bY_j),
\end{equation}
where $\bgamma^* = [\gamma^*_{i,j}]$ is the optimal coupling matrix of size $m\times n$ for the scalar-valued OT problem.

Contrasting the barycentric projection operations (\ref{eq:barycenter_SPDOT}) with (\ref{eq:barycenter_scalar}), we observe that (\ref{eq:barycenter_SPDOT}) allows to capture feature-specific correlations more appropriately. The benefit of the matrix-valued OT modeling over the scalar-valued OT modeling is reflected in the experiments below.

{\bf Experimental setup.}
We employ domain adaptation to classify the test sets (target) of multiclass image  datasets, where the training sets (source) have a different class distribution than the test sets. 
    Suppose we are given a training set $\{ \bX_i\}_{i=1}^m$ and a test set $\{  \bY_j\}_{j=1}^n$ where $\bX_i, \bY_j \in \sR^{d \times s}$ are $s$ (normalized) image samples of the same class in $d$ dimension for each image set $i,j$. Instead of constructing the cost directly on the input space, which are not permutation-invariant, we first compute the sample covariances $\bS_{x_i} = \bX_i \bX_i^\top/s$ and $\bS_{y_j} = \bY_j \bY_j^\top/s$, $\forall i,j$. Now the cost between $i, j$ is given by $\bC_{i,j} = (\bS_{x_i} - \bS_{y_j})(\bS_{x_i} - \bS_{y_j})^\top$. Once the block SPD matrix coupling is learnt, the $\bS_{x_i}$ covarinaces are projected using the barycerntric projection to obtain $\hat{\bS}_{x_i}, i \in [m]$. This is followed by nearest neighbour classification of $j$ based on the Frobenius distance $\| \hat{\bS}_{x_i} - \bS_{y_j} \|_{\rm F} \forall i,j$.

We compare the proposed {\algname} (\ref{reformulated_mw}) with the following baselines: (i) sOT: the $2$-Wasserstein OT~(\ref{W22}) with the cost $c_{i,j} = \trace(\bC_{i,j})=\|\bS_{x_i} - \bS_{y_j} \|_{\rm F}^2$ \cite{courty2016optimal}, and (ii) SPDOT: the $2$-Wasserstein OT (\ref{W22}) with the cost as the squared Riemannian geodesic distance between the SPD matrices $\bS_{x_i}$ and $\bS_{y_j}$ \cite{yair2019parallel}. 

\begin{figure}
\label{domain_adaptation_figure}
    \centering
    \subfloat[MNIST \label{mnist_fig}]{\includegraphics[width = 0.33\textwidth, height = 0.25\textwidth]{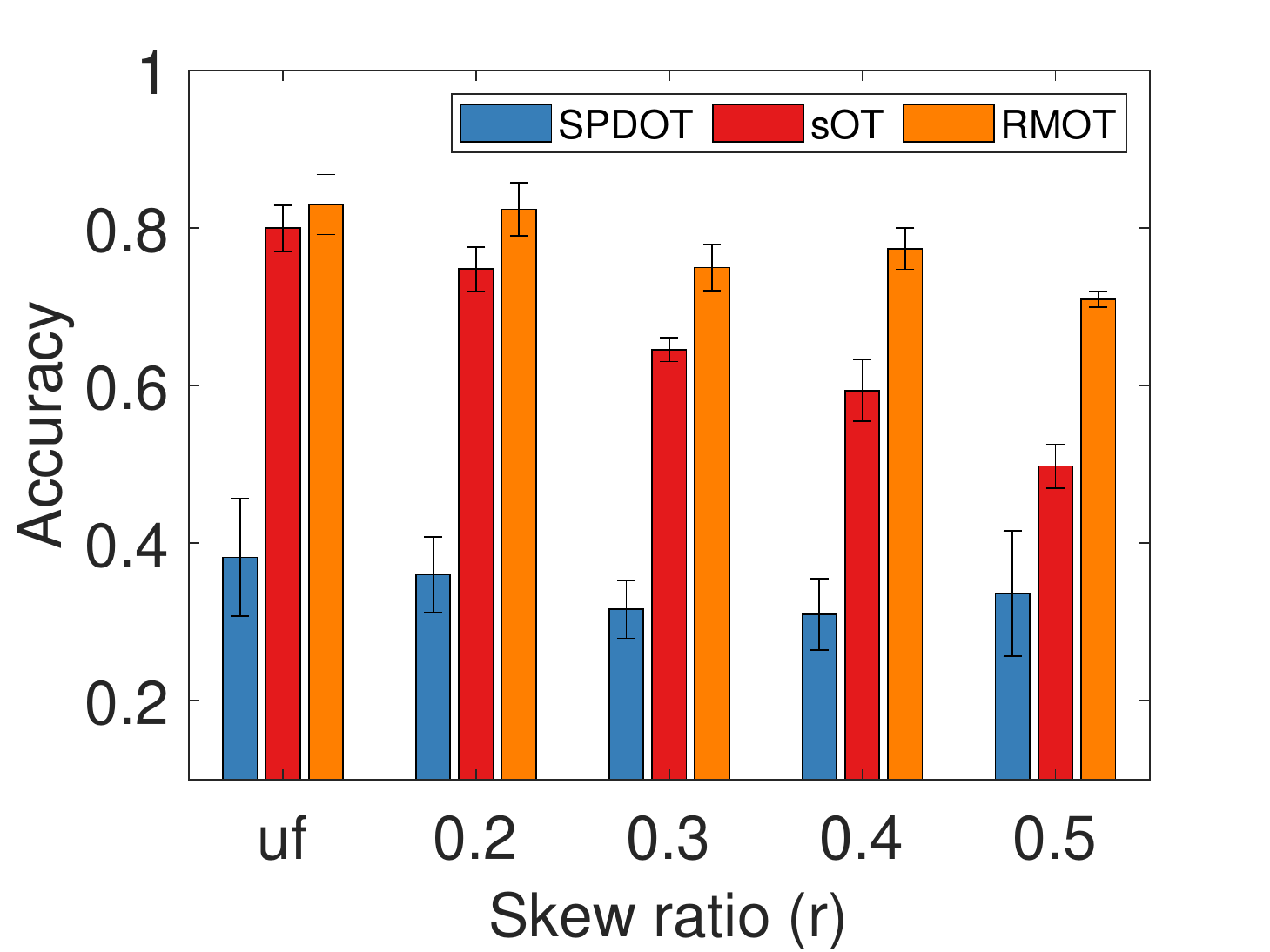}} 
    \subfloat[Fashion MNIST \label{fmnist_fig}]{\includegraphics[width = 0.33\textwidth, height = 0.25\textwidth]{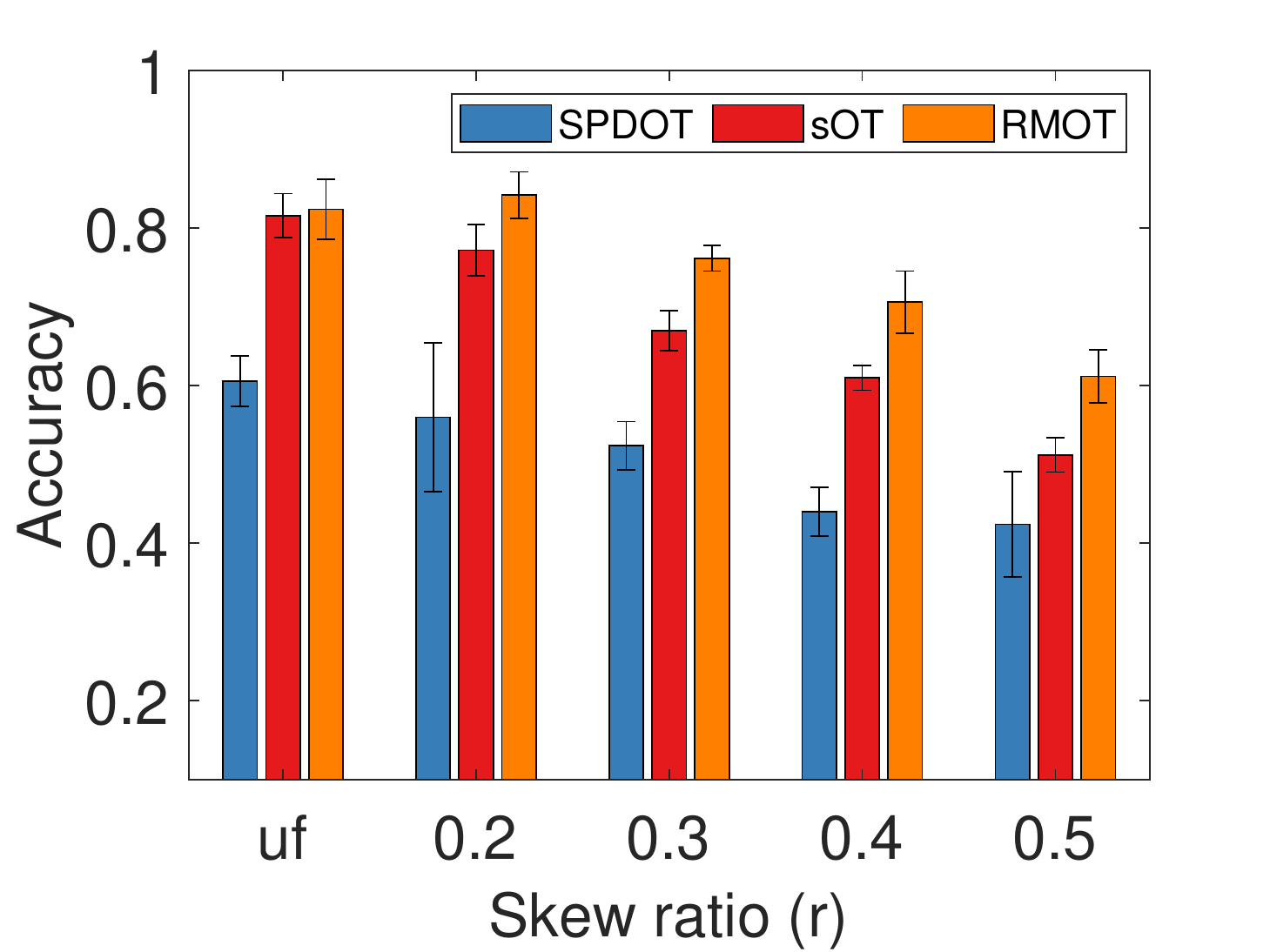}} 
    \subfloat[Letters \label{letters_fig}]{\includegraphics[width = 0.33\textwidth, height = 0.25\textwidth]{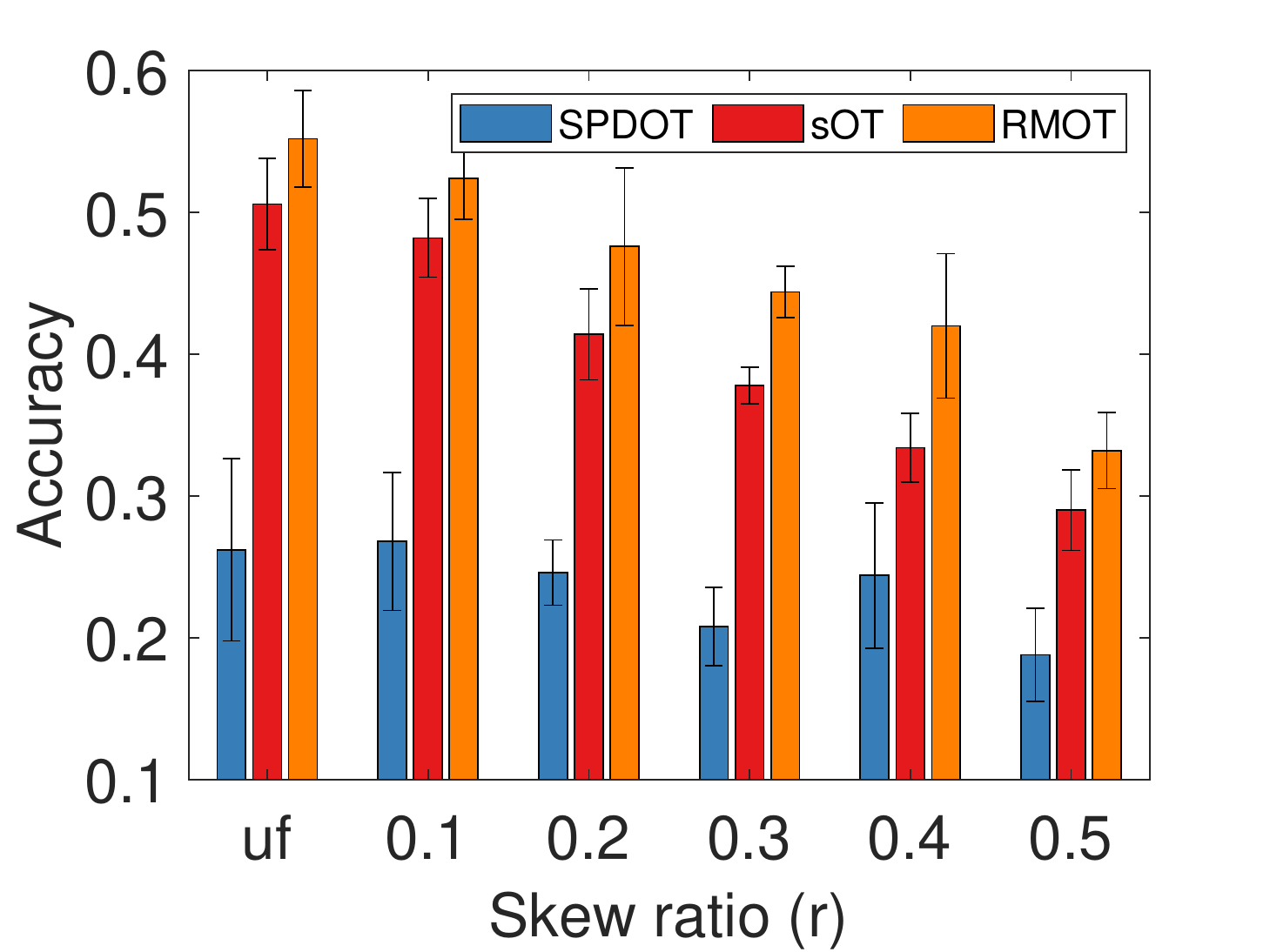}}
    \caption{Domain adaptation and classification results for three datasets: MNIST \eqref{mnist_fig}, Fashion MNIST \eqref{fmnist_fig} and Letters~\eqref{letters_fig}. The skew ratio increases from uniform (uf) to $r=0.5$. For MNIST and Fashion MNIST, uf=$0.1$ and for Letters, uf=$1/26$. We observe that the proposed RMOT performs significantly better than the baselines. 
    }
\end{figure}

\textbf{Datasets.} We experiment on three multiclass image datasets -  handwritten letters \cite{frey1991letter}, MNIST \cite{lecun1998gradient} and Fashion MNIST \cite{xiao2017fashion} - with various skewed distributions for the training set. MNIST and Fashion MNIST have $10$ classes, while Letters has $26$ classes. 
Specifically, we fix the distribution of the test set to be uniform (with the same number of image sets per class). We increase the proportion of the a randomly chosen class in the training set to the ratio $r$, where $r=\{{\rm uf},0.1,0.2,0.3,0.4,0.5\}$ and ${\rm uf}$ is the ratio corresponding to the uniform distribution of all classes. We reduce the dimension of MNIST, fashion MNIST, and Letters by PCA to $d=5$ features. We set $s=d$, $m=250$, and $n=100$ for each dataset. 

\textbf{Results.} Figures~\ref{mnist_fig}-\ref{letters_fig} shows the classification accuracy on the three datasets. 
We observe that the proposed RMOT outperforms sOT and SPDOT, especially in more challenging domain adaptation settings, i.e., higher skew ratios. This implies the usefulness of the non-trivial correlations learned by the SPD matrix valued couplings of {\algname}. 


\subsection{Tensor Gromov-Wasserstein distance averaging for shape interpolation}
We consider an application of the proposed block SPD Gromov-Wasserstein OT formulation (Section~\ref{sec:proposedGW}) for interpolating tensor-valued shapes. We are given two distance-marginal pairs $(\bD^0, \bP^0), (\bD^1, \bP^1)$ where $\bD^0, \bD^1 \in \sR^{n \times n}$ are distance matrices computed from the shapes and $\bP^0, \bP^1$ are given tensor fields. The aim is to interpolate between the distance matrices with weights $\bomega = (t, 1-t), t \in [0,1]$. The interpolated distance matrix $\bD^t$ is computed by solving \eqref{gw_barycenter_problem} via Riemannian optimization and Proposition \ref{average_distance_update}, with the barycenter tensor fields ${\bP}^t$ given. Finally, the shape is recovered by performing multi-dimensional scaling to the distance matrix.

\begin{figure}
\centering
\begin{tabular}{@{}c@{ \hspace*{1.5pt} }c@{ \hspace*{2pt} }c@{ \hspace*{2pt} }c@{ \hspace*{2pt} }c@{ \hspace*{2pt} }c@{ \hspace*{2pt} }c@{ \hspace*{2pt} }}
\raisebox{.025\textwidth}{\footnotesize (a)}&
\includegraphics[width=.11\textwidth]{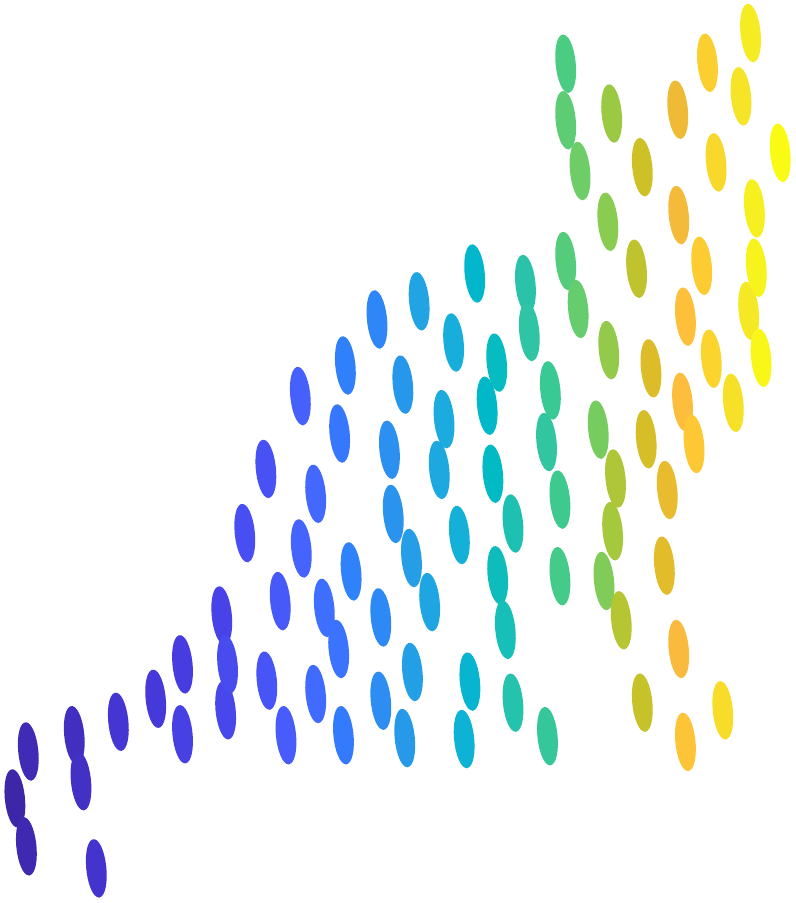}&
\includegraphics[width=.11\textwidth]{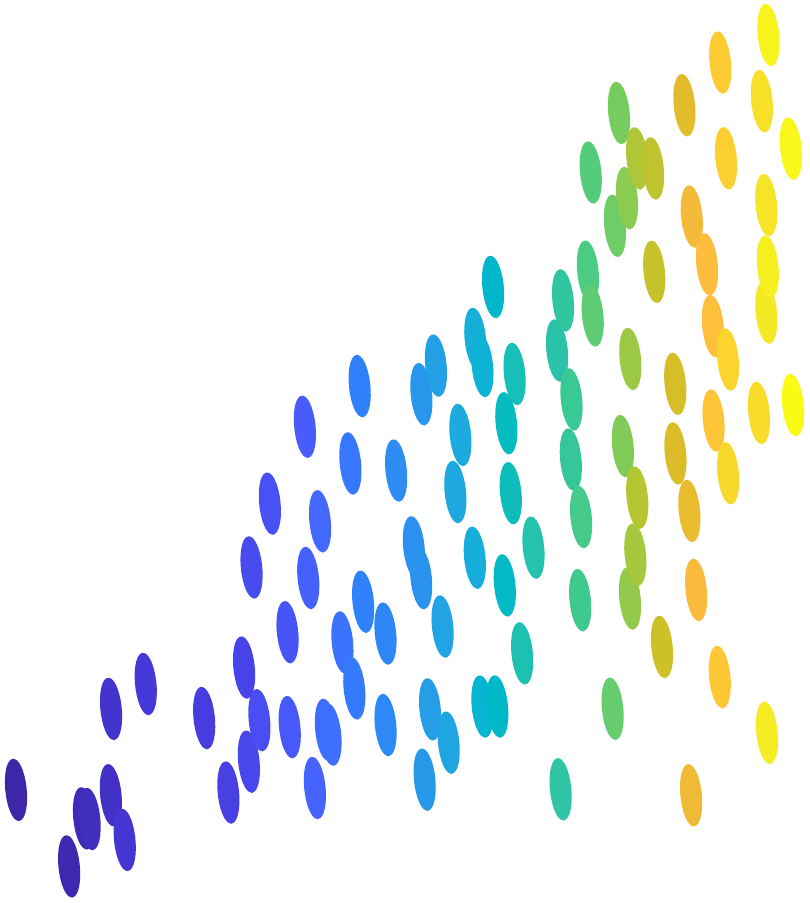}&
\includegraphics[width=.11\textwidth]{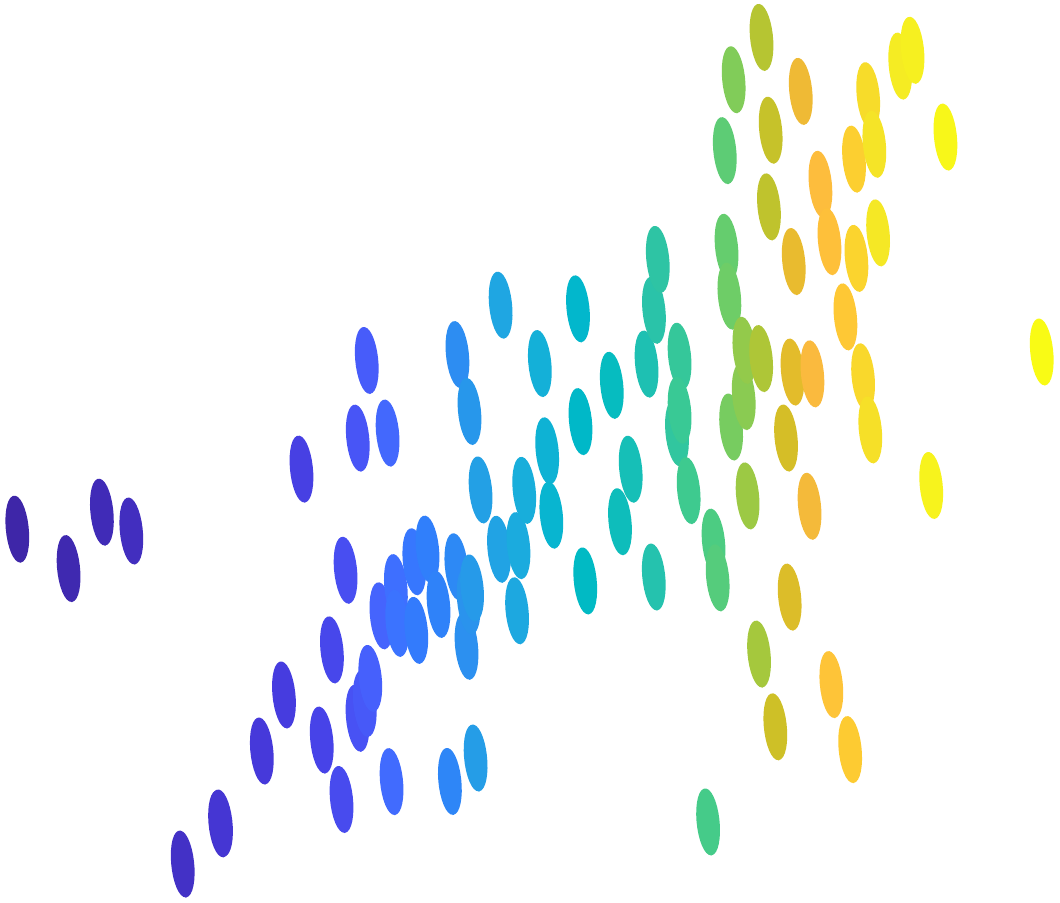}&
\includegraphics[width=.11\textwidth]{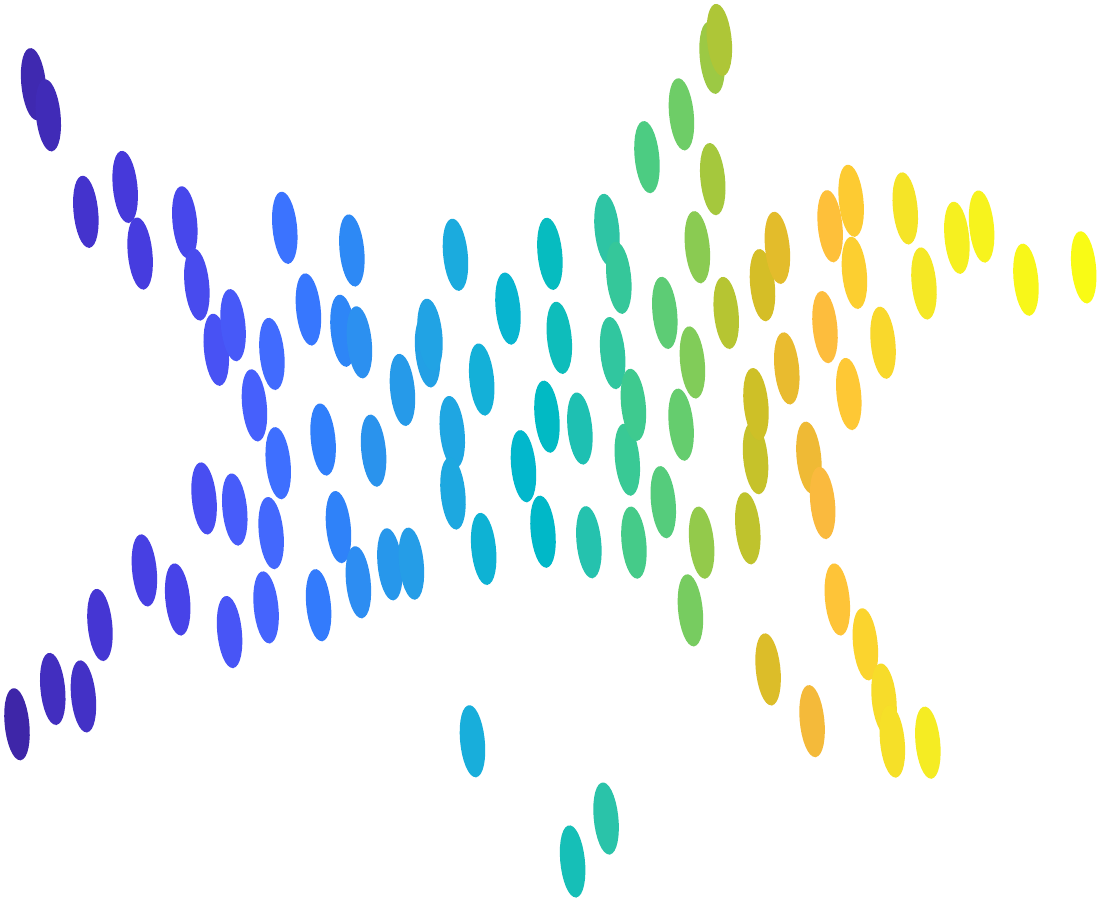}&
\includegraphics[width=.11\textwidth]{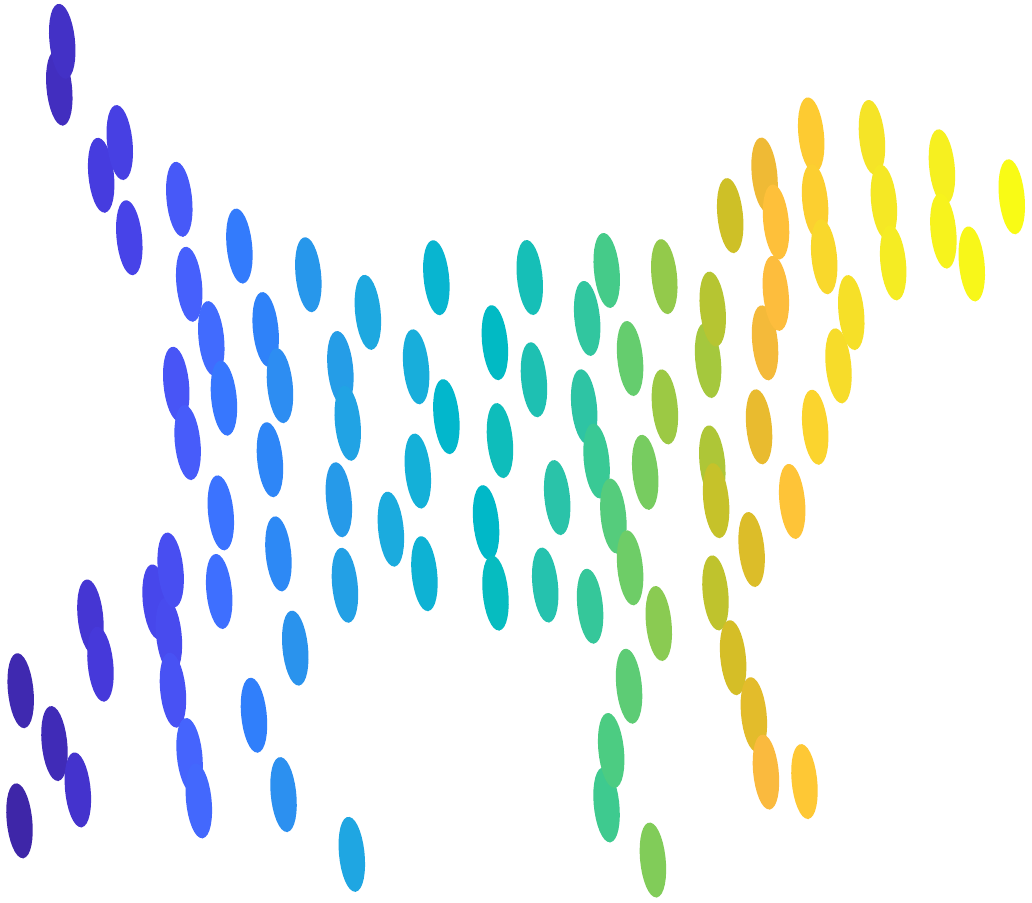}&
\includegraphics[width=.11\textwidth]{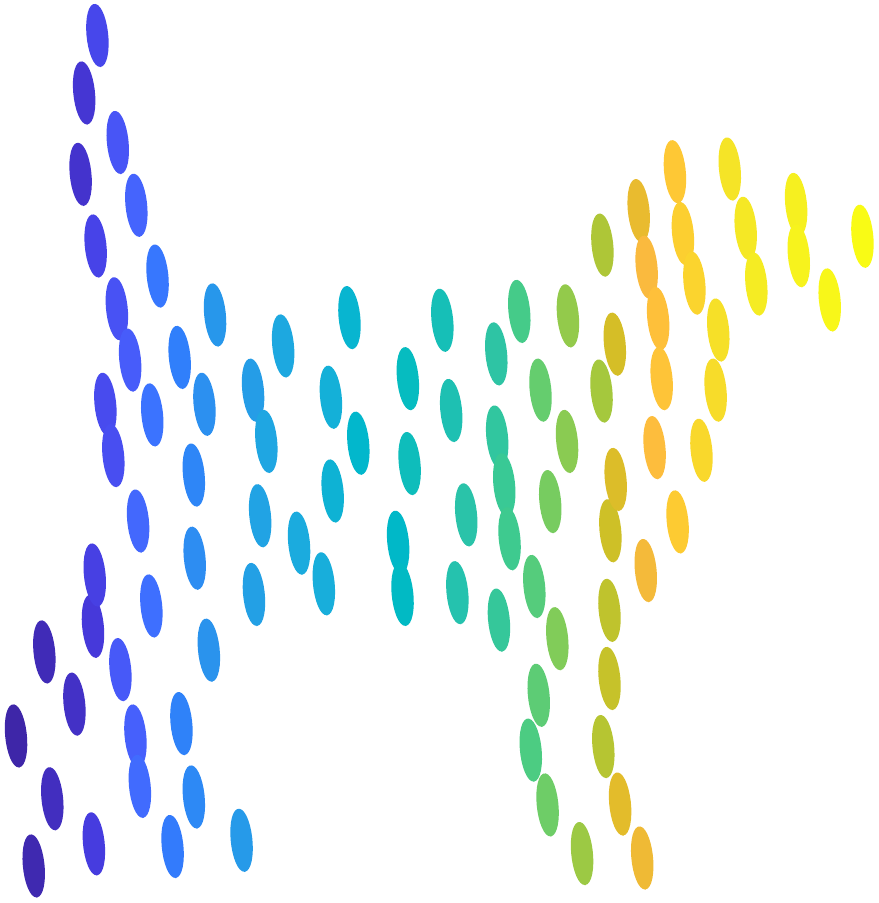}\\[5pt]

\raisebox{.025\textwidth}{\footnotesize (b)}&
\includegraphics[width=.11\textwidth]{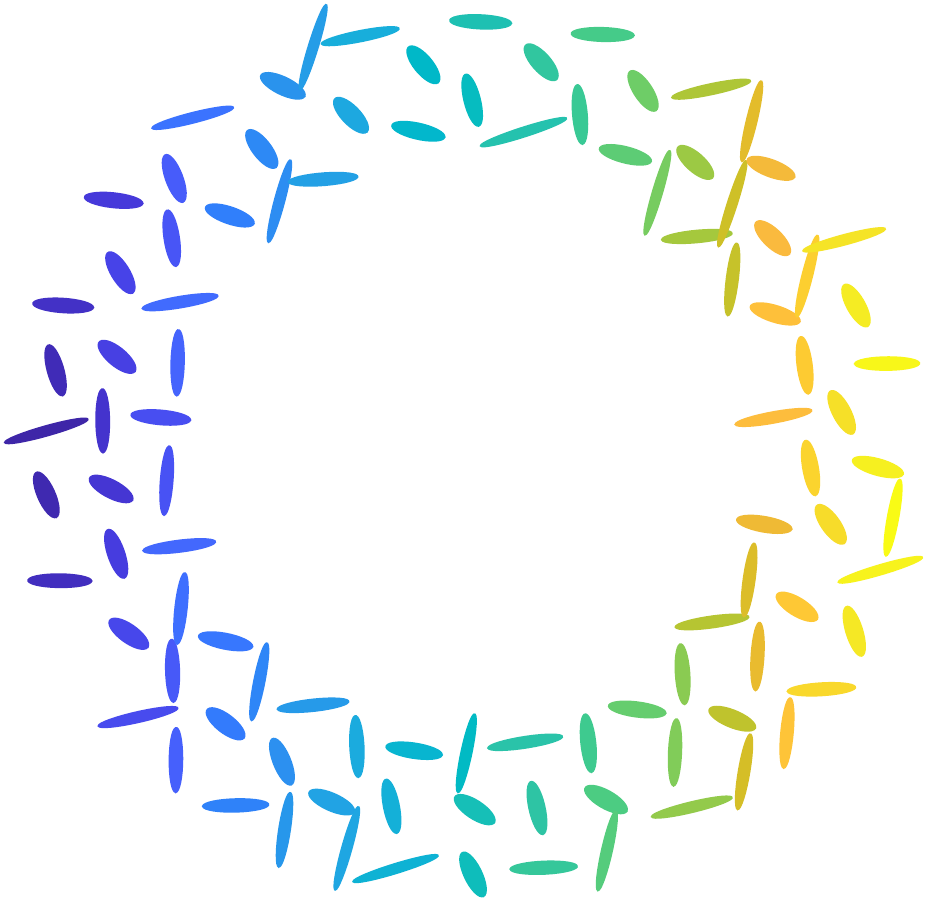}&
\includegraphics[width=.11\textwidth]{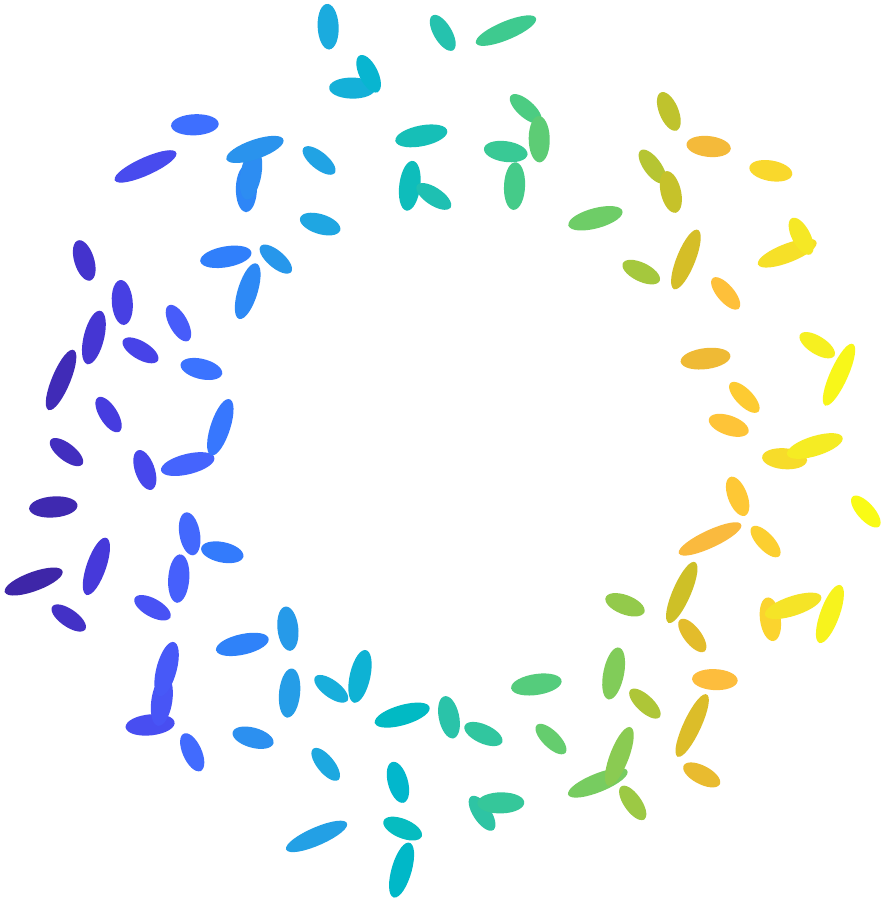}&
\includegraphics[width=.11\textwidth]{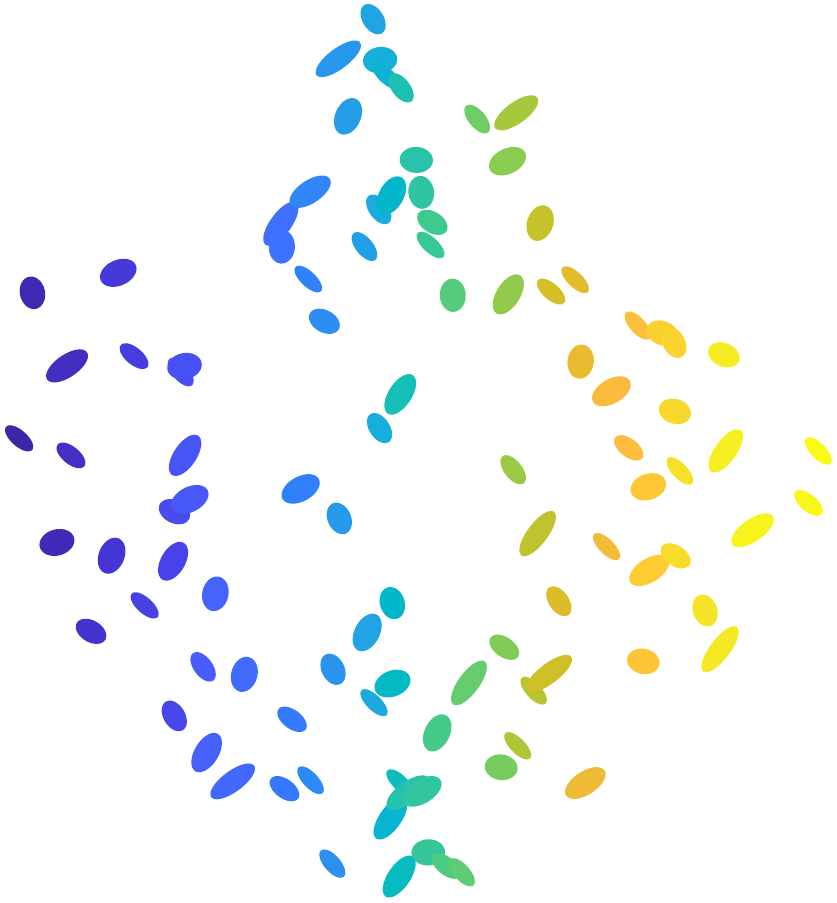}&
\includegraphics[width=.11\textwidth]{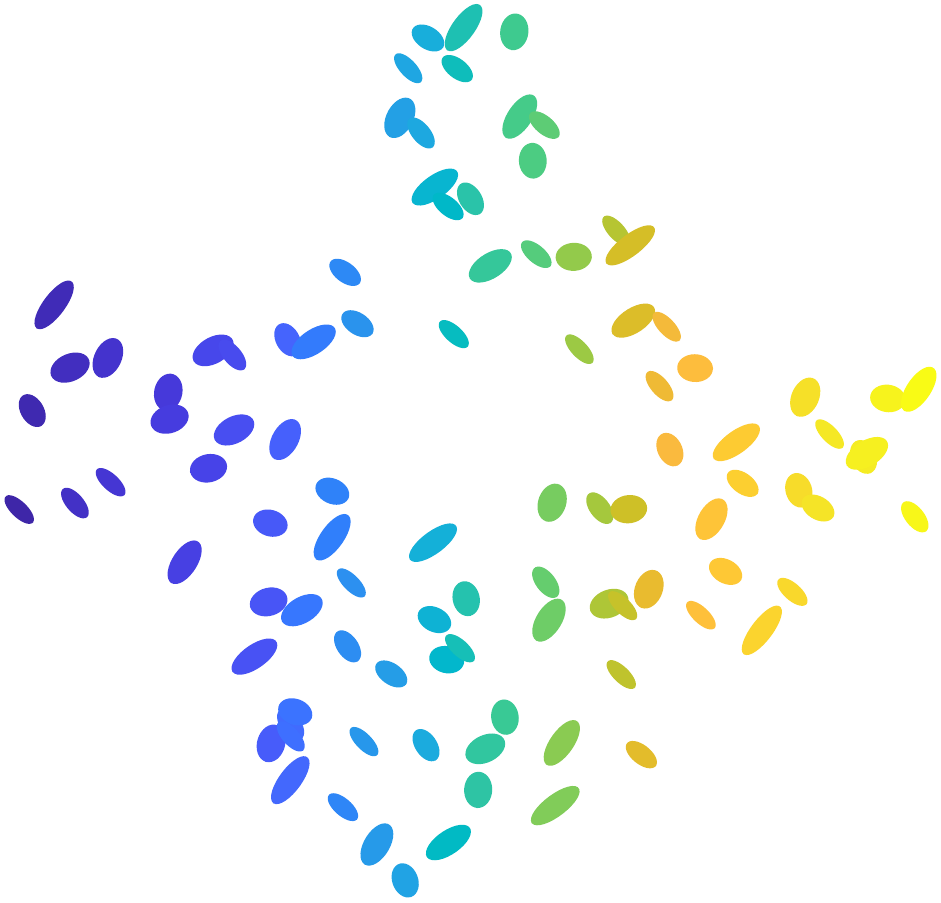}&
\includegraphics[width=.11\textwidth]{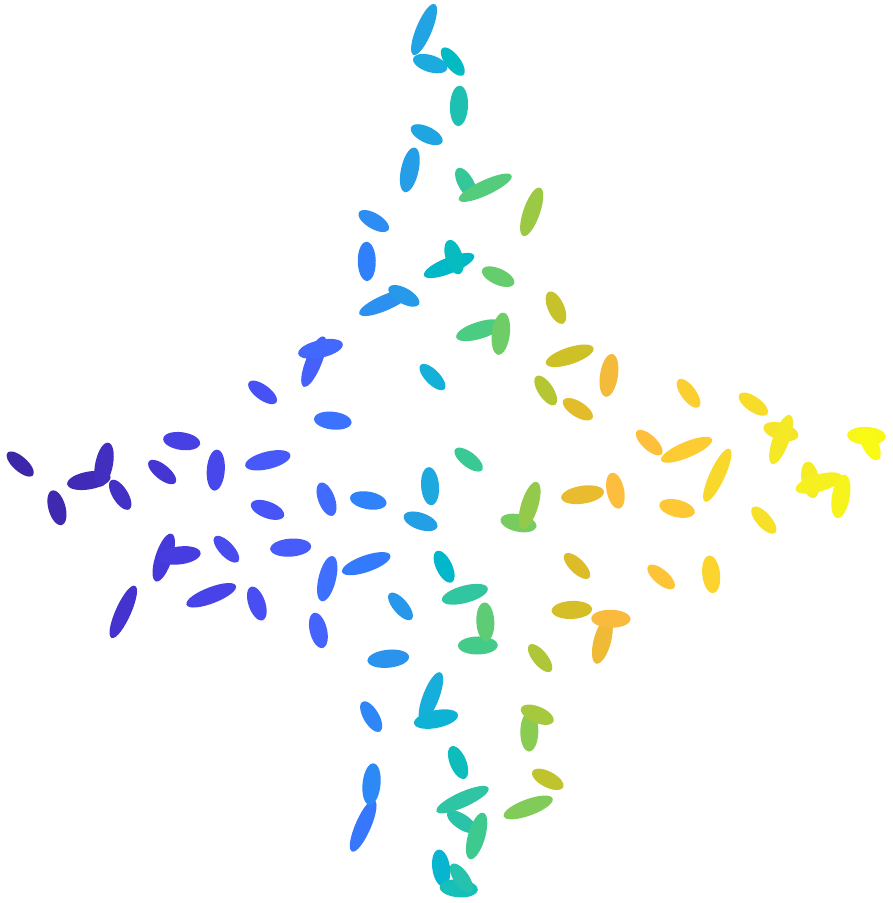}&
\includegraphics[width=.11\textwidth]{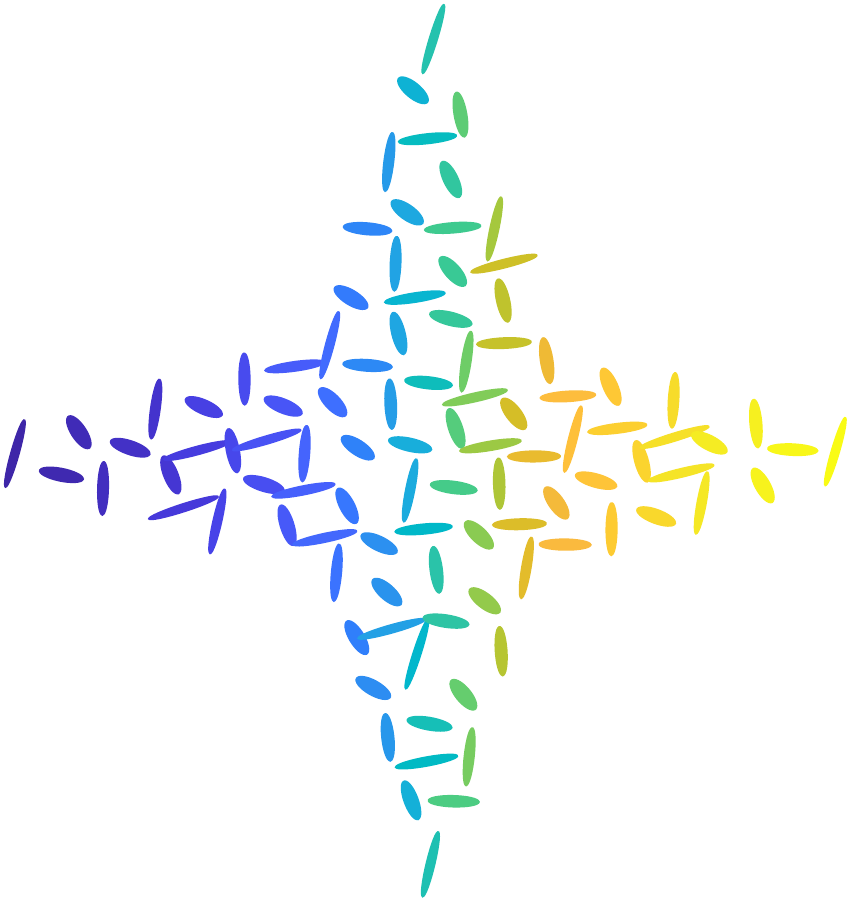}\\[5pt]

\raisebox{.025\textwidth}{\footnotesize (c)}&
\includegraphics[width=.11\textwidth]{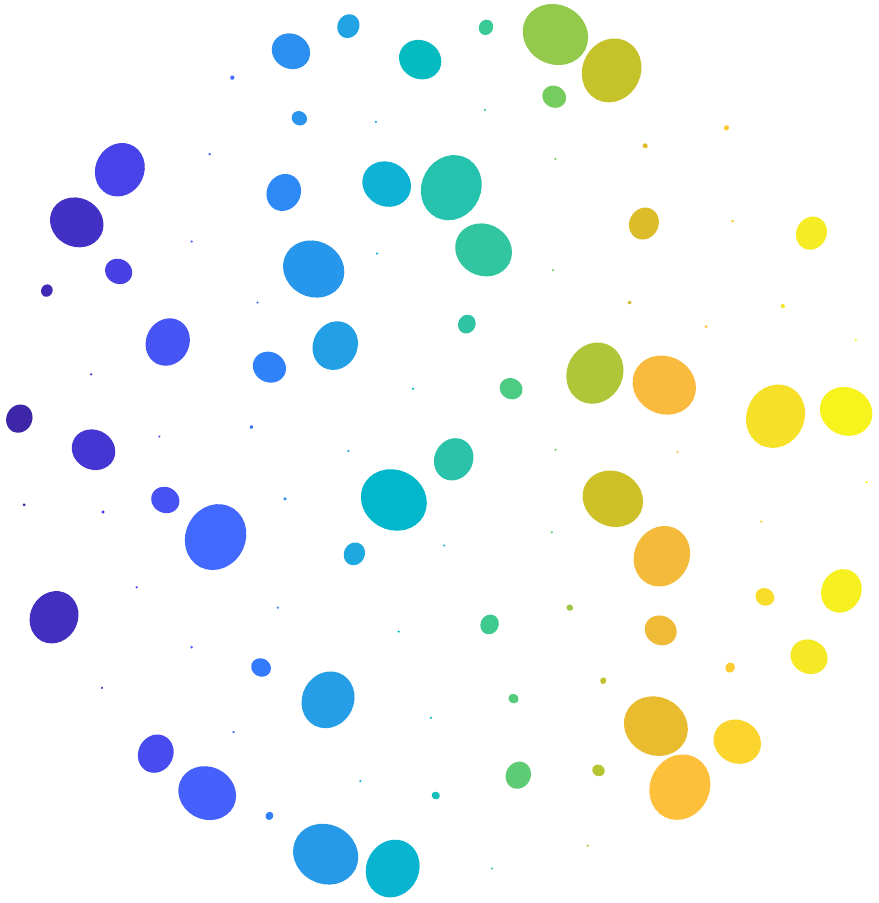}&
\includegraphics[width=.11\textwidth]{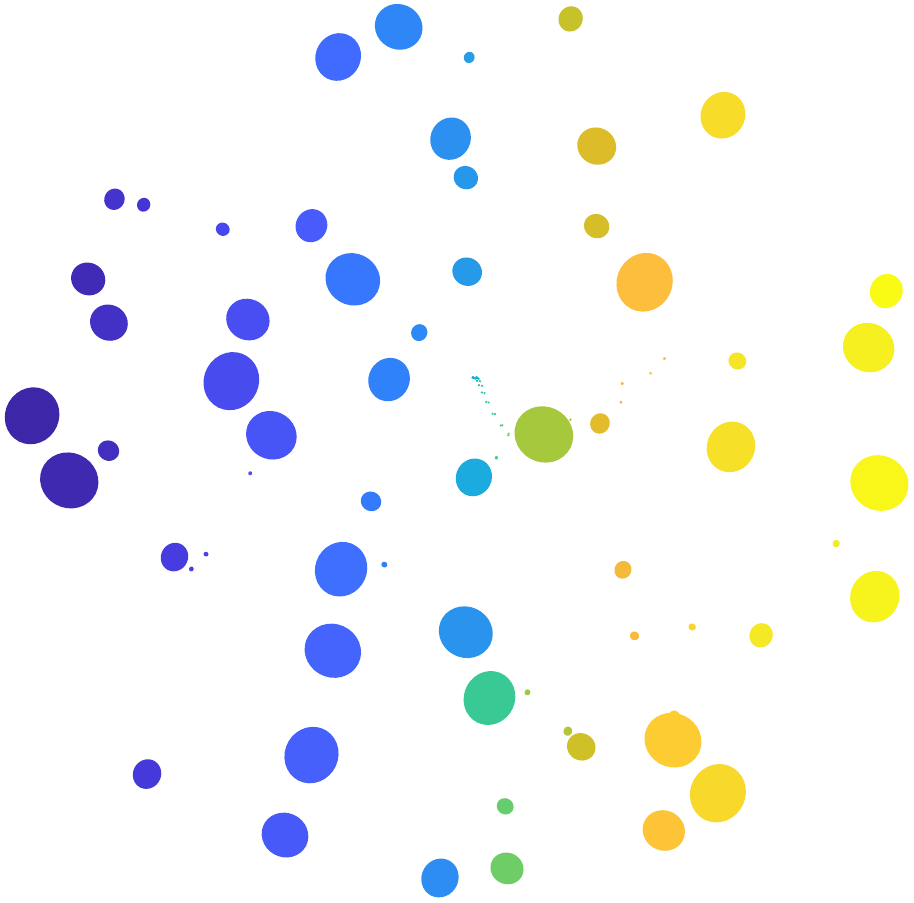}&
\includegraphics[width=.11\textwidth]{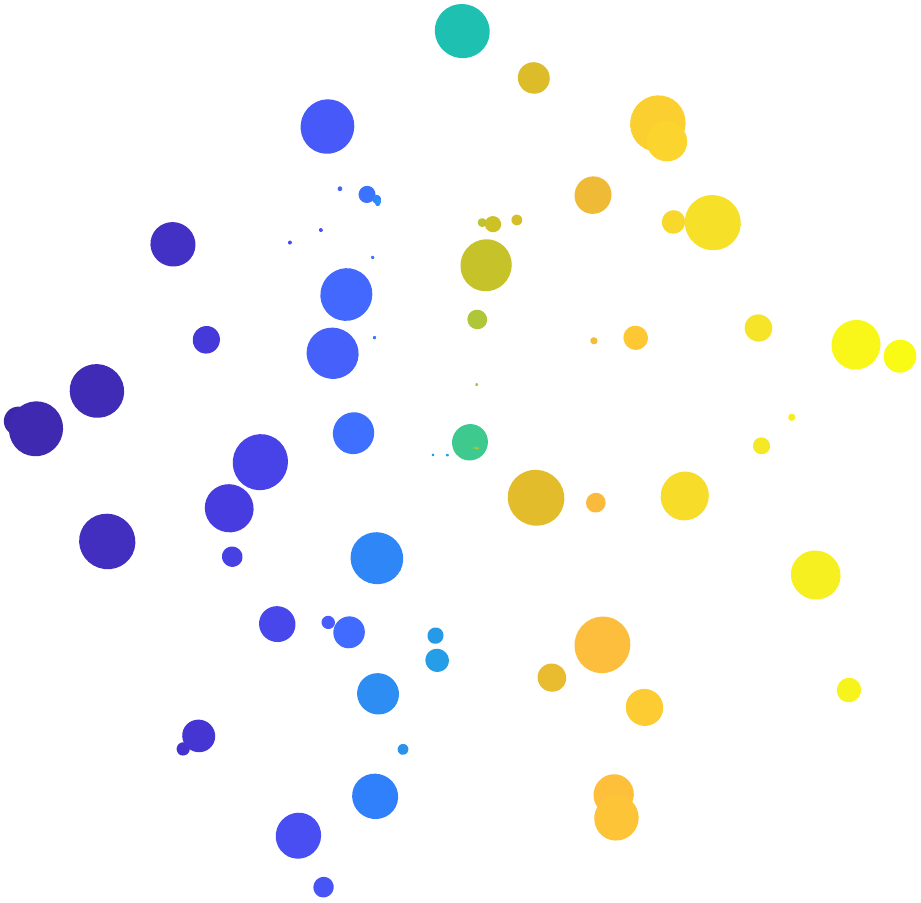}&
\includegraphics[width=.11\textwidth]{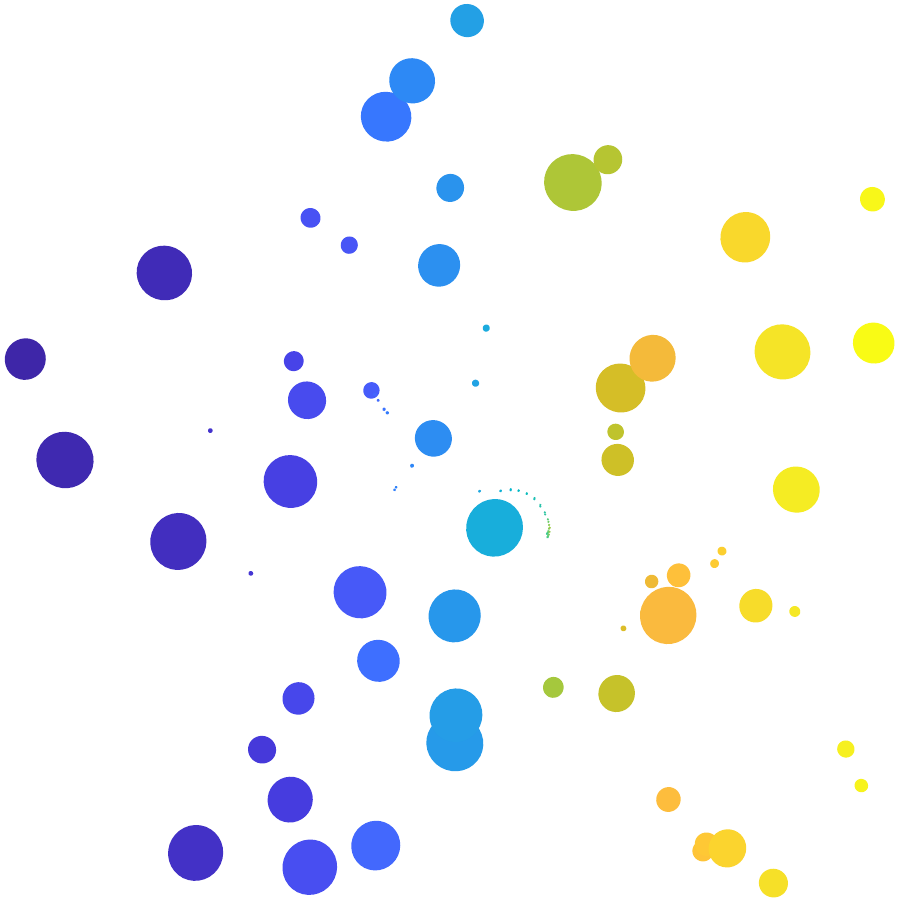}&
\includegraphics[width=.11\textwidth]{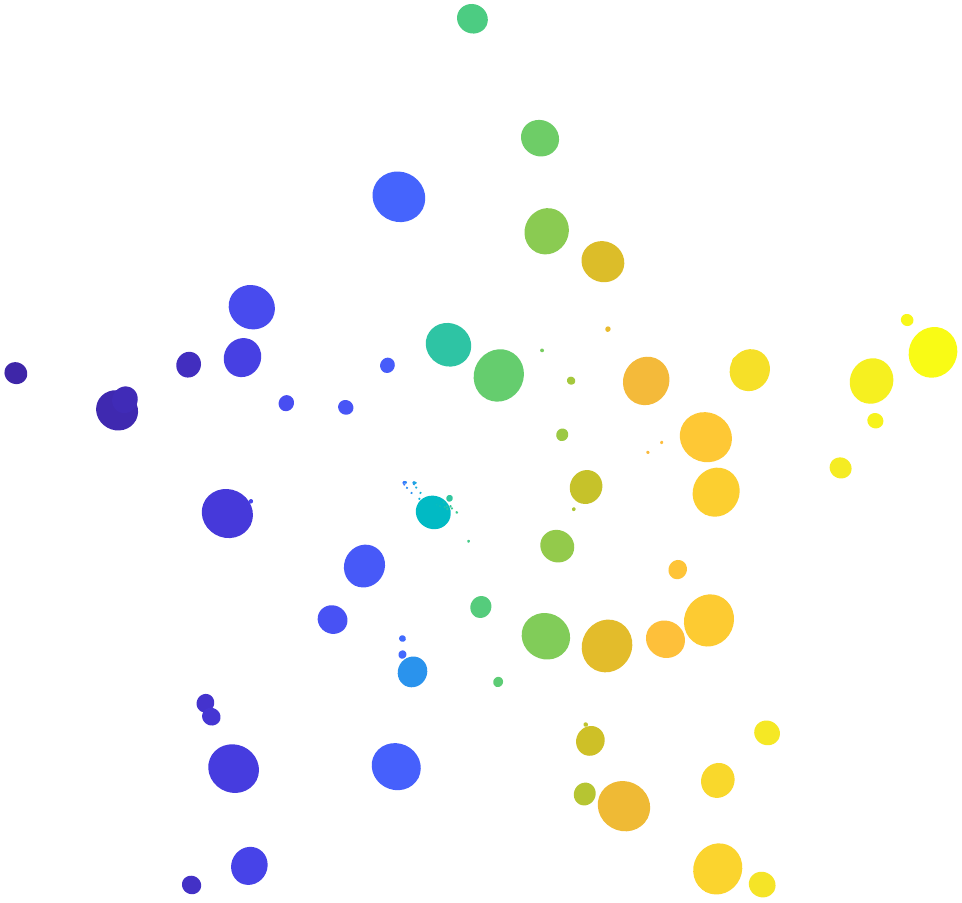}&
\includegraphics[width=.11\textwidth]{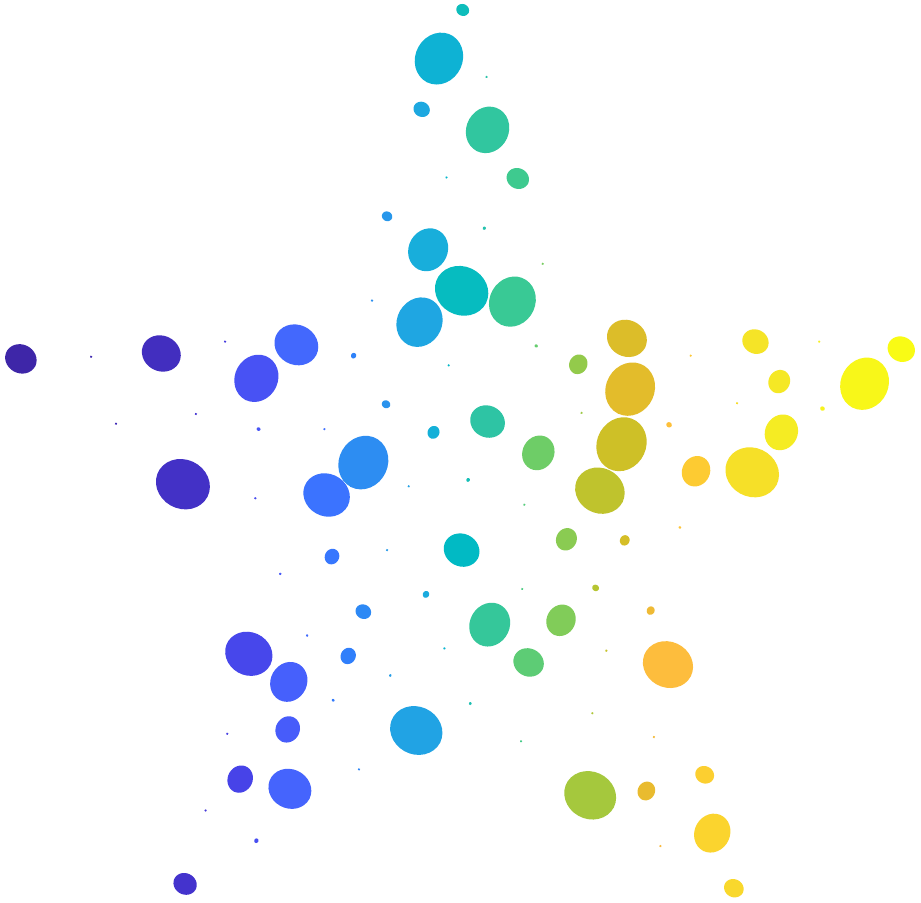}\\[3pt]
&{\footnotesize $t=0$} & {\footnotesize $t=0.2$} & {\footnotesize $t=0.4$} & {\footnotesize $t=0.6$} & {\footnotesize $t=0.8$} & {\footnotesize $t=1$} \\
\end{tabular}
\caption{Tensor-valued shape interpolation obtained using the proposed Gromov-Wasserstein RMOT formulation~(Section~\ref{sec:proposedGW}). We note that each shape consists of several SPD-matrix/tensor valued fields (displayed using ellipses). \revision{In (a), the tensors follow a uniform distribution. In (b), the tensors are generated with multiple orientation and in (c), tensors vary smoothly in both size and orientation.} The proposed approach takes the anisotropy and orientation of tensor fields into account while interpolating shapes. }%
\label{GW_distance_averaging}
\end{figure}

Figure \ref{GW_distance_averaging} presents the interpolated shapes with $n = 100$ sample points for the input shapes. The matrices $\bD^0, \bD^1$ are given by the Euclidean distance and we consider $L_2$ loss for $\L$. The input tensor fields $\bP^0, \bP^1$ are generated as uniformly random in (a), cross-oriented in (b) and smoothly varying in (c). \revision{For simplicity, we consider the barycenter tensor fields given by the linear interpolation of the inputs, i.e., $\bP^t = (1-t) \bP^0 + t \bP^1$. In \cite{peyre2016gromov}, we highlight that the marginals are scalar-valued and fixed to be uniform. Here, on the other hand, the marginals are tensor-valued and the resulting distance matrix interpolation would be affected by the relative mass of the tensors, as shown by Proposition \ref{average_distance_update}.}
The results show the proposed Riemannian optimization approach (Section~\ref{manifold_geometry_sect}) converges to reasonable stationary solutions for non-convex OT problems.

\begin{figure*}
    \centering
    \subfloat[Linear interpolation]{\includegraphics[width = 0.19\textwidth, height = 0.1\textwidth]{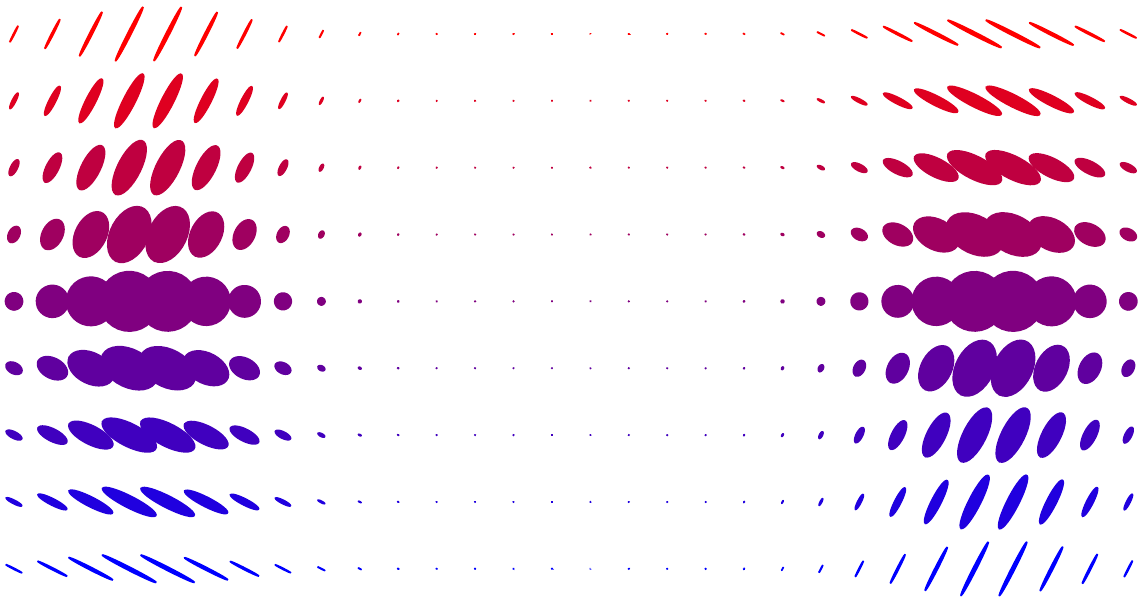}} 
    \hspace*{5pt}
    \subfloat[QOT ($\rho$=50)]{\includegraphics[width = 0.19\textwidth, height = 0.1\textwidth]{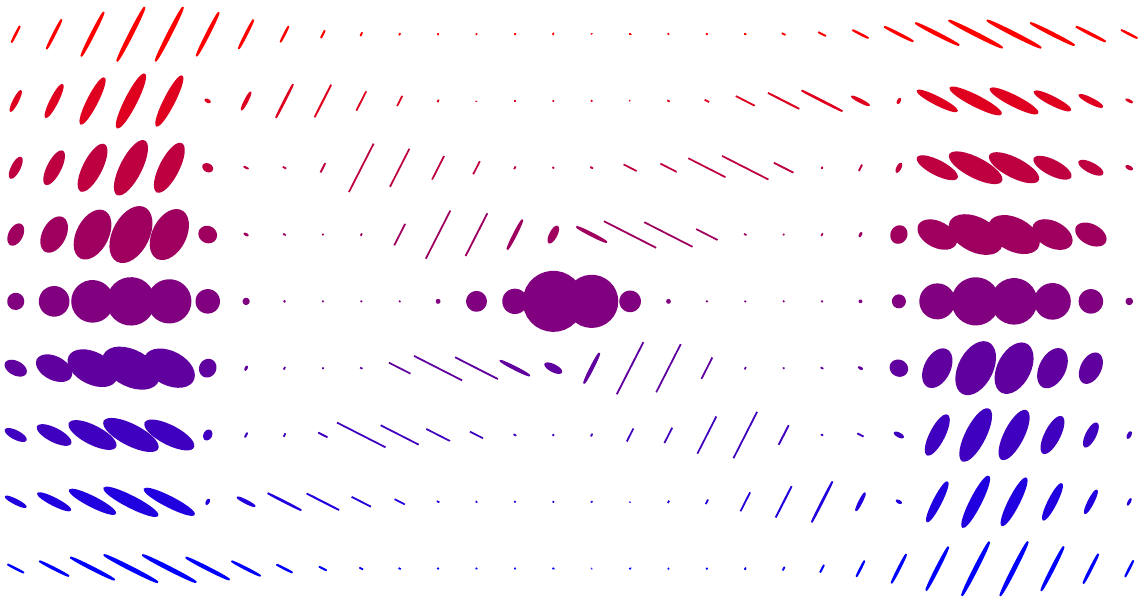}}
    \hspace*{5pt}
    \subfloat[QOT($\rho$=100)]{\includegraphics[width = 0.19\textwidth, height = 0.1\textwidth]{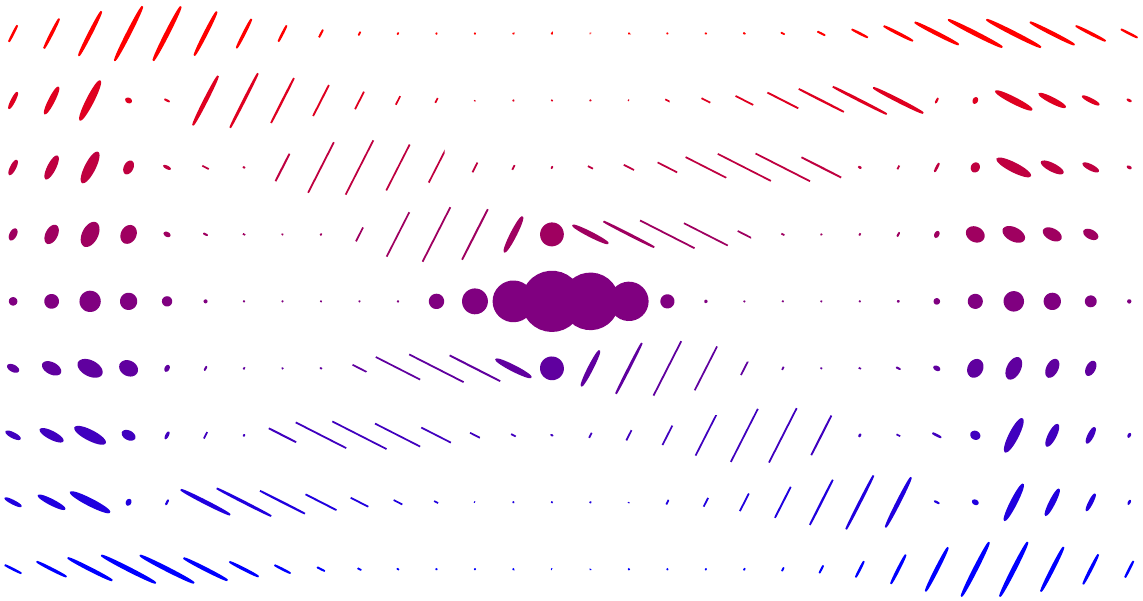}}
    \hspace*{5pt}
    \subfloat[QOT($\rho$=500)]{\includegraphics[width = 0.19\textwidth, height = 0.1\textwidth]{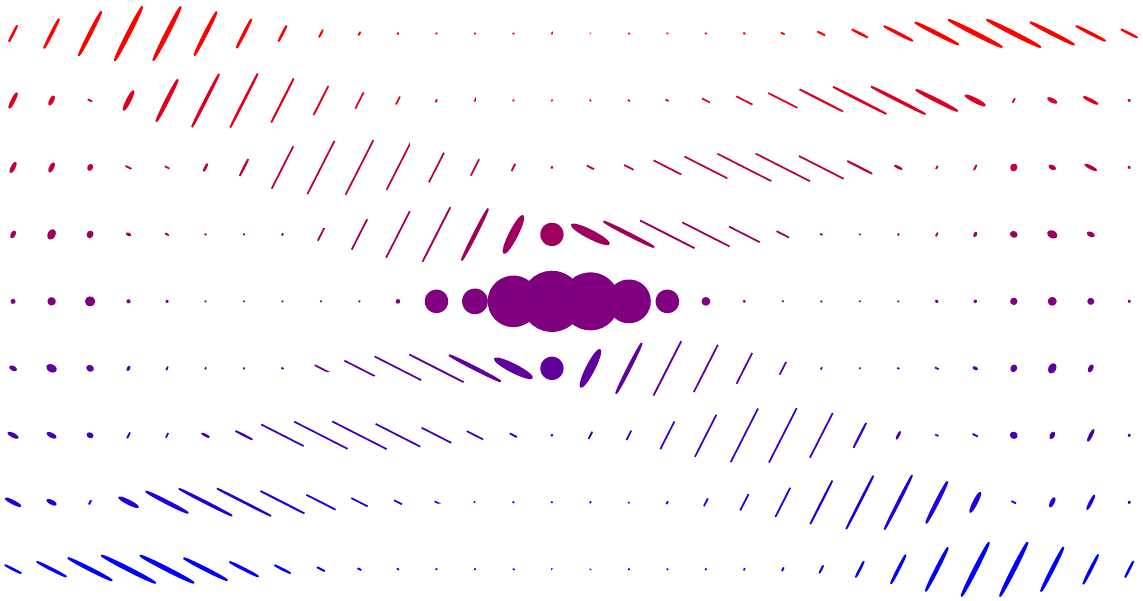}}
    \hspace*{5pt}
    \subfloat[{\algname}]{\includegraphics[width = 0.19\textwidth, height = 0.1\textwidth]{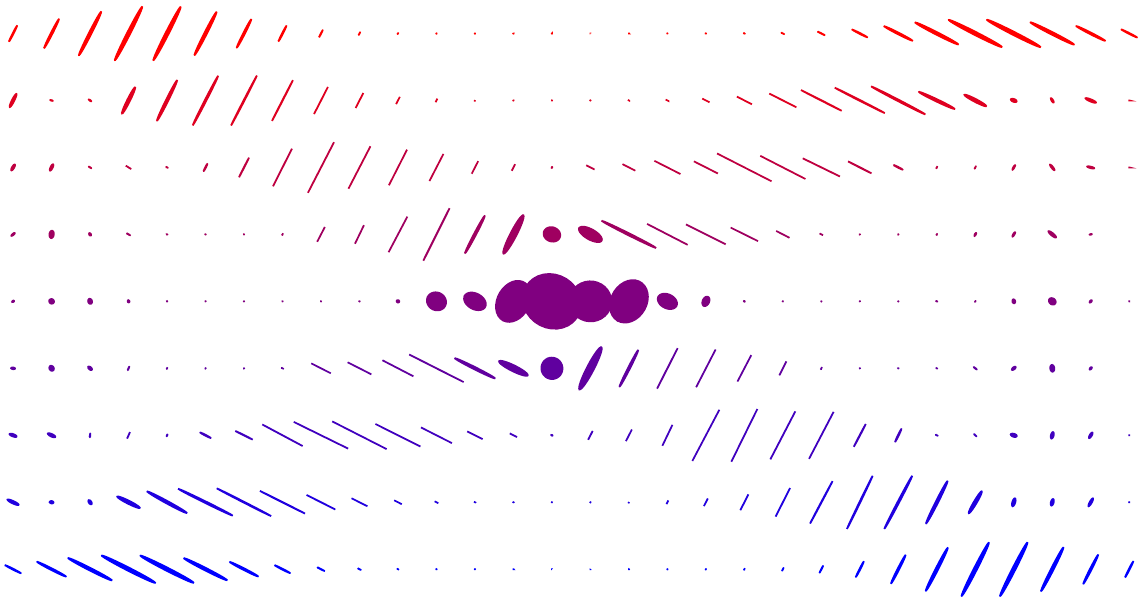}} \\[5pt]
    
    \subfloat[Input ($t = 0$) \label{input_1}]{\includegraphics[width = 0.12\textwidth, height = 0.12\textwidth]{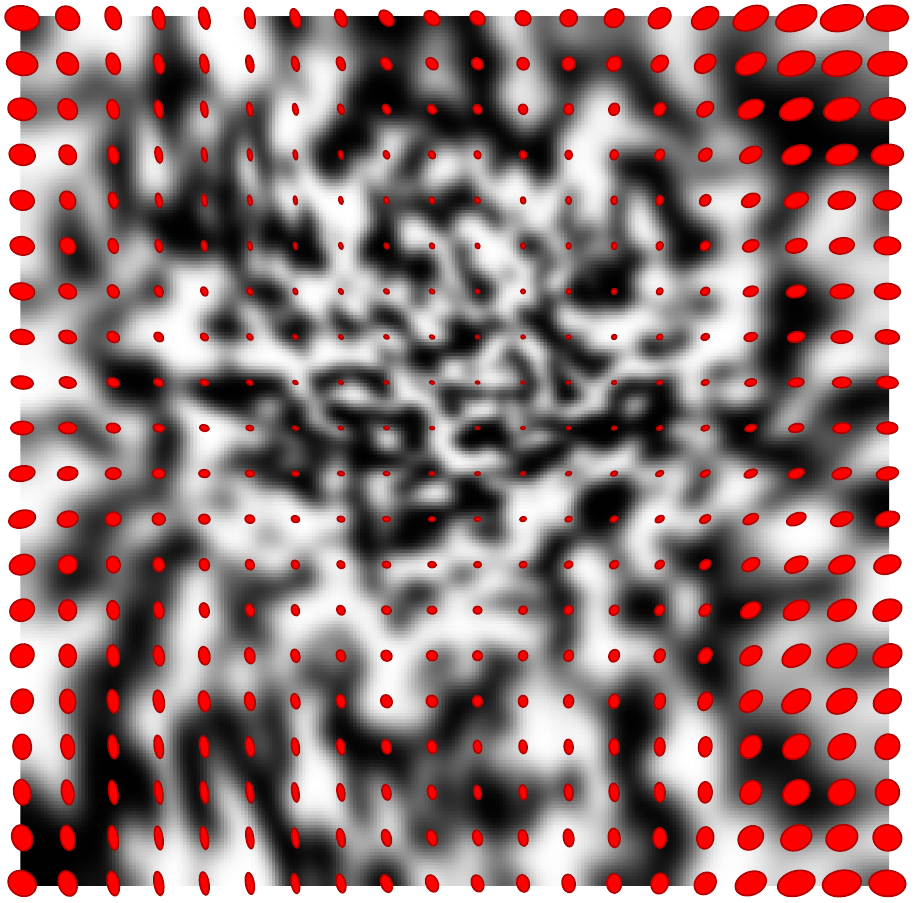}} 
    \hspace*{3pt}
    \subfloat[Input ($t=1$) \label{input_2}]{\includegraphics[width = 0.12\textwidth, height = 0.12\textwidth]{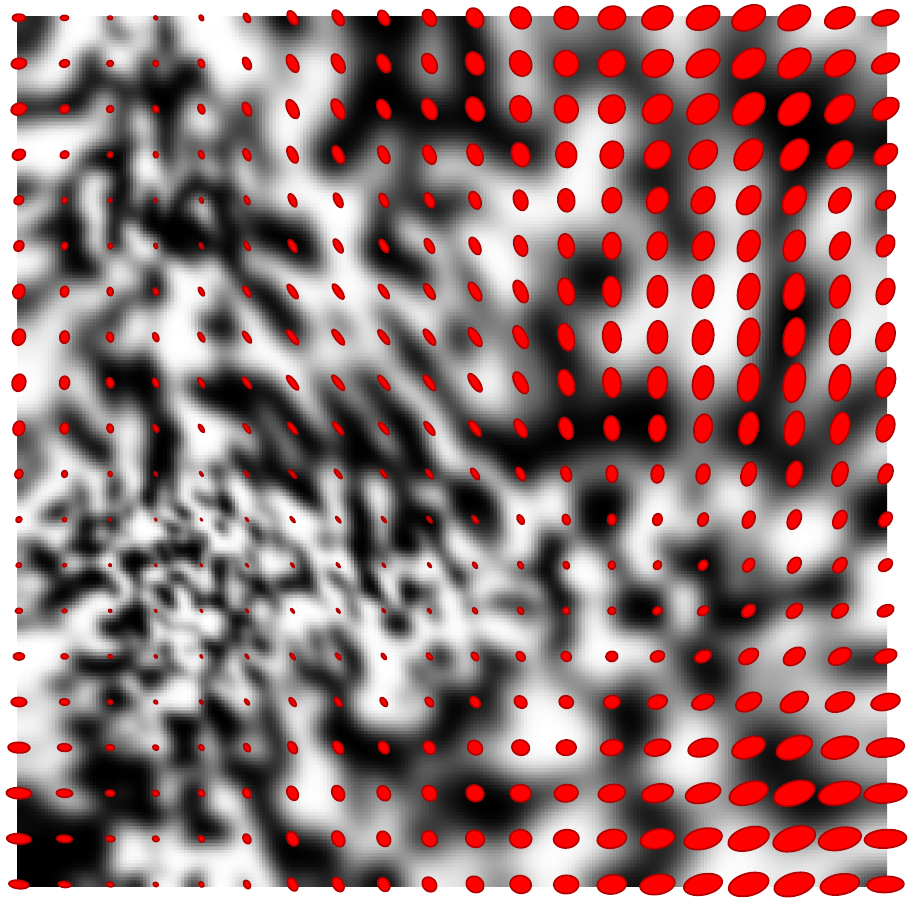}} 
    \hspace*{3pt}
    \subfloat[QOT ($t=0.25$)]{\includegraphics[width = 0.12\textwidth, height = 0.12\textwidth]{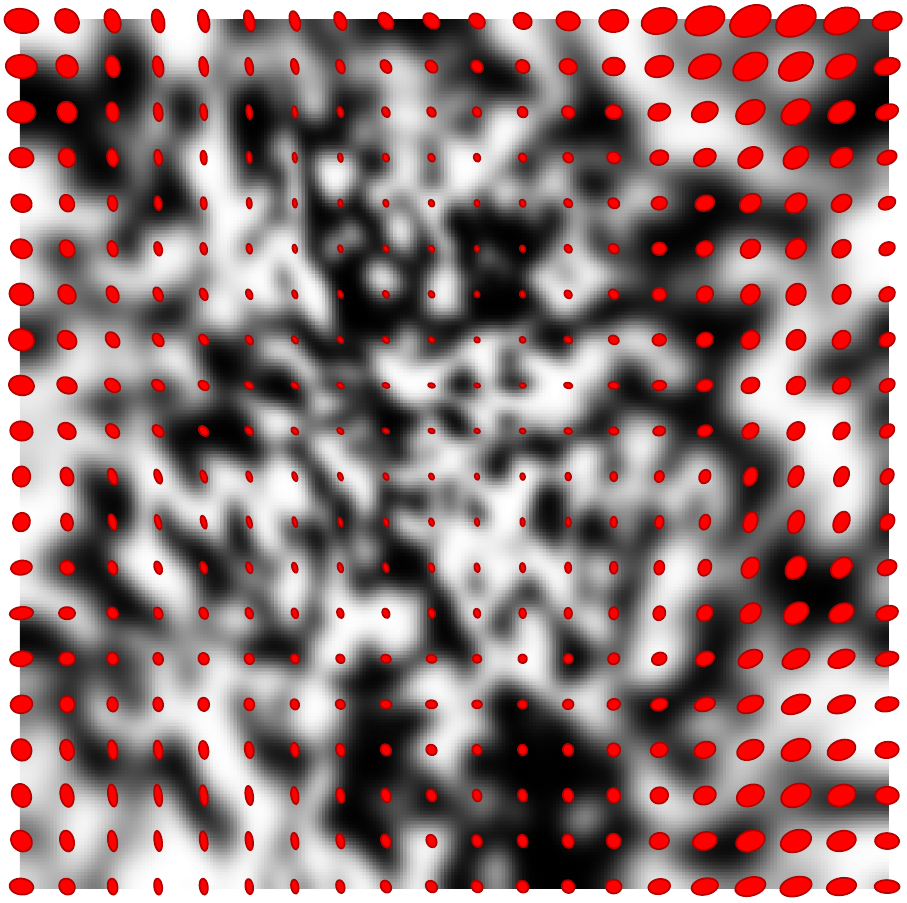}}
    \hspace*{3pt}
    \subfloat[QOT ($t=0.5$)]{\includegraphics[width = 0.12\textwidth, height = 0.12\textwidth]{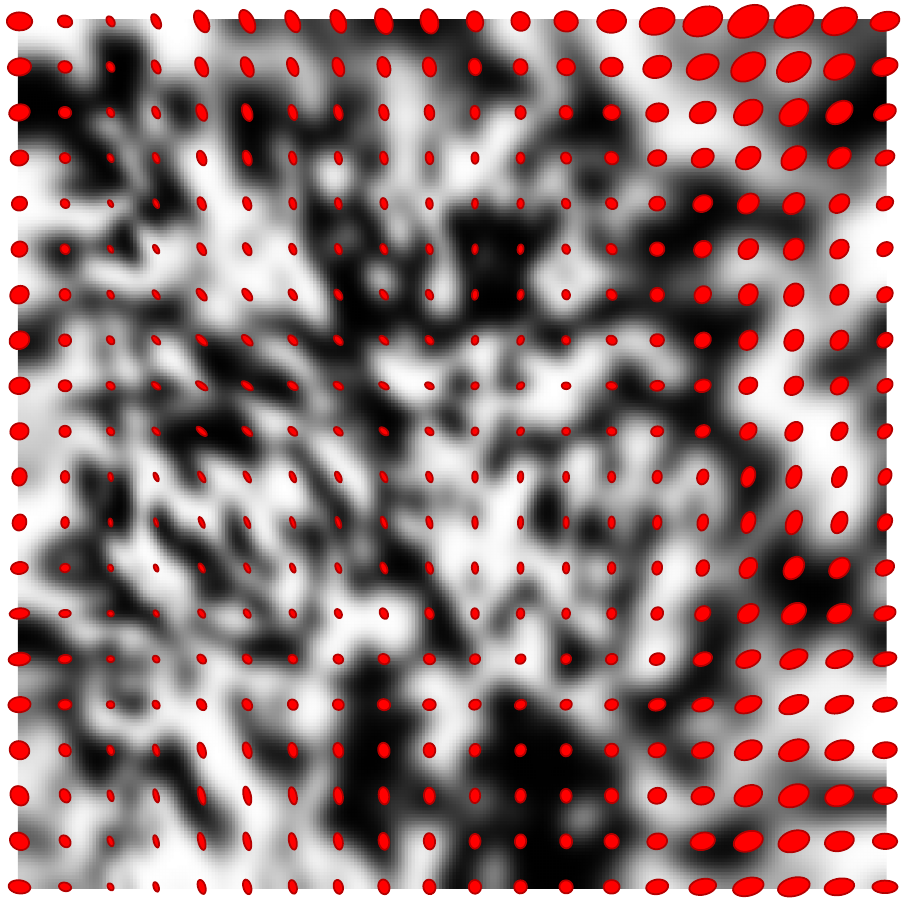}}
    \hspace*{3pt}
    \subfloat[QOT ($t=0.75$)]{\includegraphics[width = 0.12\textwidth, height = 0.12\textwidth]{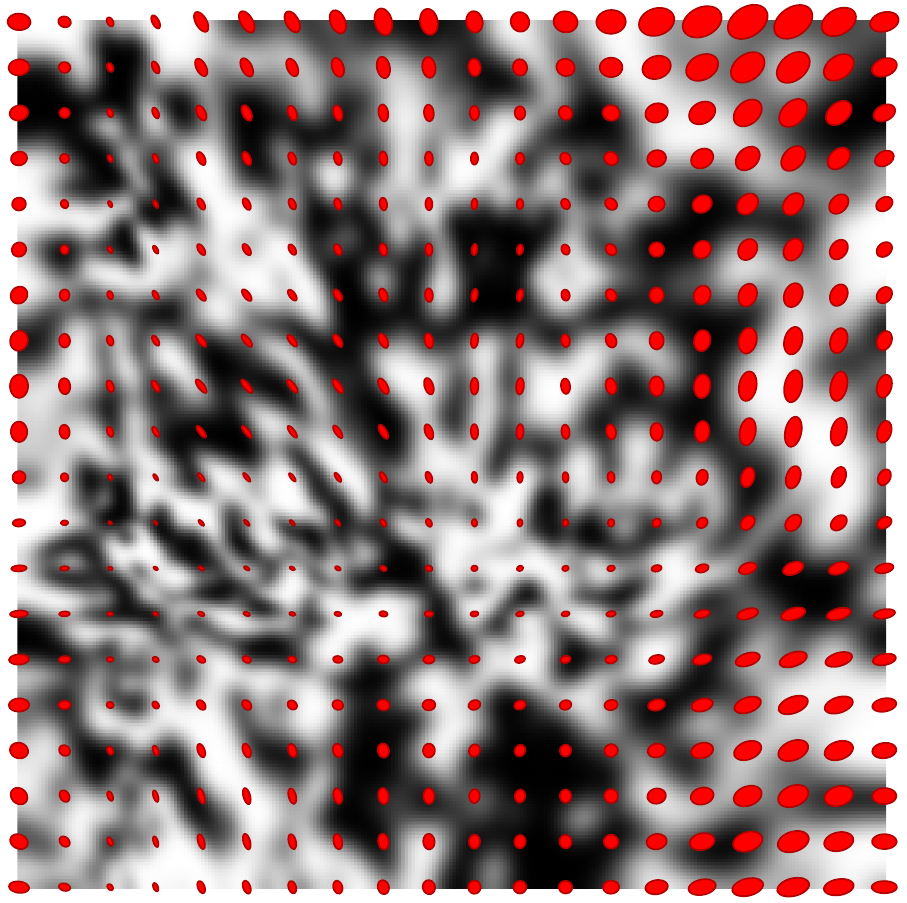}} 
    \hspace*{3pt}
    \subfloat[{\algname} ($t=0.25$)]{\includegraphics[width = 0.12\textwidth, height = 0.12\textwidth]{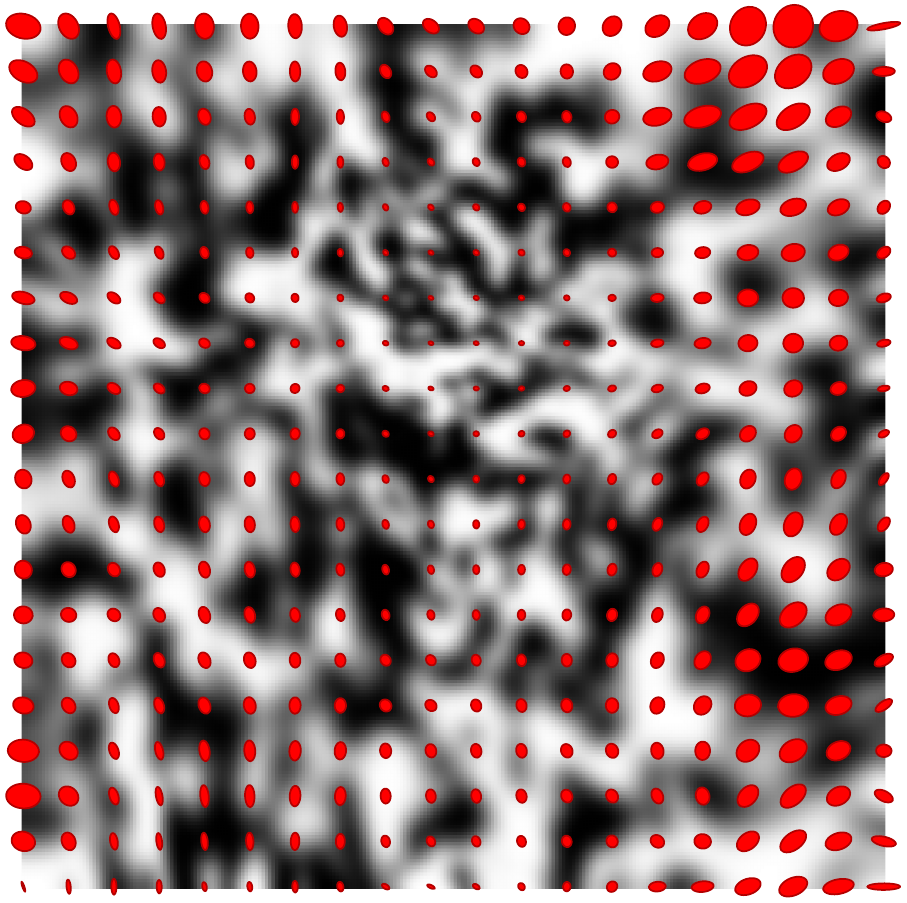}}
    \hspace*{3pt}
    \subfloat[{\algname} ($t=0.5$)]{\includegraphics[width = 0.12\textwidth, height = 0.12\textwidth]{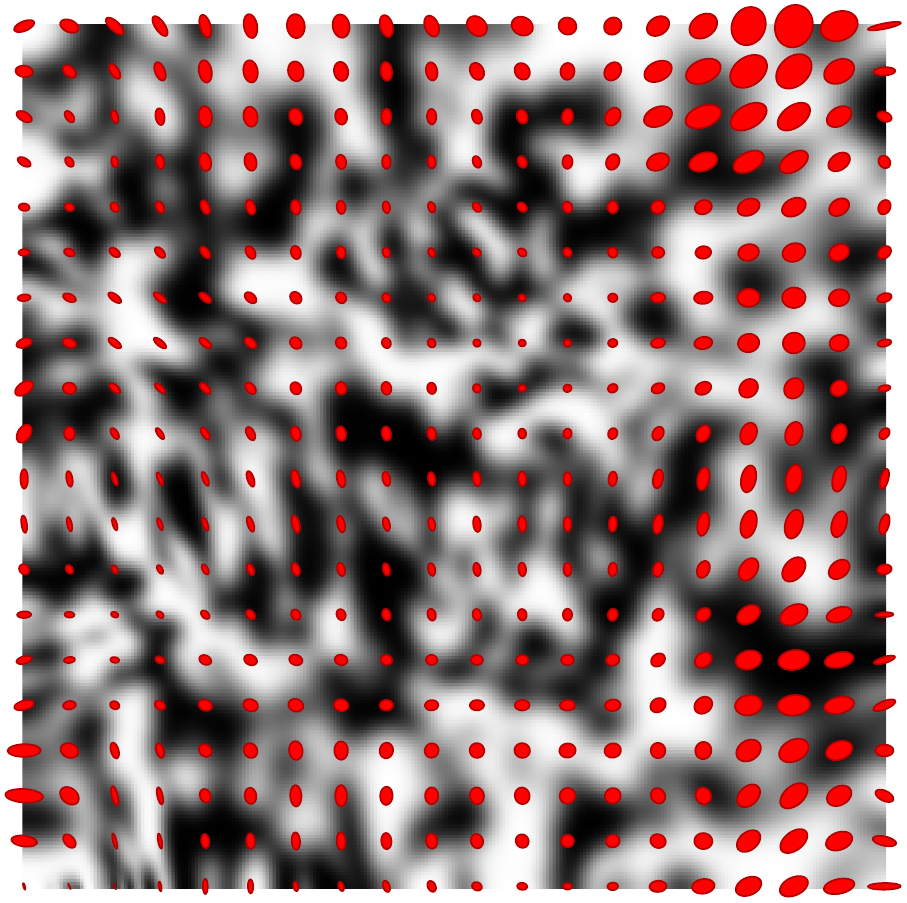}}
    \hspace*{3pt}
    \subfloat[{\algname} ($t=0.75$)]{\includegraphics[width = 0.12\textwidth, height = 0.12\textwidth]{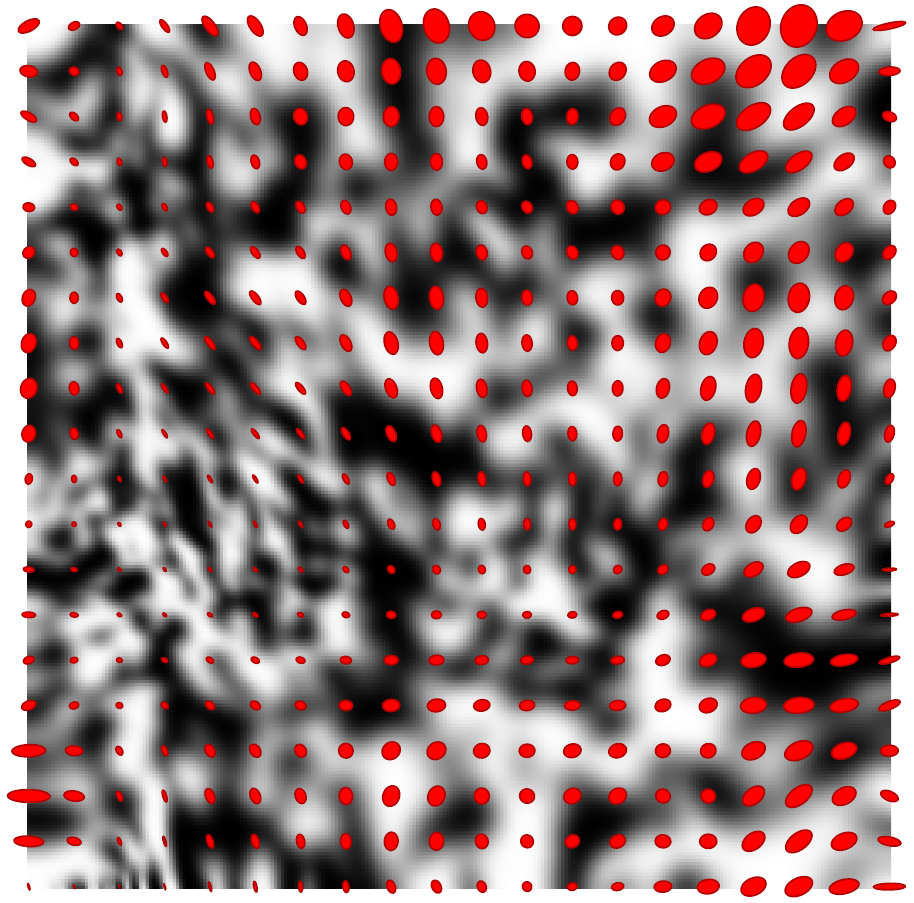}}
    \caption{Tensor field mass interpolation on 1-$d$ (top) and 2-$d$ (bottom) grids. On the top, each row corresponds to an interpolation where we show $7$ evenly-spaced interpolated tensor fields. On the bottom, the inputs are given in (f) and (g). We set $\rho = 100$ for QOT and show $3$ evenly-spaced interpolated tensor fields.} 
    \label{interpolation_figure}
\end{figure*}

\subsection{Tensor field optimal transport mass interpolation}
We consider performing optimal transport and displacement interpolation between two tensor fields supported on regular 1-$d$ (or 2-$d$) grids~\cite{peyre2016quantum}. 
We consider a common domain $\gD = [0,1]$ (or $[0,1]^2$) with the cost defined as $\bC_{i,j} = \| \bx_i - \by_j \|^2 \bI$ for $\bx_i, \by_j \in \gD$. The marginals $\bP, \bQ$ are given tensor fields. We first compute the balanced coupling $\bGamma$ by solving an entropy regularized OT problem~(\ref{eqn:regMWOT}): 
\begin{equation*}
    \min_{\bGamma \in \M_{m \times n}^d(\bP, \bQ)} \sum_{i,j} \Big( \trace(\bC_{i,j}\bGamma_{i,j}) - \epsilon \, H(\bGamma_{i,j}) \Big),
\end{equation*} 
where the quantum entropy is defined as $H(\bGamma_{i,j}) \coloneqq - \trace(\bGamma_{i,j} \log(\bGamma_{i,j}) - \bGamma_{i,j})$. 
Then, the coupling is used to interpolate between the two tensor fields by generalizing the displacement interpolation \cite{mccann1997convexity} to SPD-valued marginals. Please refer to~\cite[Section~2.2]{peyre2016quantum} for more details. It should be noted that due to the balanced nature of our formulation, we do not need to adjust the couplings after matching as required in~\cite{peyre2016quantum}.

We compare interpolation results of the proposed (balanced) {\algname} with both linear interpolation $(1-t)\bP + t \bQ$ for $t \in [0,1]$ and the \textit{unbalanced} quantum OT (QOT) of \cite{peyre2016quantum}. The QOT solves the following problem with quantum KL regularization, i.e.,
\begin{equation*}
    \min\limits_{\bGamma} \sum_{i,j} \Big( \trace(\bC_{i,j} \bGamma_{i,j}) - \epsilon \, H(\bGamma_{i,j}) + \rho \, {\rm KL}(\bGamma \sone \vert \bP) + \rho \,{\rm KL}(\bGamma^\top \sone \vert \bQ) \Big), 
\end{equation*}
where ${\rm KL}(\bP \vert \bQ) \coloneqq \sum_i \trace \big(\bP_i \log(\bP_i) - \bP_i \log(\bQ_i) -\bP_i +\bQ_i \big)$  and $\bGamma \sone \coloneqq [\sum_j (\bGamma_{i,j})]_{m \times 1}$  and  $\bGamma^\top \sone \coloneqq [\sum_i (\bGamma_{i,j})]_{n \times 1}$. For comparability, we set the same $\epsilon$ for both QOT and {\algname}.

Figure \ref{interpolation_figure} compares the mass interpolation for both 1-$d$ (top) and 2-$d$ (bottom) grids. 
For the 2-$d$ tensor fields, we further render the tensor fields via a background texture where we perform anisotropic smoothing determined by the tensor direction. \revision{To be specific, we follow the procedures in \cite{peyre2016quantum} by applying the tensor to the gradient vector of the textures on the grid such that the texture is stretched in the main eigenvector directions of the tensor.}
In both the settings, we observe that the tensor fields generated from {\algname} respect the marginal constraints more closely.

\subsection{Tensor field Wasserstein barycenter}
We also analyze the Wasserstein barycenters learned by the proposed RMOT approach and qualitatively compare with QOT barycenter \cite[Section 4.1]{peyre2016quantum}. We test on two tensor fields ($n = 4$) supported 2-$d$ grids.  

Figure~\ref{barycenter_fig} compares barycenter from QOT (top) and {\algname} (bottom) initialized from the normalized solution of QOT. We observe that the QOT solution is not optimal when the marginal constraint is enforced and the barycenter obtained does not lie in the simplex of tensors. 
\revision{Such a claim is strengthened by comparing the objective value versus the optimal value, obtained by the CVX toolbox \cite{grant2014cvx}. The objective can be further decreased when initialized from the (normalized) QOT solution, see more discussions in Appendix \ref{appendix:additional_experiments}. }

\begin{figure}
\centering
  \subfloat{
	\begin{minipage}[c][0.08\textwidth]{
	   0.08\textwidth}
	   \centering
	   \includegraphics[width=0.5\textwidth]{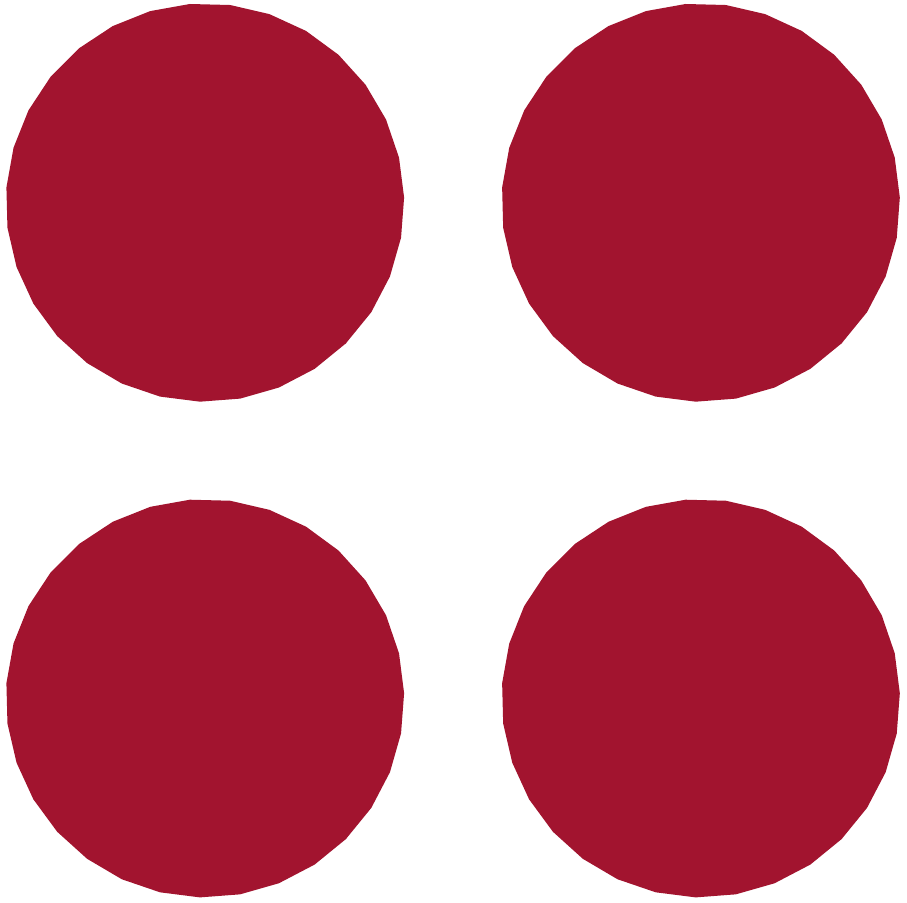}
	\end{minipage}}
  \subfloat{
	\begin{minipage}[c][0.08\textwidth]{
	   0.08\textwidth}
	   \centering
	   \includegraphics[width=0.5\textwidth]{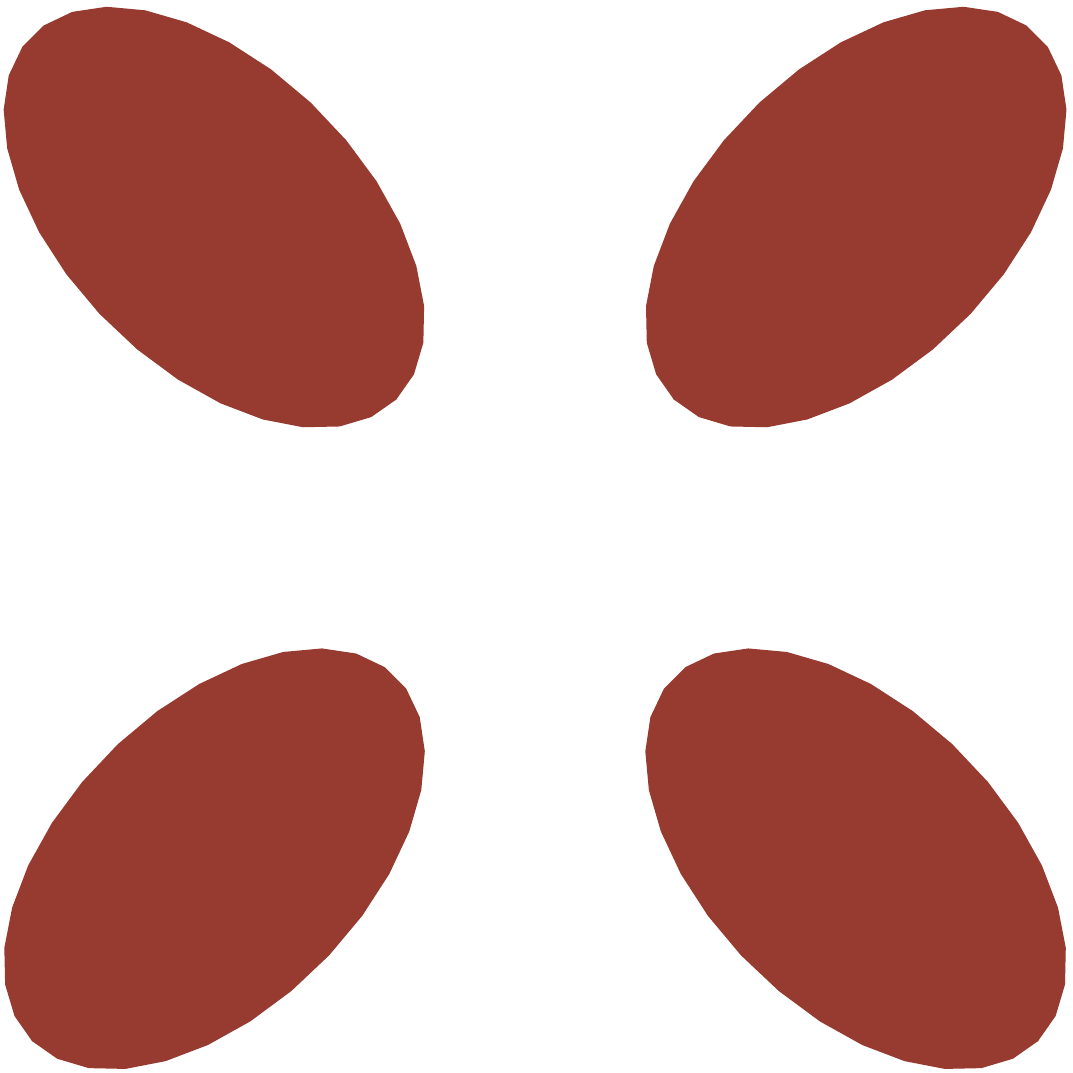} \\[8pt]
	   \includegraphics[width=0.5\textwidth]{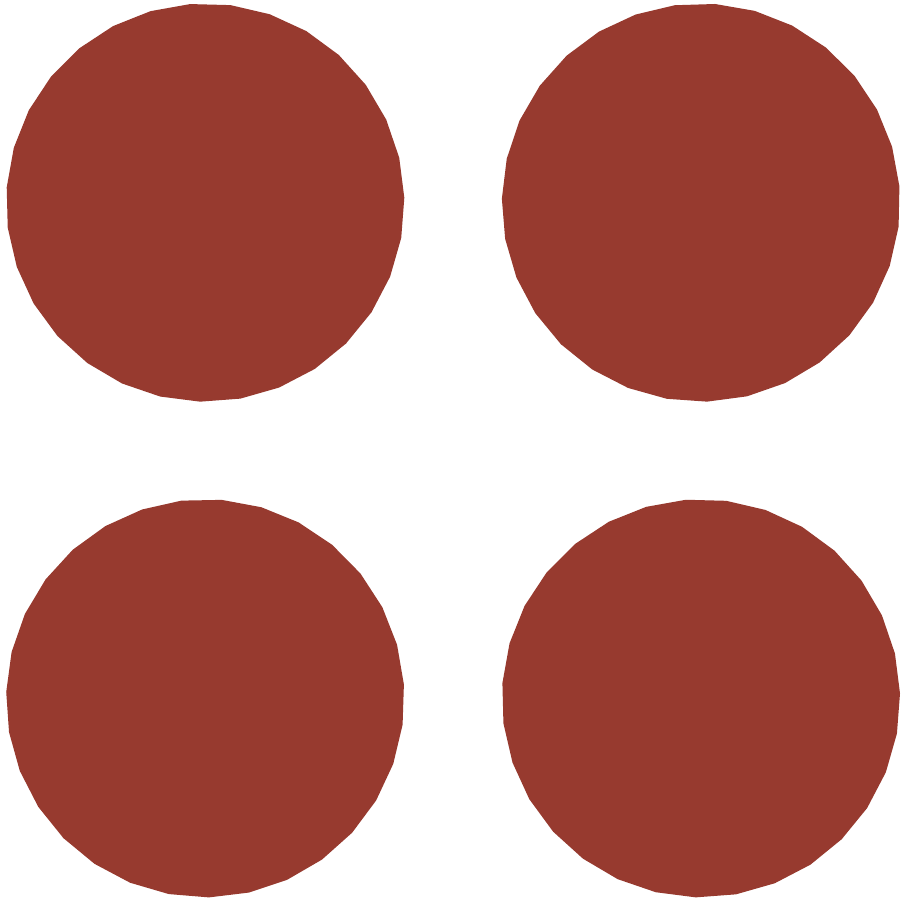}
	\end{minipage}}
	\subfloat{
	\begin{minipage}[c][0.08\textwidth]{
	   0.08\textwidth}
	   \centering
	   \includegraphics[width=0.5\textwidth]{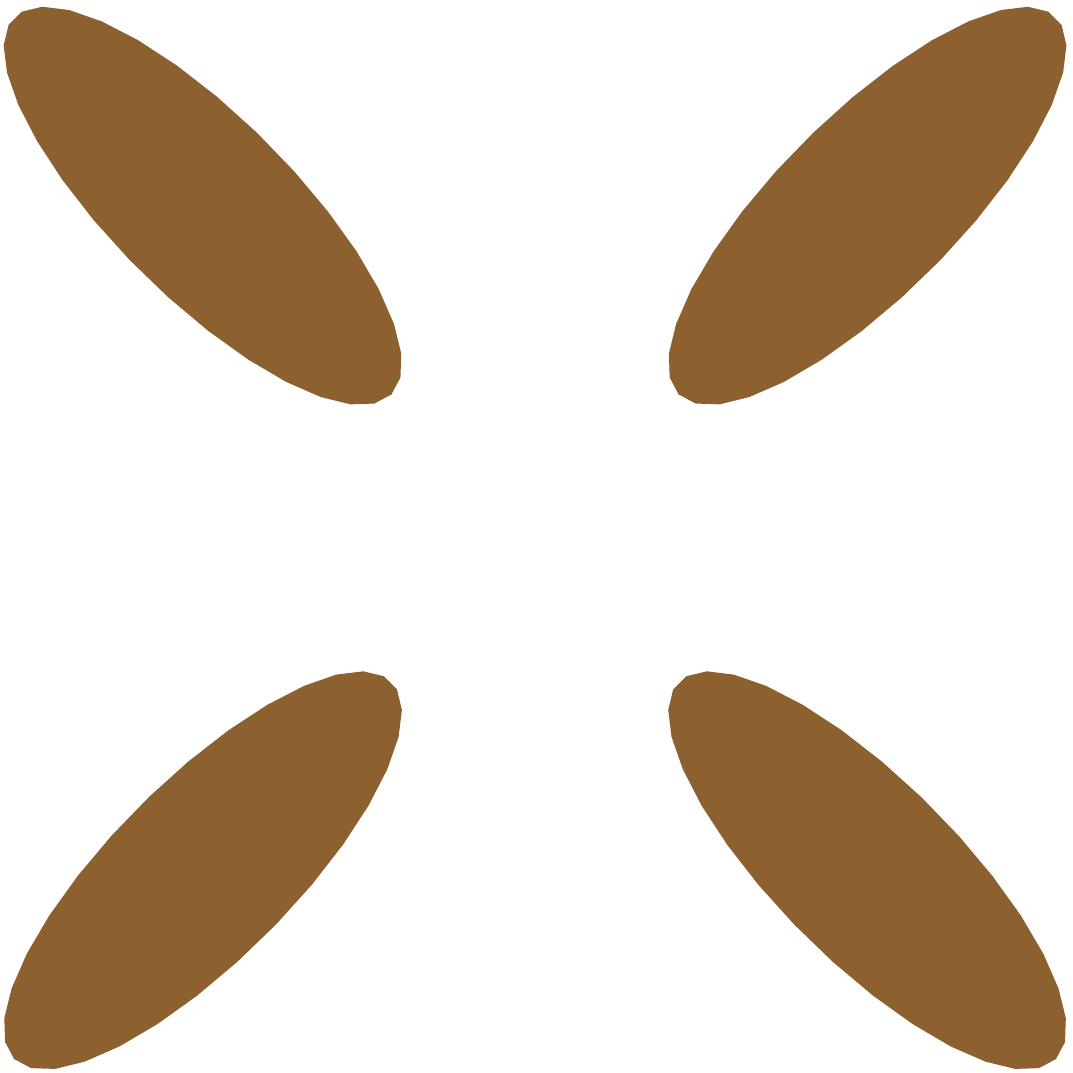} \\[8pt]
	   \includegraphics[width=0.5\textwidth]{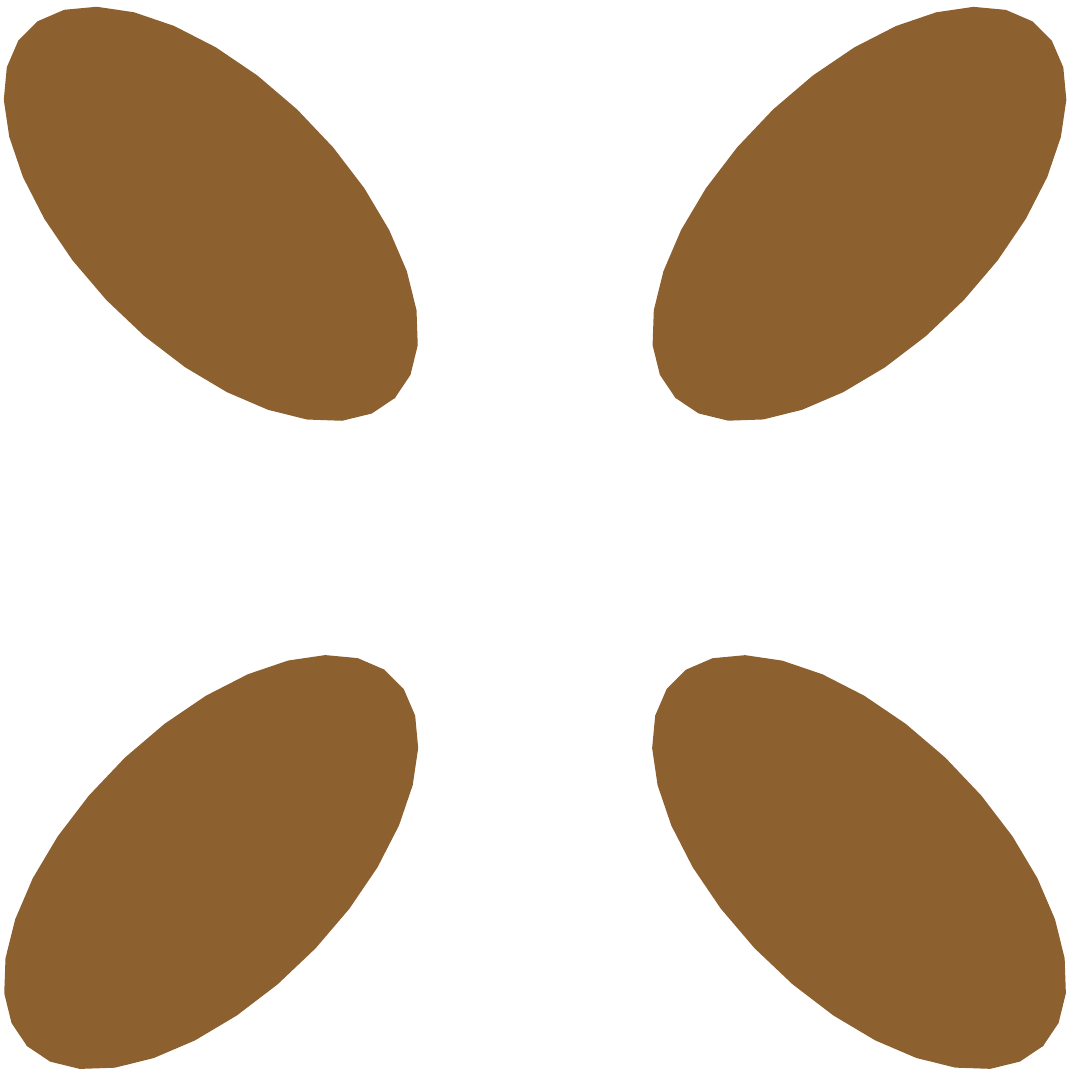}
	\end{minipage}}
	\subfloat{
	\begin{minipage}[c][0.08\textwidth]{
	   0.08\textwidth}
	   \centering
	   \includegraphics[width=0.5\textwidth]{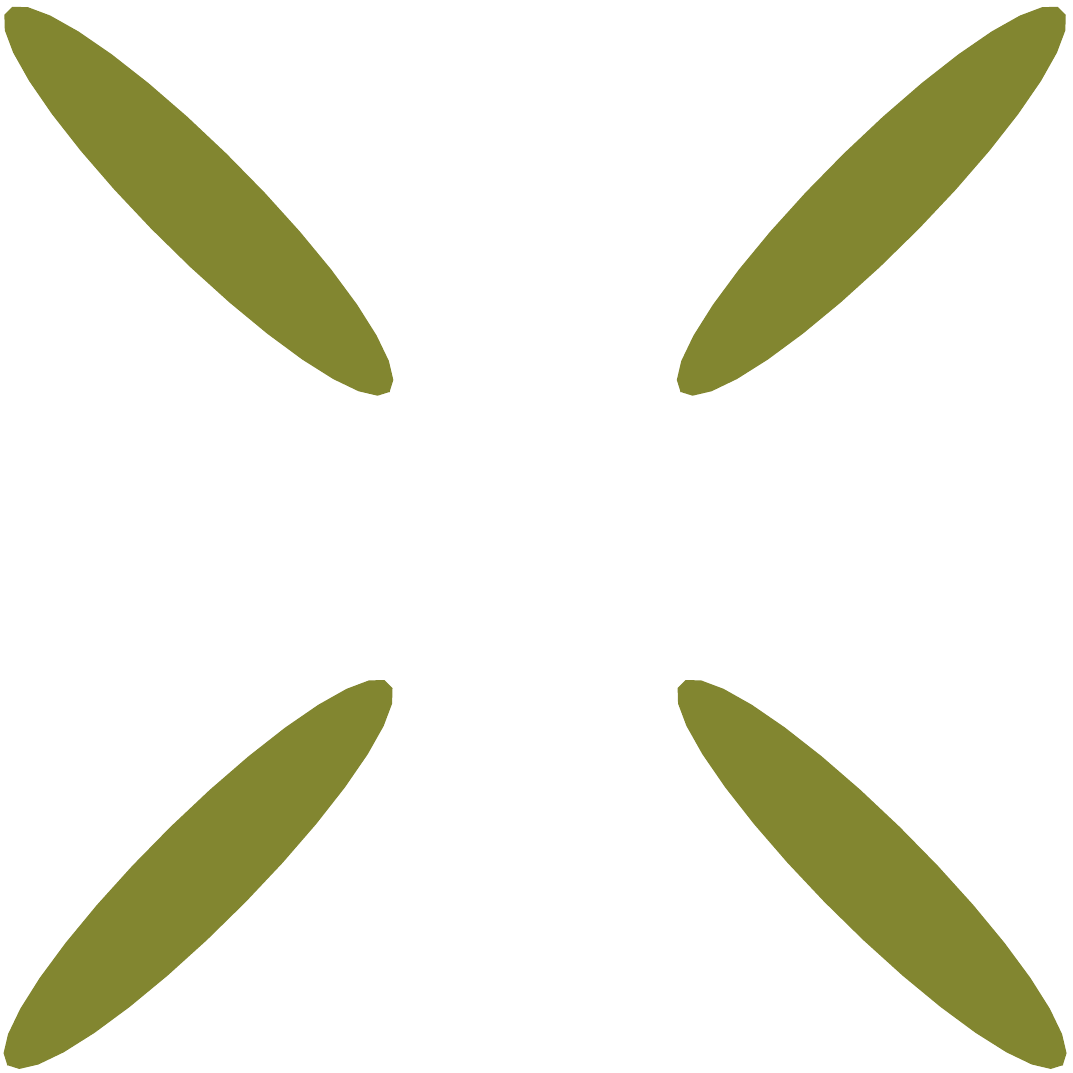} \\[8pt]
	   \includegraphics[width=0.5\textwidth]{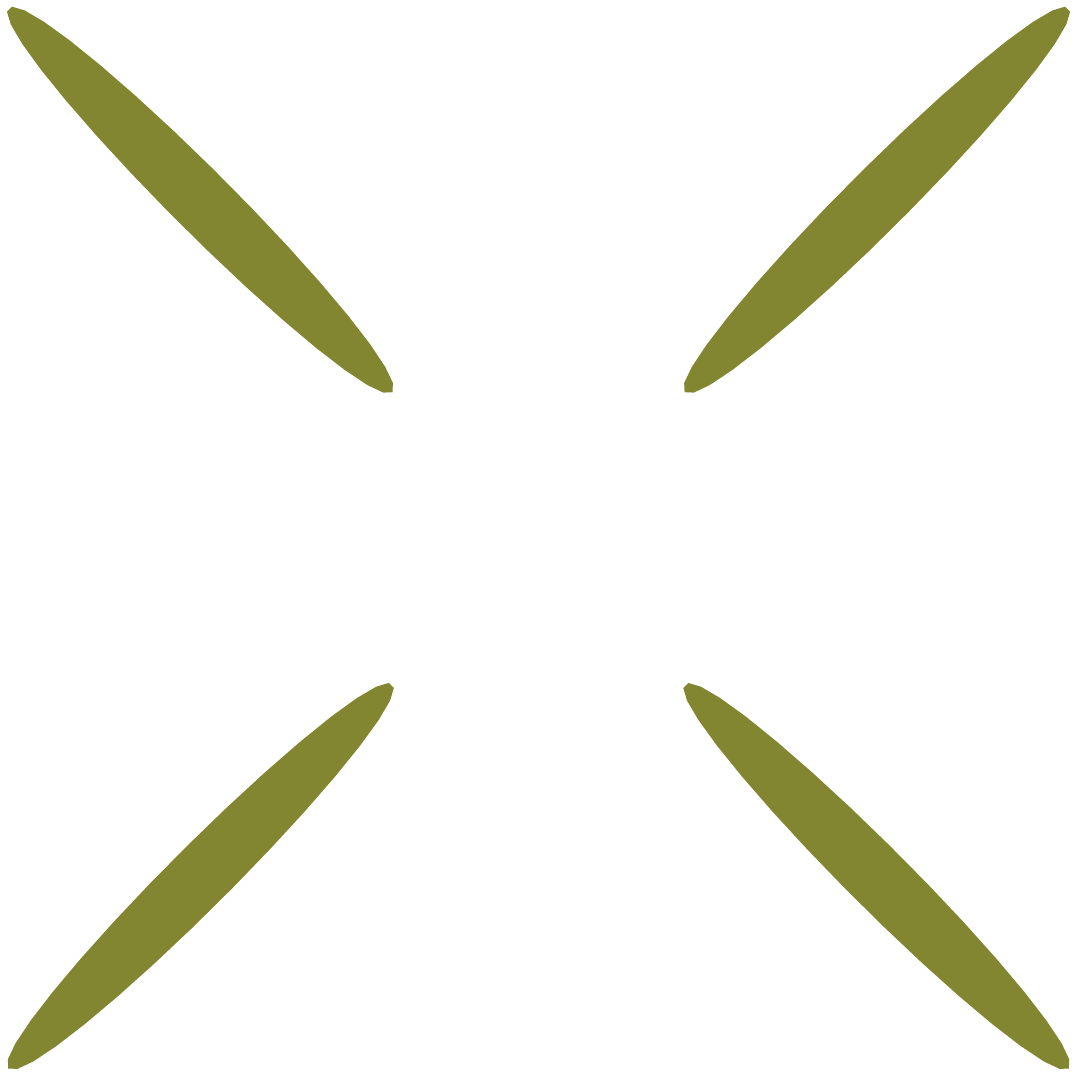}
	\end{minipage}}
  \subfloat{
	\begin{minipage}[c][0.08\textwidth]{
	   0.08\textwidth}
	   \centering
	   \includegraphics[width=0.5\textwidth]{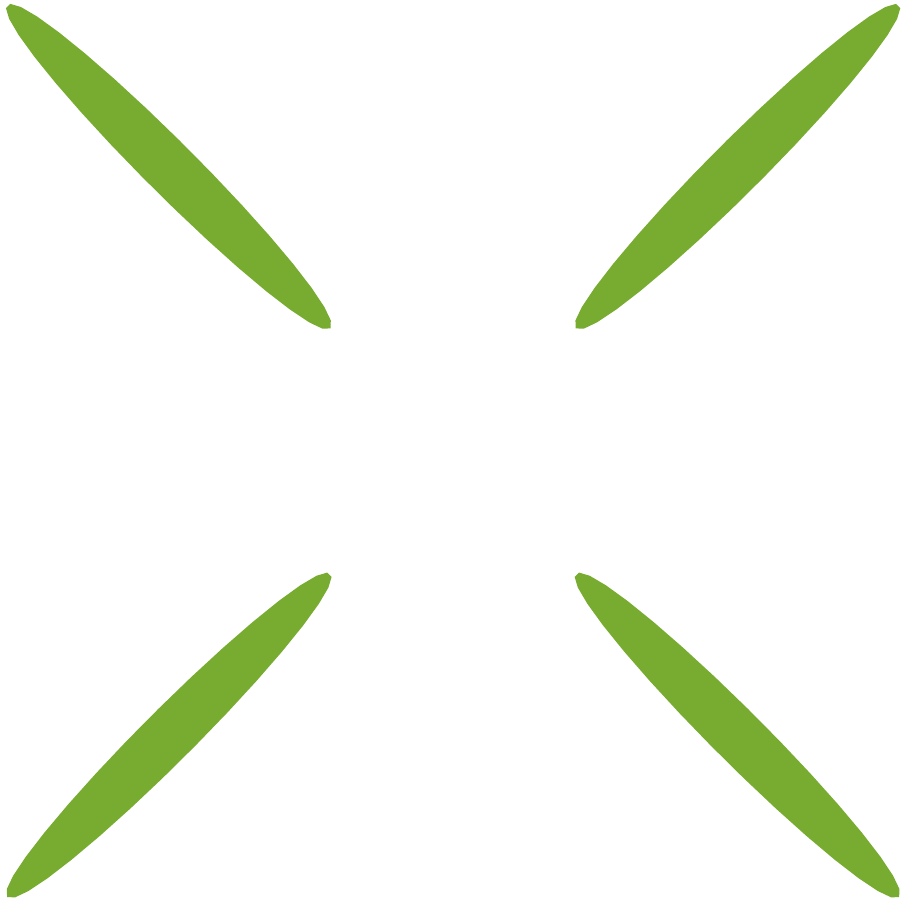}
	\end{minipage}}
\caption{Barycenter interpolation. From left to right $t = 0$ (input), $t= 0.25, 0.5, 0.75$ (barycenters), $t=1$ (input). The top row is QOT and the bottom is {\algname}.}
\label{barycenter_fig}
\end{figure}

\section{Conclusion}\label{sec:conclusion}
In this paper, we have discussed the balanced optimal transport (OT) problem involving SPD matrix-valued measures. For the SPD matrix-valued OT problem, the coupling matrix is a block matrix where each block is a symmetric positive definite matrix. The set of such coupling matrices can be endowed with Riemannian geometry, which enables optimization both linear and non-linear objective functions. We have also shown how the SPD-valued OT setup extend many optimal transport problems to general SPD-valued marginals, including the Wasserstein barycenter and the Gromov-Wasserstein (GW) discrepancy. Experiments in a number of applications confirm the benefit of our approach.

\bibliographystyle{apalike}
\bibliography{references}

\clearpage

\appendix

\section{Convergence of block matrix balancing algorithm and validity of retraction}
\label{convergence_bmb_retr_appendix}
In Section \ref{manifold_geometry_sect}, we generalize the matrix scaling algorithm to block matrix cases, which is essential to derive the retraction for the manifold $\M_{m,n}^d$. Here, we empirically show that the algorithm quickly converges and the proposed retraction is valid and satisfies the two conditions: 1) $R_x(0) = x$ and 2) $\D R_x(0)[u] = u$, where $\D f(x)[u]$ is the derivative of a function at $x$ along direction $u$.

{\bf Convergence.} We show in Figure \ref{algorithm_retr_fig} the convergence of the proposed block matrix balancing procedure in Algorithm \ref{bmb_algorithm}. We generate the marginals as random SPD matrices for different dimensions $d$ and size $m,n$. The convergence is measured as the relative gap to satisfy the constraints. We observe that the number of iterations for convergence are similar with different parameters while the runtime increases by increasing the dimension and size. 

{\bf Validity of retraction.} The first condition of retraction is easily satisfied as $R_\bGamma(\bzero) = {\rm MBalance}(\bGamma) = \bGamma$. For the second one, we have for any $\bGamma \in \M_{m,n}^d$ and $\bU \in T_\bGamma \M_{m,n}^d$,
\begin{align*}
    \D R_\bGamma(\bzero)[\bU] = \lim_{h \xrightarrow{} 0} \frac{R_\bGamma(h\bU) -R_\bGamma(\bzero)}{h} .
\end{align*}
Hence, we need to numerically verify $R_\bGamma(h\bU) - \bGamma = O(h) \bU$ for any $\bGamma, \bU$. We compute an approximation error in terms of the inner product on the tangent space $T_\bGamma \M_{m,n}^d$ as
\begin{equation*}
    \varepsilon = \left\vert \langle {\rm P}_\bGamma( R_\bGamma(\revision{h \bU)} - \bGamma ), \bV \rangle_\bGamma - \langle  h\bU,\bV\rangle_\bGamma \right\vert,
\end{equation*}
for any $\bV \in T_\bGamma\M_{m,n}^d$ different from $\bU$. In Figure \ref{algorithm_retr_fig}(c), we show that the slope of the approximation error (as a function of $h$) matches the dotted line $h = 0$, which suggests hat the error $\varepsilon = O(1)$, thereby indicating that the retraction is valid.

\begin{figure*}[!th]
    \centering
    \subfloat[]{\includegraphics[width = 0.24\textwidth, height = 0.2\textwidth]{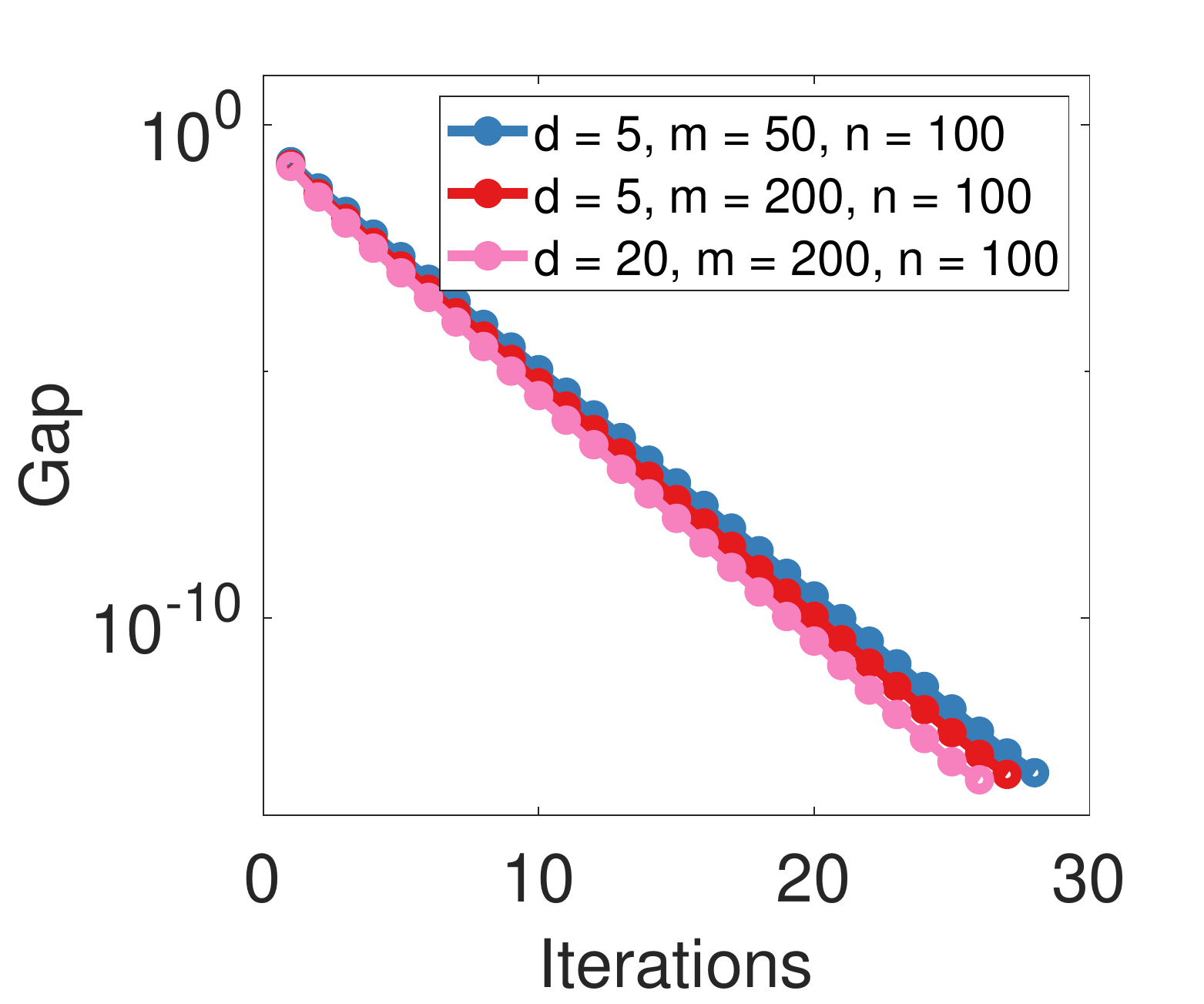}} 
    \hspace*{5pt}
    \subfloat[]{\includegraphics[width = 0.24\textwidth, height = 0.2\textwidth]{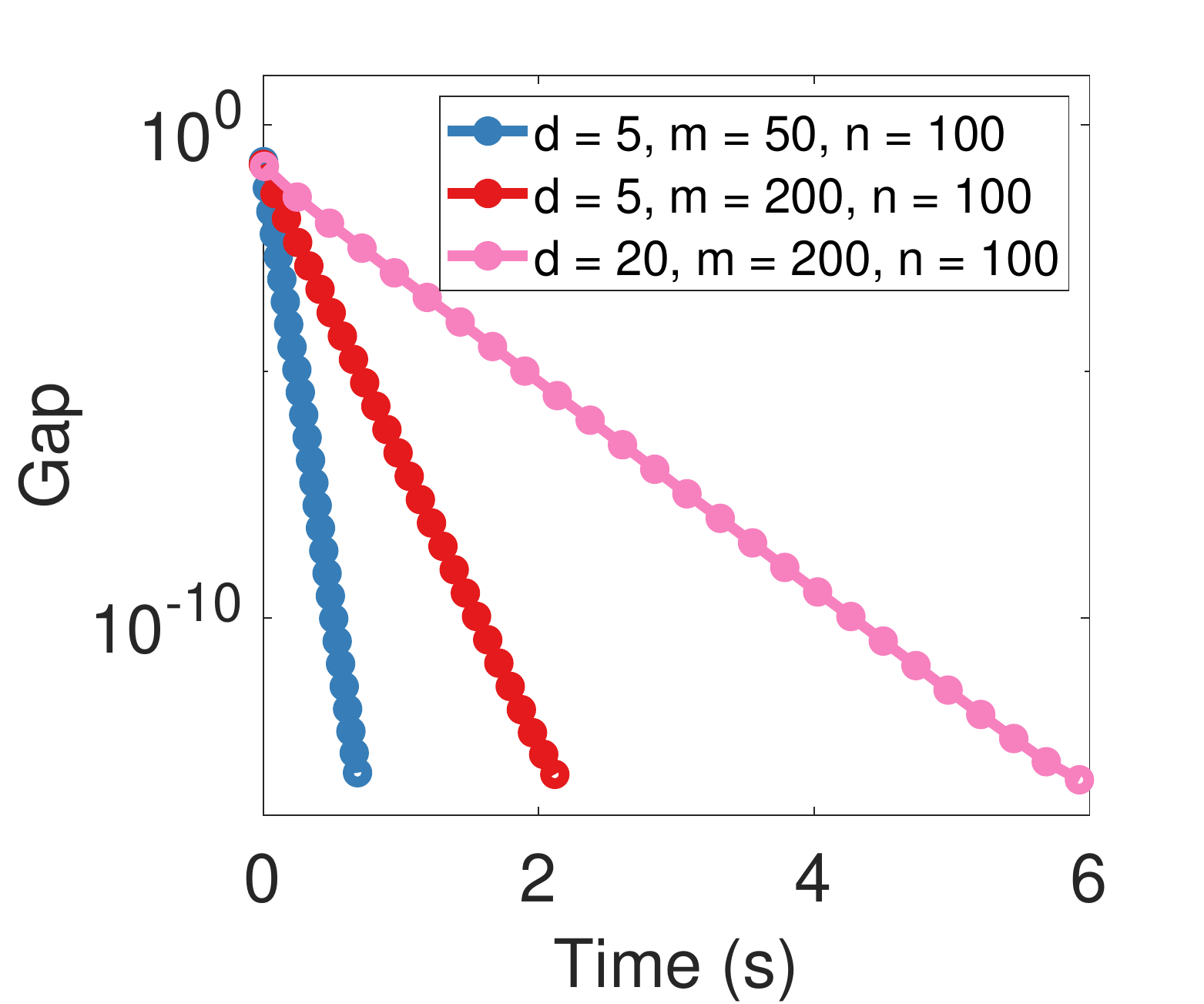}} 
    \hspace*{5pt}
    \subfloat[]{\includegraphics[width = 0.24\textwidth, height = 0.2\textwidth]{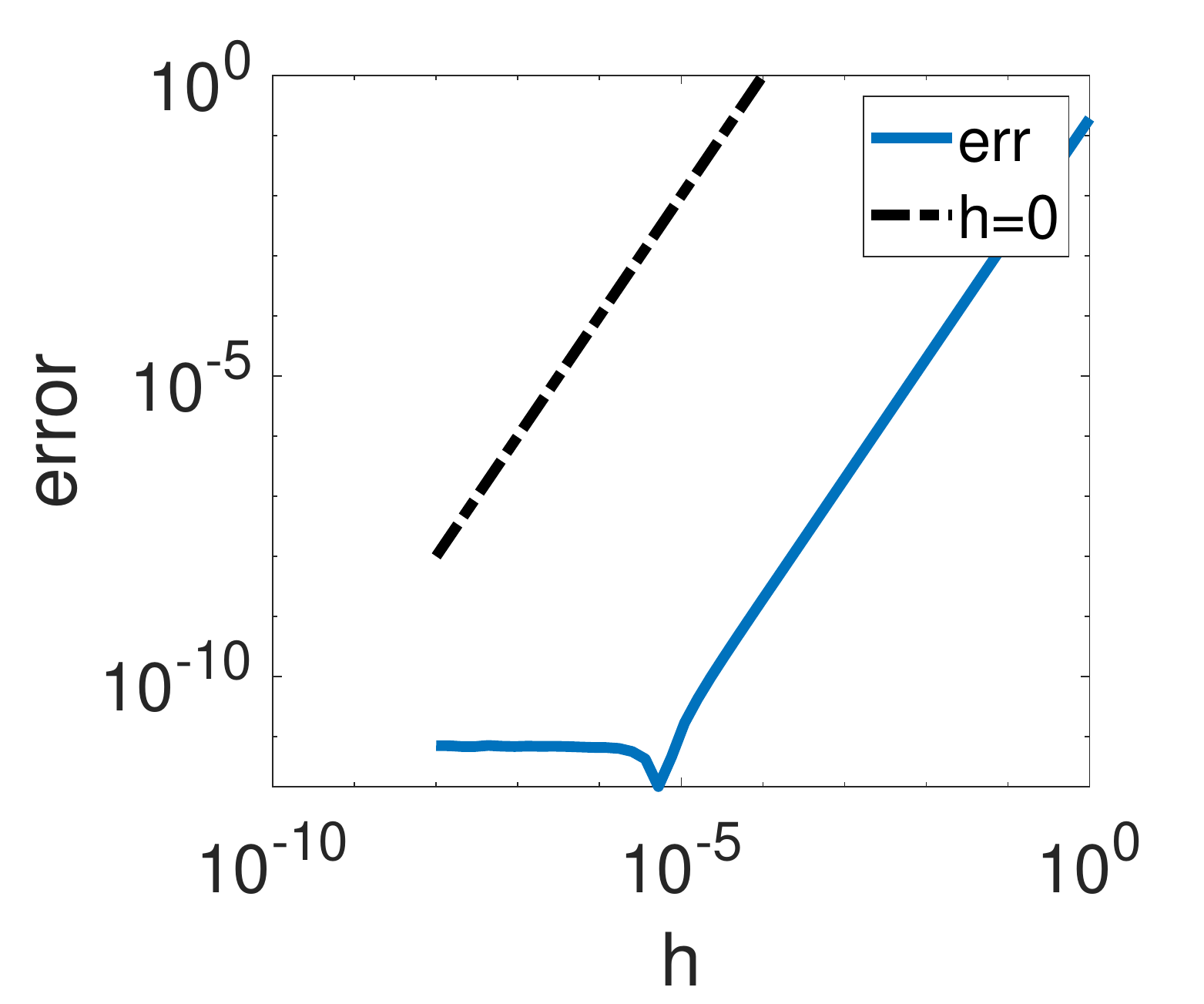}} 
    \caption{\small Convergence of Algorithm \ref{bmb_algorithm} in terms of iterations (a), runtime (b), and validity test for retraction (c). For the retraction to be valid, the slope of the continuous line should match the dotted line (which represents the line $h = 0$) and also start from $0$ when $h$ tends to $0$.}
    \label{algorithm_retr_fig}
\end{figure*}

\section{Discussion on construction of matrix-valued cost}
\label{cost_choice_appendix}
As highlighted in Proposition \ref{mw_distance_prop} for ${\rm MW}(\bP, \bQ)$ to be a metric for probability measures there are some conditions for the cost $[\bC_{i,j}]_{m \times n}$ to satisfy. In the following, we give some examples of how such costs are constructed:
\begin{enumerate}
    \item Let the samples are given by $\{ \bX_i \}_{i \in [m]}$, $\{ \bY_j\}_{j\in [n]}$, where $\bX_i, \bY_j \in \sR^{d \times s}$.  Define $\bC_{i,j} = d(\bX_i ,\bY_j)^2 \, \bI$, where $d: \sR^{d \times s} \times \sR^{d \times s} \xrightarrow{} \sR_{+}$ is a distance function. 
    \item Let the samples are given by $\{ \bX_i \}_{i \in [m]}$, $\{ \bY_j\}_{j\in [n]}$, where $\bX_i, \bY_j \in \sR^{d \times s}$, where $s\geq d$. Assume the matrix $\bX_i-\bY_j$ has column full rank. Define $\bC_{i,j} = (\bX_i - \bY_j)(\bX_i - \bY_j)^\top$. 
\end{enumerate}
\begin{proof}
(1) The first definition of cost trivially satisfies all the conditions due to the metric properties of a well-defined scalar-valued distance.

(2) For the second definition of cost, The first two conditions, i.e., symmetric and positive definite conditions are easily satisfied and we only need to verify the third condition in Proposition \ref{mw_distance_prop}. The third condition is also satisfied due to the triangle inequality of Mahalanobis distance metric in the vectorized form. That is, for any $\bA \succeq \bzero$, we consider three sets of samples $\{\bX_i\}, \{ \bY_k\}, \{ \bZ_j\} \subset \sR^{d \times s}$. Then, we have
\begin{align*}
    \sqrt{\trace(\bC_{i,j} \bA)} &= \sqrt{\trace( (\bX_i - \bZ_j)^\top \bA (\bX_i -\bZ_j) )} \\
    &= \sqrt{(\Vec(\bX_i) -\Vec(\bZ_j))^\top (\bI \otimes \bA) (\Vec(\bX_i) - \Vec(\bZ_j))} \\
    &\leq  \sqrt{(\Vec(\bX_i) - \Vec(\bY_k))^\top (\bI \otimes \bA) (\Vec(\bX_i) - \Vec(\bY_k))} \\
    &\quad + \sqrt{(\Vec(\bY_k) - \Vec(\bZ_j))^\top (\bI \otimes \bA) (\Vec(Y_k) - \Vec(\bZ_j))}\\
    &=  \sqrt{ \trace((\bX_i -\bY_k)^\top \bA (\bX_i - \bY_k)) } + \sqrt{ \trace((\bY_k -\bZ_j)^\top \bA (\bY_k -\bZ_j)) } \\
    &= \sqrt{\trace(\bC_{i,k} \bA)} + \sqrt{\trace(\bC_{k,j} \bA)},
\end{align*}
where $\Vec(\bC)$ denotes the vectorization of matrix $\bC$ by stacking the columns.
\end{proof}


\section{Additional experiments} \label{appendix:additional_experiments}
In this section, we give additional experiments to further substantiate the claims made in the main text.

\subsection{Tensor field optimal transport mass interpolation}
\revision{
We first provide more details on displacement interpolation considered in the experiment. After we obtain the optimal $\bGamma^*$, for $t \in [0,1]$, we compute the interpolated measure at $t$ as 
\begin{equation*}
    \sum_{i,j} ( (1 - t) \bP_i + t \bQ_j ) \bGamma_{i,j} \delta_{x^t_{i,j}},
\end{equation*}
where $x^t_{i,j}$ is the interpolated location on the $2$-d grid.

In addition to the experiments presented in the main texts, we also show other examples of tensor fields mass interpolation in \revision{Figures \ref{tensor_field_mass_1d_appendix}} and \ref{tensor_field_mass_2d_appendix}. In Figure \ref{tensor_field_mass_1d_appendix}, the inputs are given as $1$-d tensor fields, which are the first and last row for each subfigure. We compare the interpolation given by the linear interpolation (first column), QOT with different values of $\rho$ and RMOT (last column).
In Figure \ref{tensor_field_mass_2d_appendix}, Input-1 and Input-5 are with $t = 0$ and $t = 1$, respectively. QOT-2 and RMOT-2 are with $t = 0.25$. QOT-3 and RMOT-3 are with $t = 0.5$. QOT-4 and RMOT-4 are with $t = 0.75$.
}

\begin{figure*}[!th]
\captionsetup{justification=centering}

    \centering
    \subfloat{\includegraphics[width = 0.18\textwidth, height = 0.1\textwidth]{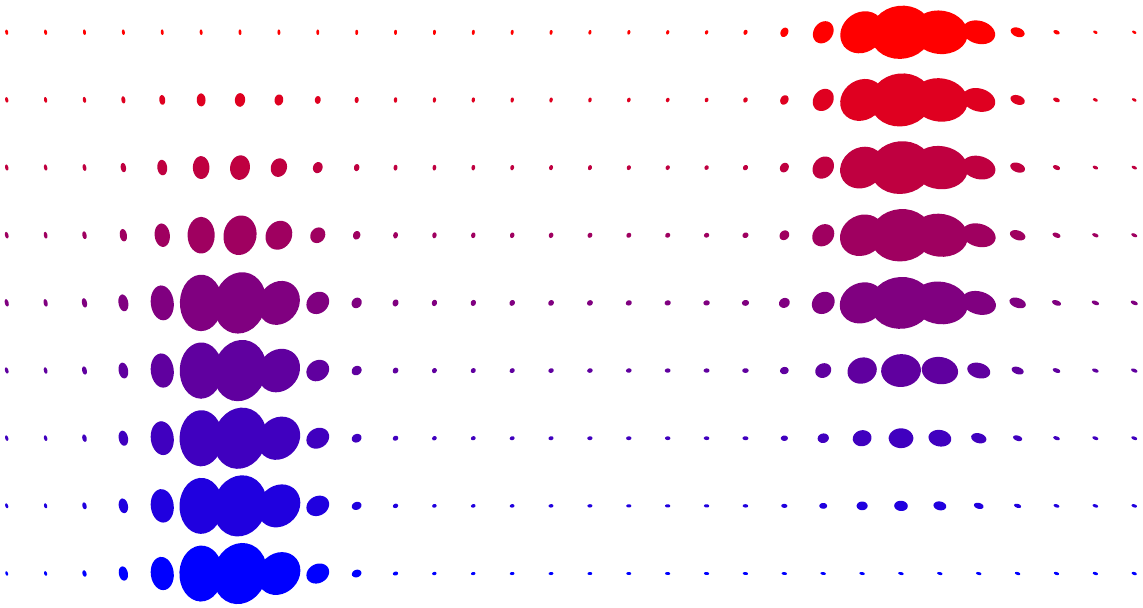}} 
    \hspace*{5pt}
    \subfloat{\includegraphics[width = 0.18\textwidth, height = 0.1\textwidth]{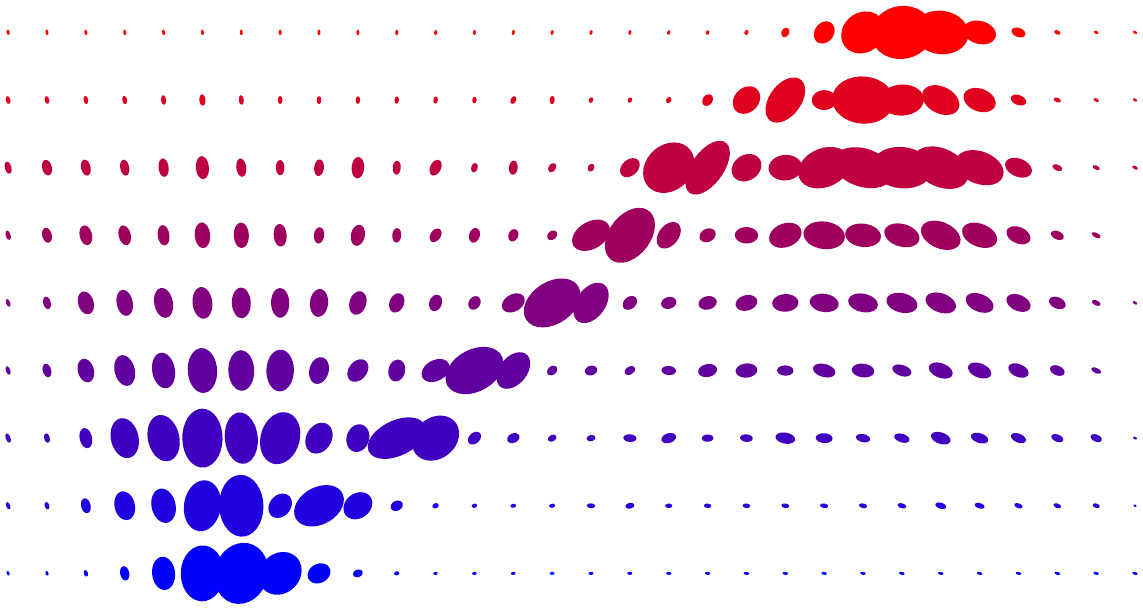}}
    \hspace*{5pt}
    \subfloat{\includegraphics[width = 0.18\textwidth, height = 0.1\textwidth]{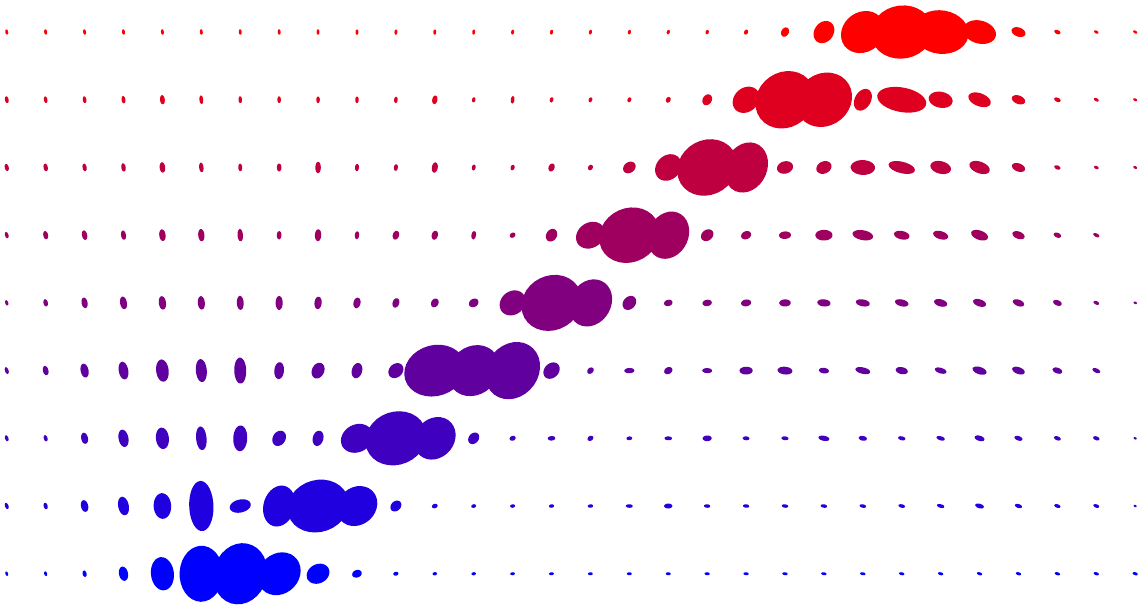}}
    \hspace*{5pt}
    \subfloat{\includegraphics[width = 0.18\textwidth, height = 0.1\textwidth]{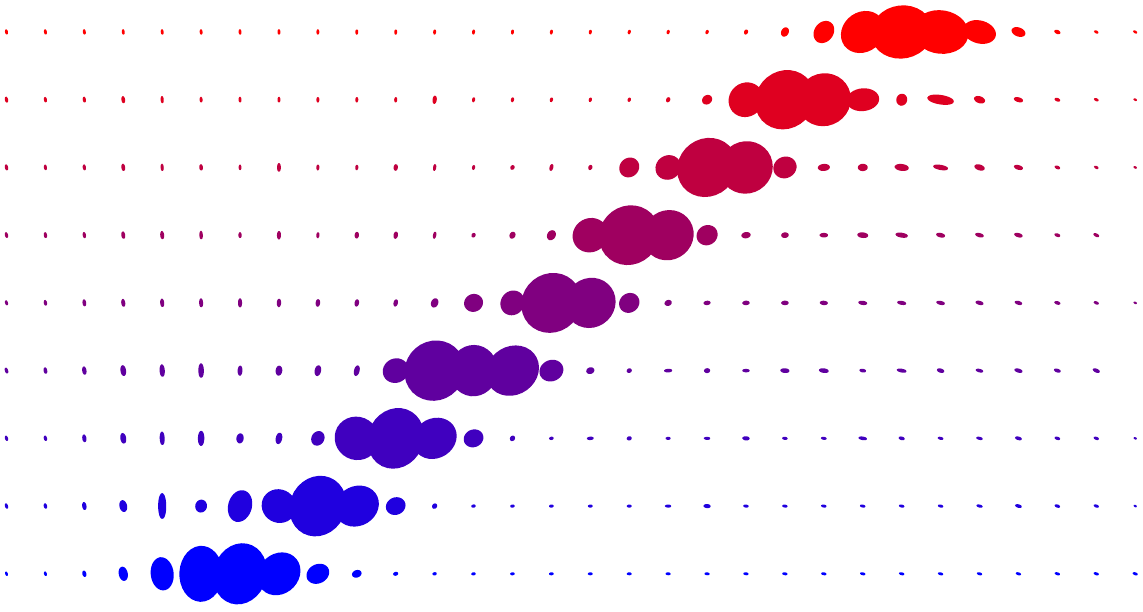}}
    \hspace*{5pt}
    \subfloat{\includegraphics[width = 0.18\textwidth, height = 0.1\textwidth]{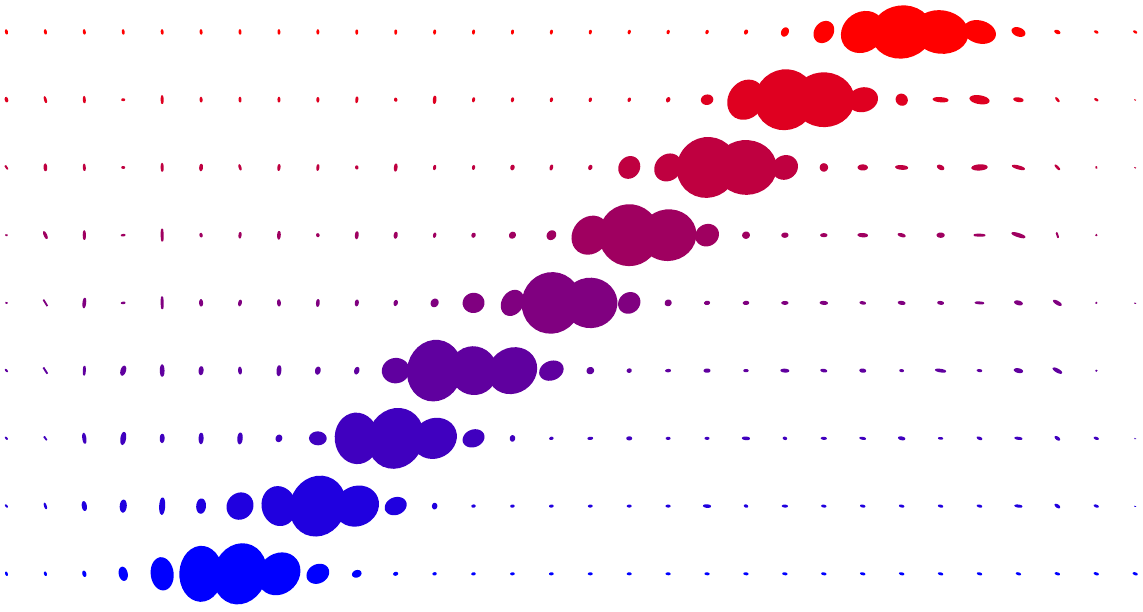}} \\[5pt]
    
    \subfloat{\includegraphics[width = 0.18\textwidth, height = 0.1\textwidth]{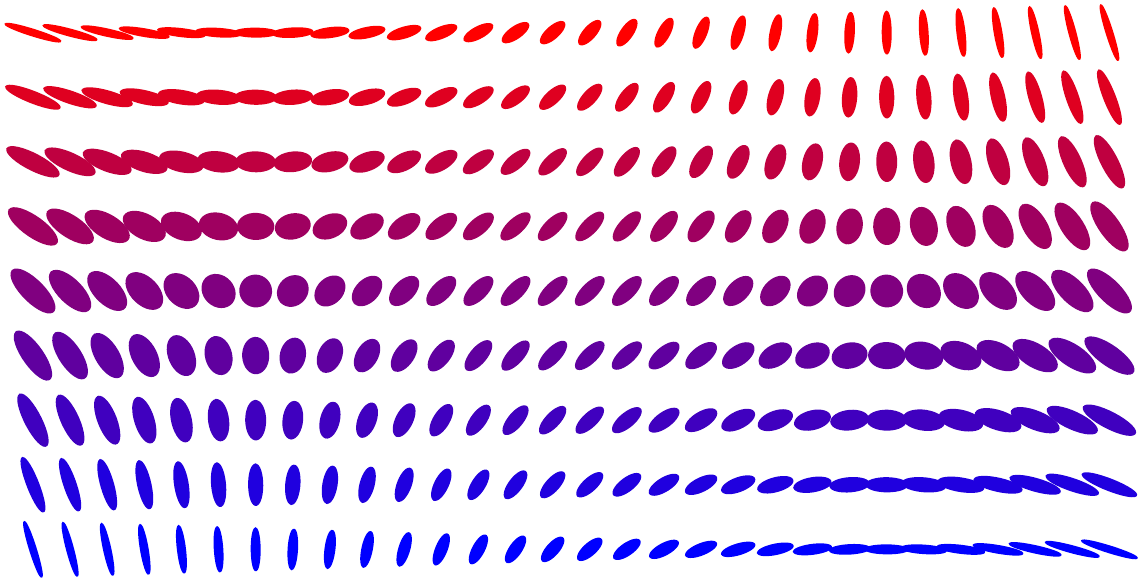}} 
    \hspace*{5pt}
    \subfloat{\includegraphics[width = 0.18\textwidth, height = 0.1\textwidth]{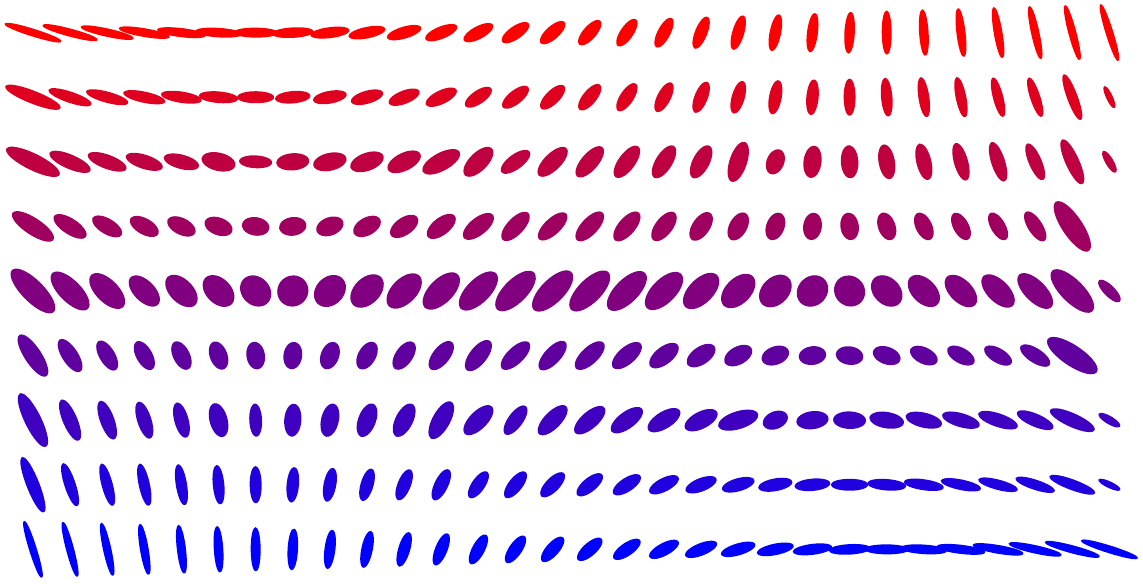}}
    \hspace*{5pt}
    \subfloat{\includegraphics[width = 0.18\textwidth, height = 0.1\textwidth]{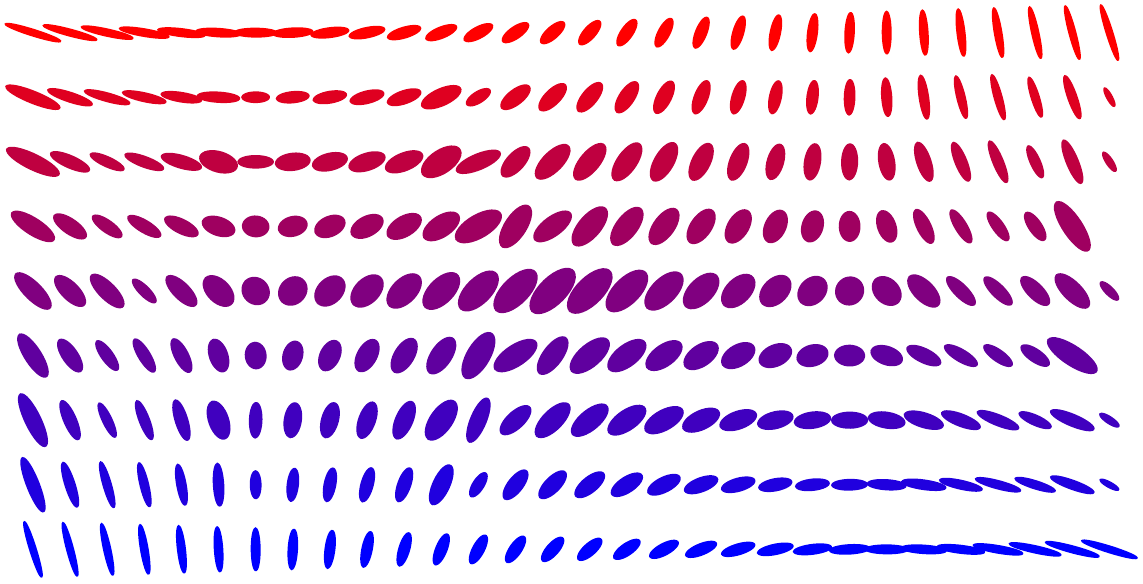}}
    \hspace*{5pt}
    \subfloat{\includegraphics[width = 0.18\textwidth, height = 0.1\textwidth]{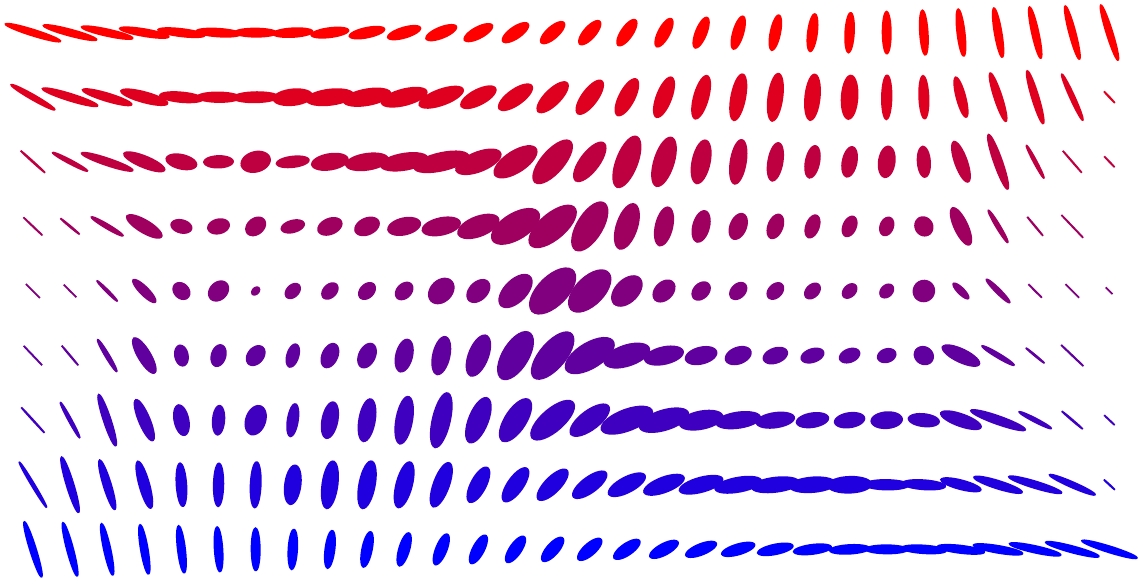}}
    \hspace*{5pt}
    \subfloat{\includegraphics[width = 0.18\textwidth, height = 0.1\textwidth]{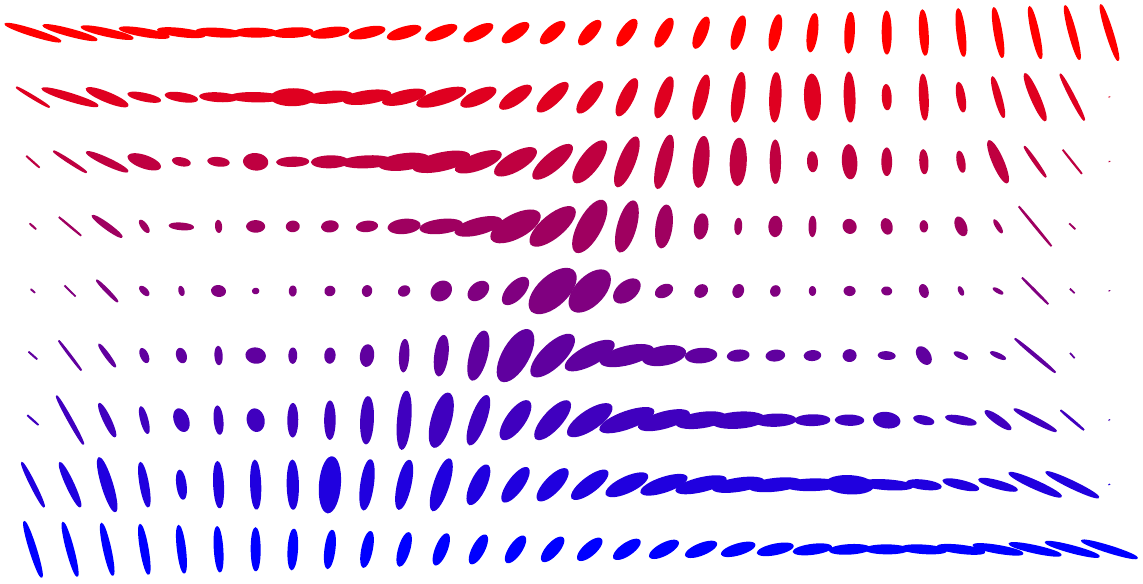}} \\[5pt]
    
    \subfloat{\includegraphics[width = 0.18\textwidth, height = 0.1\textwidth]{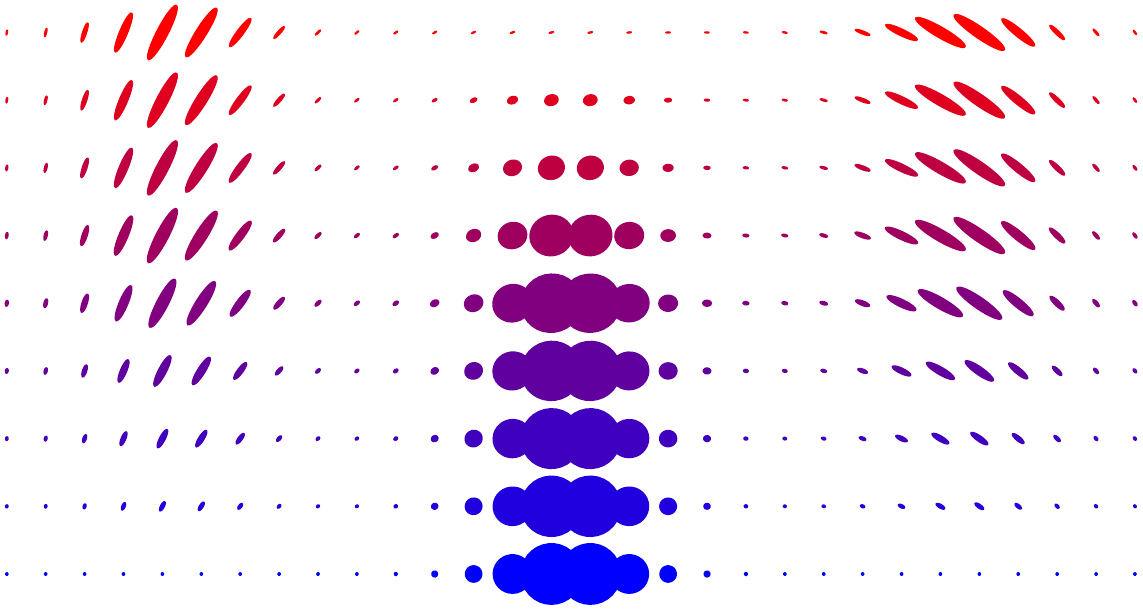}} 
    \hspace*{5pt}
    \subfloat{\includegraphics[width = 0.18\textwidth, height = 0.1\textwidth]{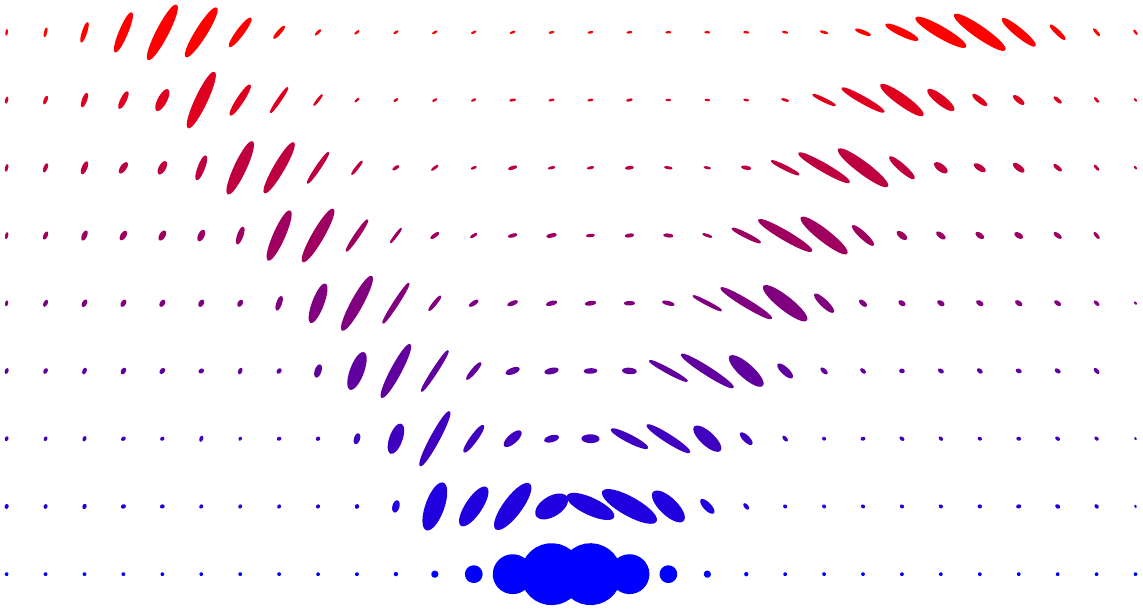}}
    \hspace*{5pt}
    \subfloat{\includegraphics[width = 0.18\textwidth, height = 0.1\textwidth]{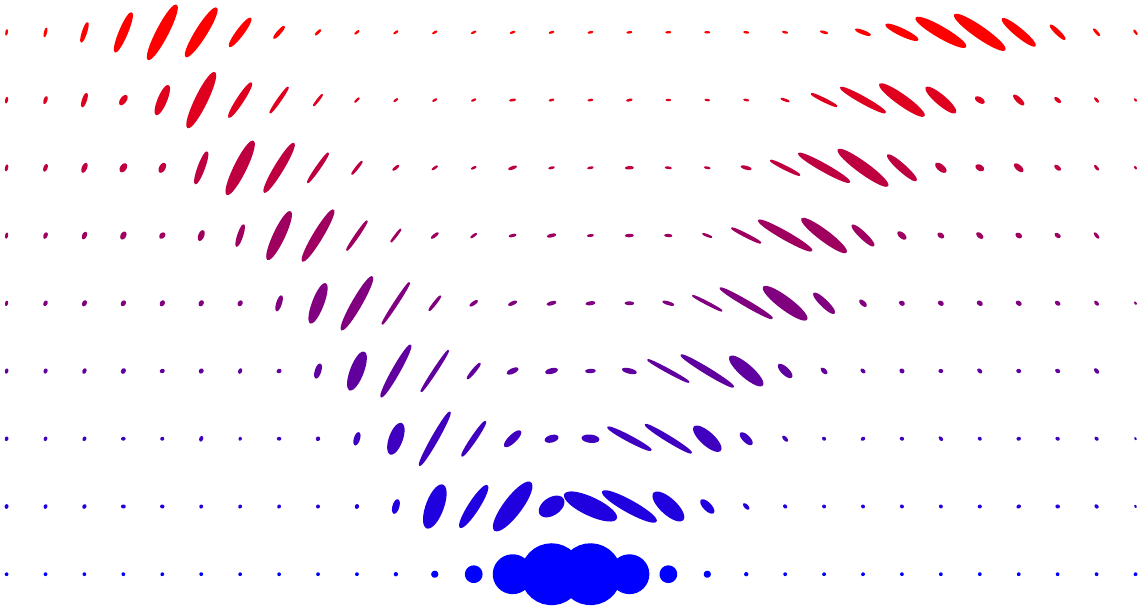}}
    \hspace*{5pt}
    \subfloat{\includegraphics[width = 0.18\textwidth, height = 0.1\textwidth]{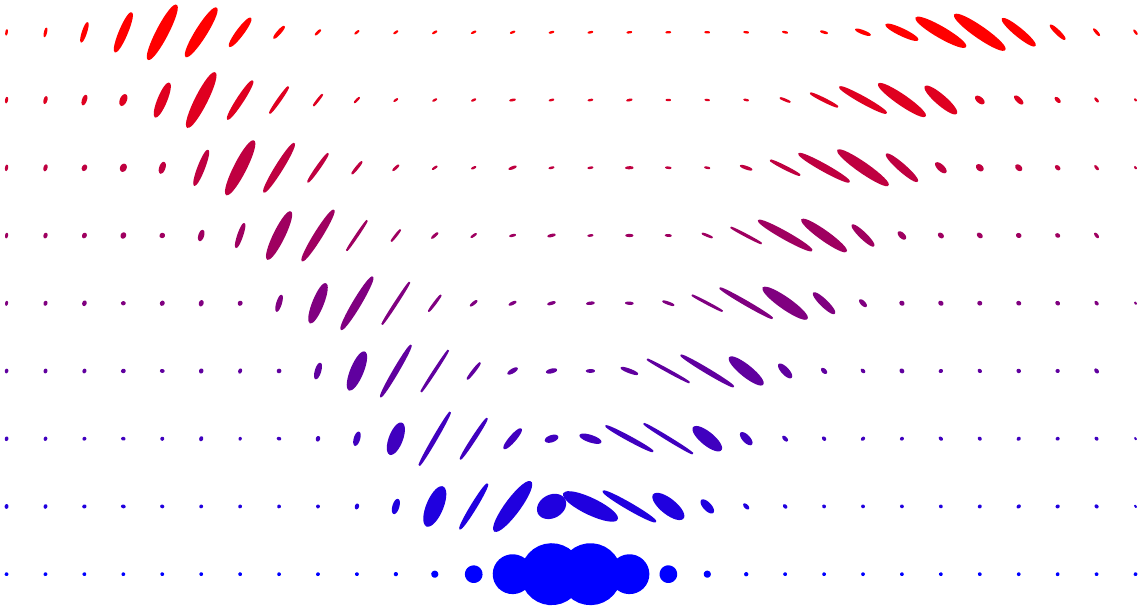}}
    \hspace*{5pt}
    \subfloat{\includegraphics[width = 0.18\textwidth, height = 0.1\textwidth]{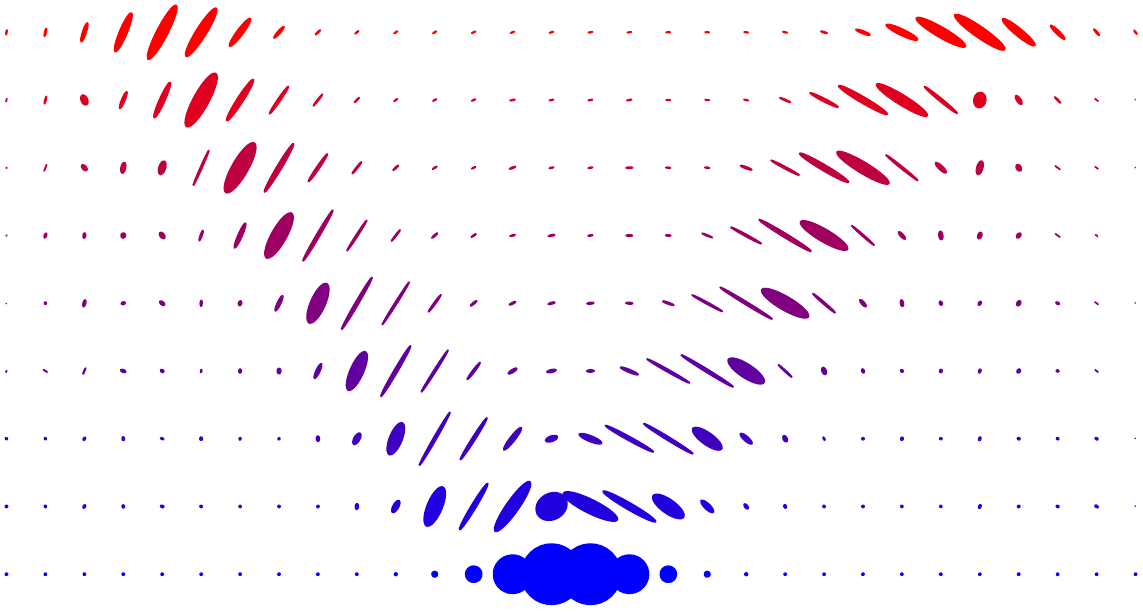}} \\[5pt]
    
    \stackunder[5pt]{\includegraphics[width=0.18\textwidth, height = 0.1\textwidth]{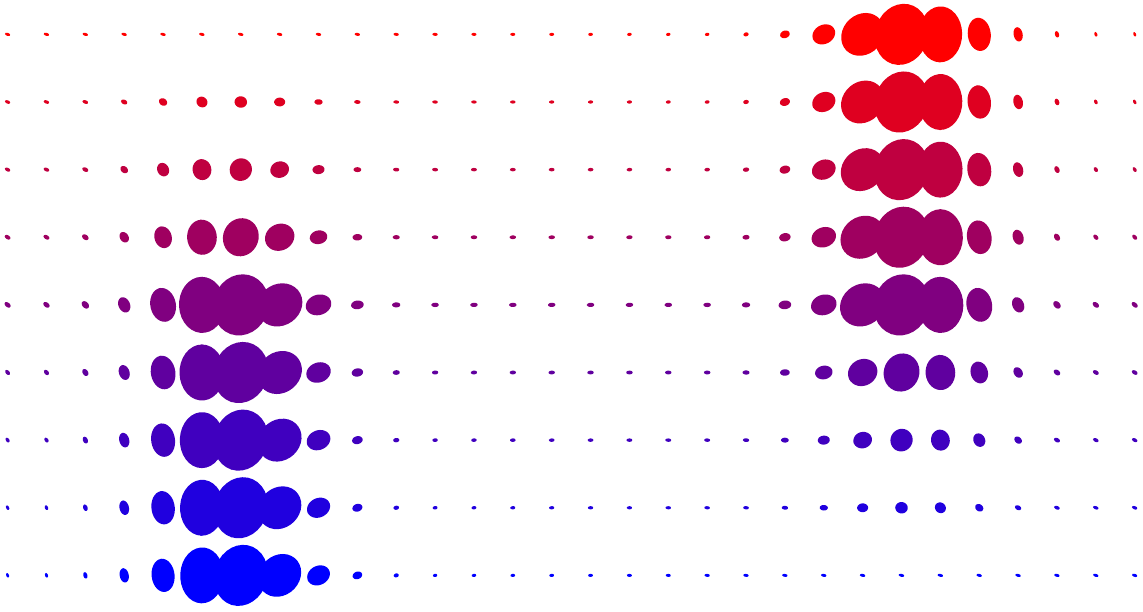}}{\footnotesize Linear}
    \hspace*{2.8pt}
    \stackunder[5pt]{\includegraphics[width=0.18\textwidth, height = 0.1\textwidth]{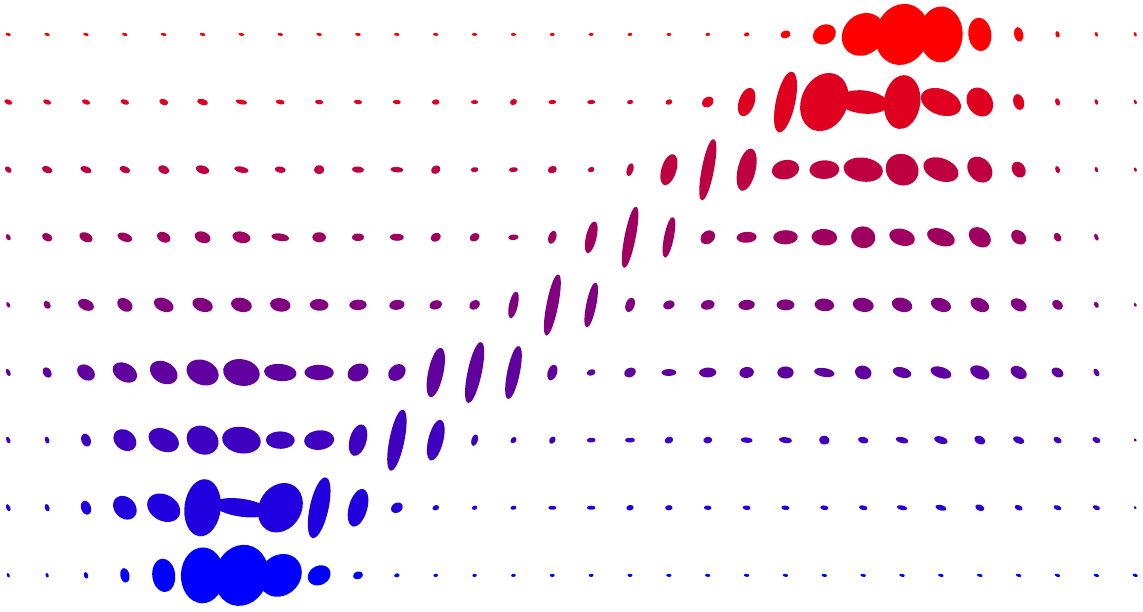}}{\footnotesize QOT ($\rho = 50$)}
    \hspace*{2.8pt}
    \stackunder[5pt]{\includegraphics[width=0.18\textwidth, height = 0.1\textwidth]{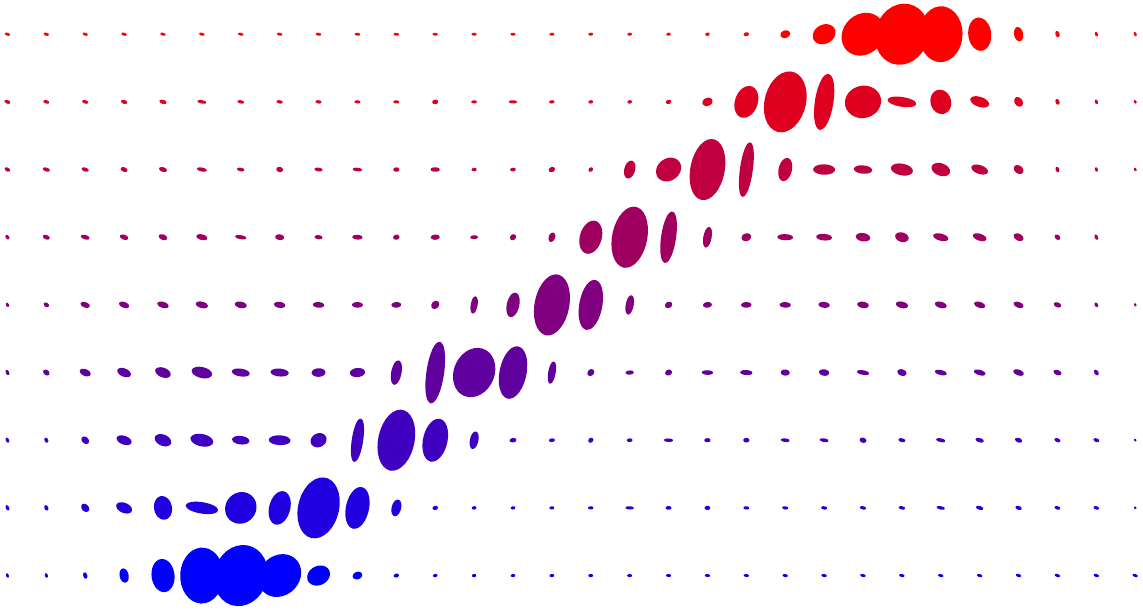}}{\footnotesize QOT ($\rho = 100$)}
    \hspace*{2.8pt}
    \stackunder[5pt]{\includegraphics[width=0.18\textwidth, height = 0.1\textwidth]{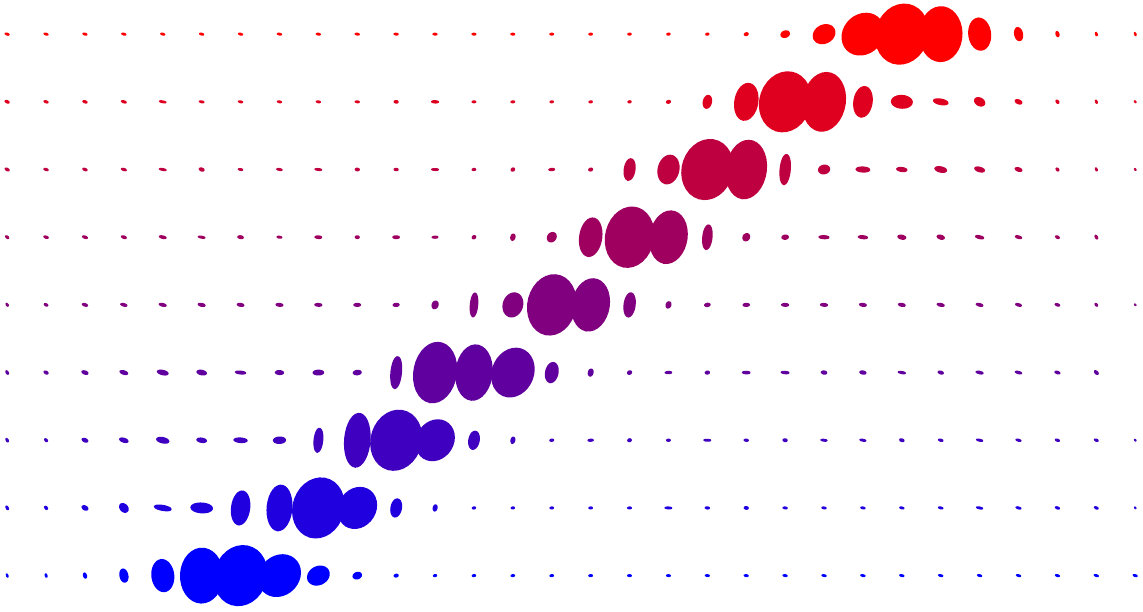}}{\footnotesize QOT ($\rho = 500$)}
    \hspace*{2.8pt}
    \stackunder[5pt]{\includegraphics[width=0.18\textwidth, height = 0.1\textwidth]{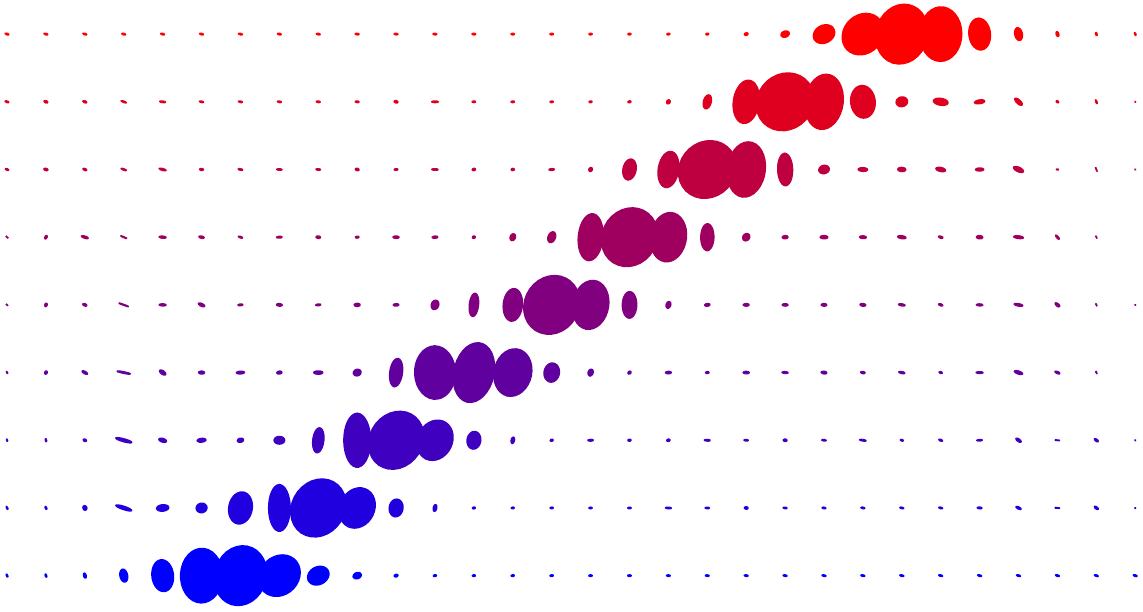}}{\footnotesize {\algname}}
    \caption{\small \revision{1-$d$ tensor fields mass interpolation. Tensor fields are generated as cross-oriented (first row), multi-oriented (second row), split (third row) and iso-oriented (fourth row).}}
    \label{tensor_field_mass_1d_appendix}
\end{figure*}

\begin{figure*}[!th]
\captionsetup{justification=centering}
    \centering
    \subfloat{\includegraphics[width = 0.11\textwidth, height = 0.11\textwidth]{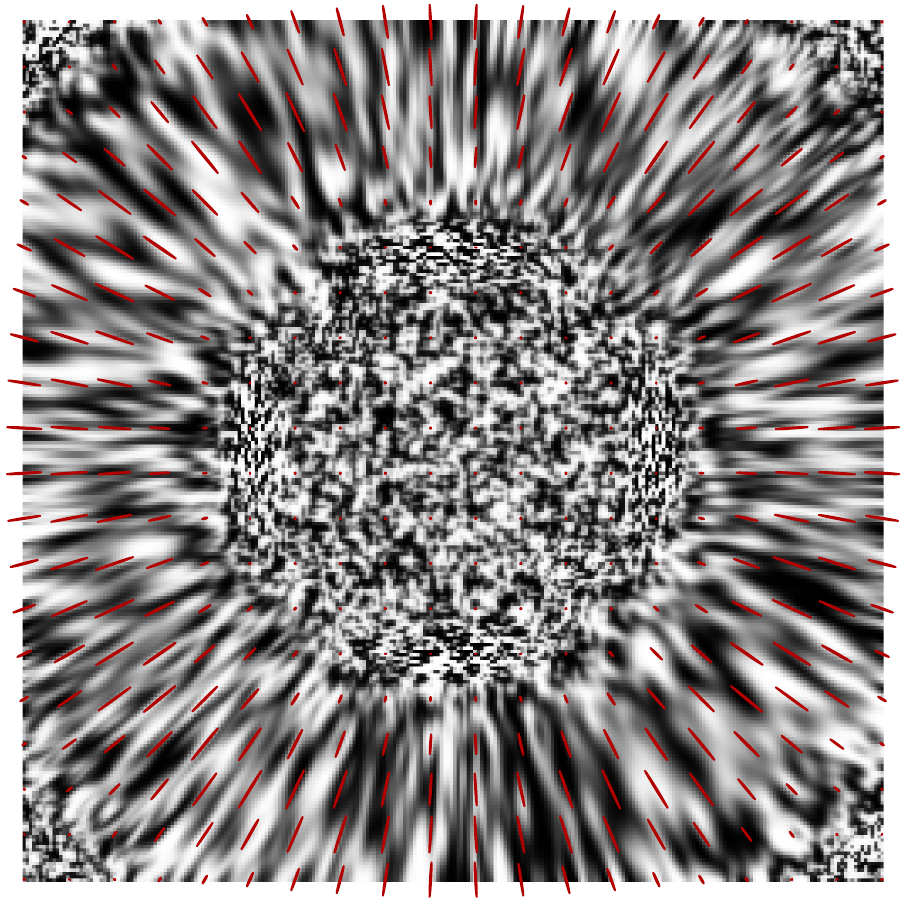}} 
    \hspace*{3pt}
    \subfloat{\includegraphics[width = 0.11\textwidth, height = 0.11\textwidth]{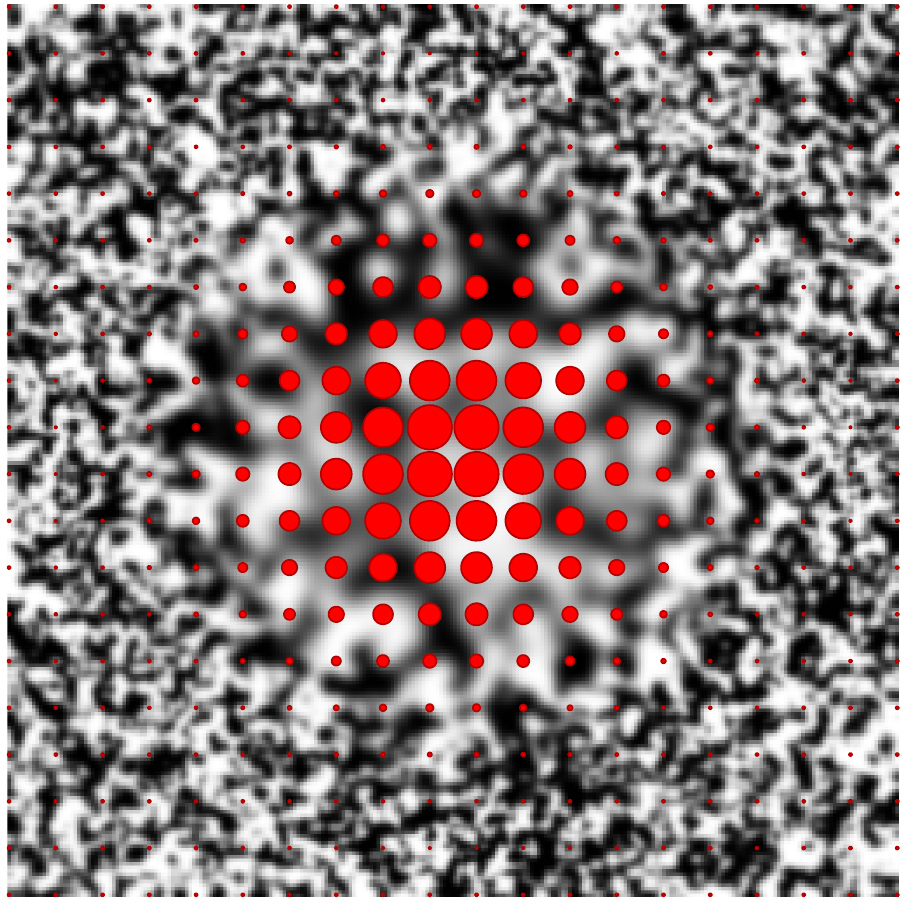}} 
    \hspace*{3pt}
    \subfloat{\includegraphics[width = 0.11\textwidth, height = 0.11\textwidth]{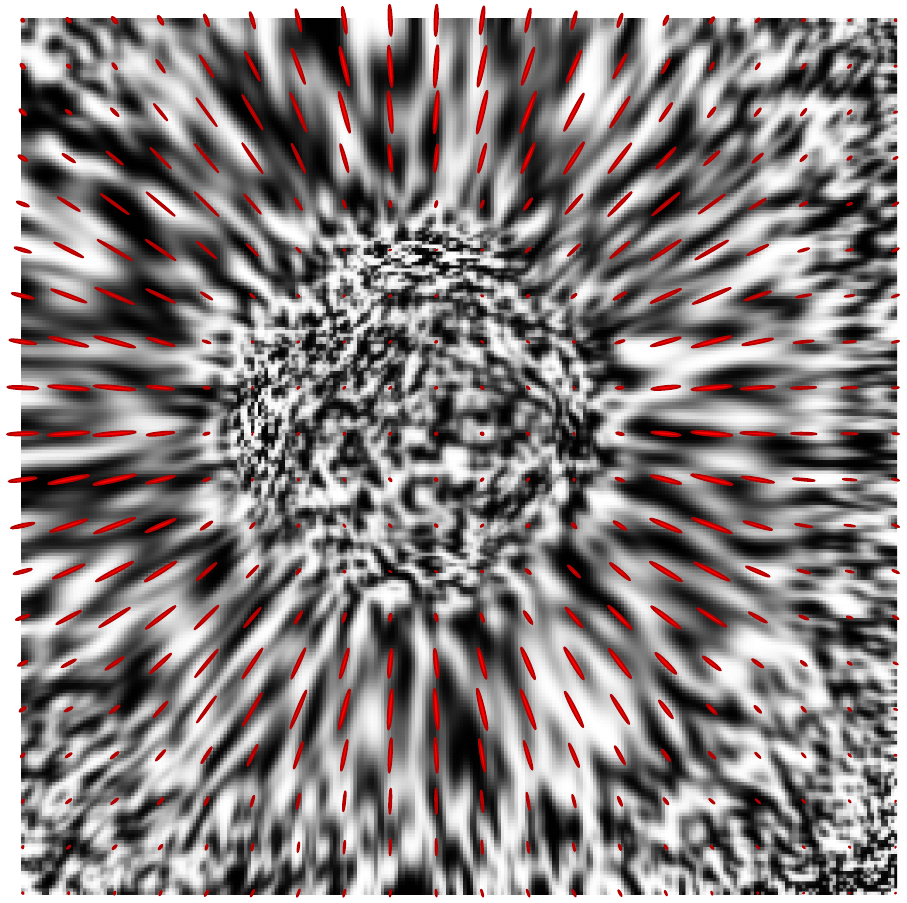}}
    \hspace*{3pt}
    \subfloat{\includegraphics[width = 0.11\textwidth, height = 0.11\textwidth]{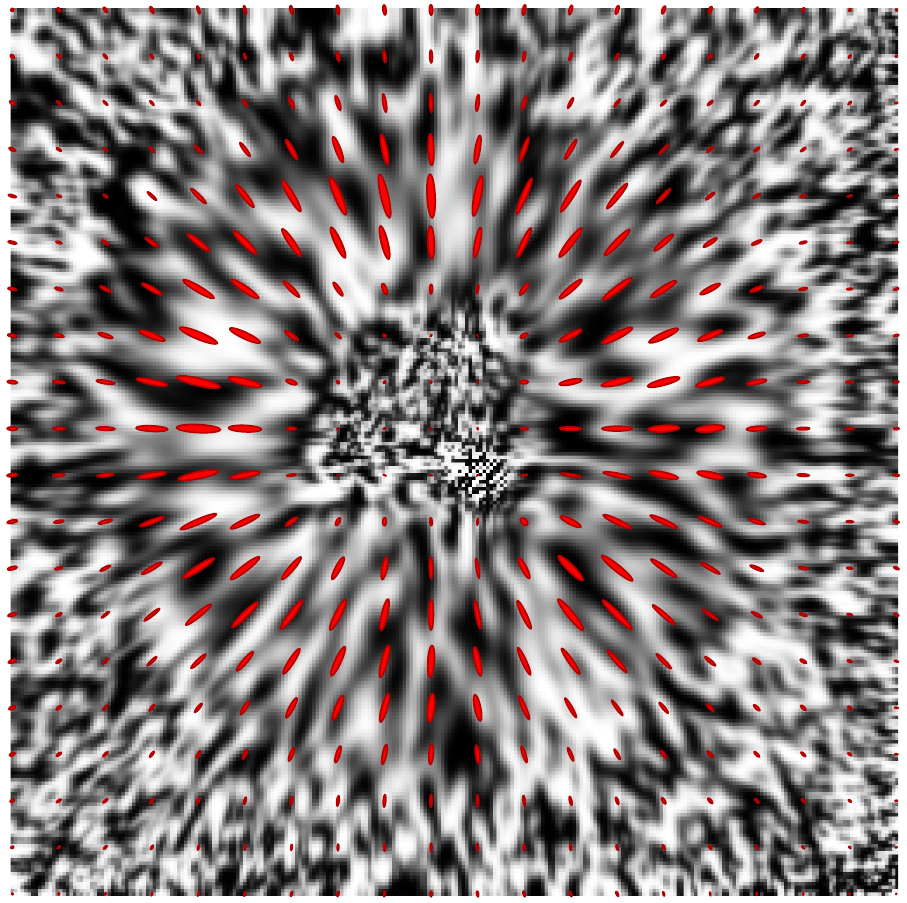}}
    \hspace*{3pt}
    \subfloat{\includegraphics[width = 0.11\textwidth, height = 0.11\textwidth]{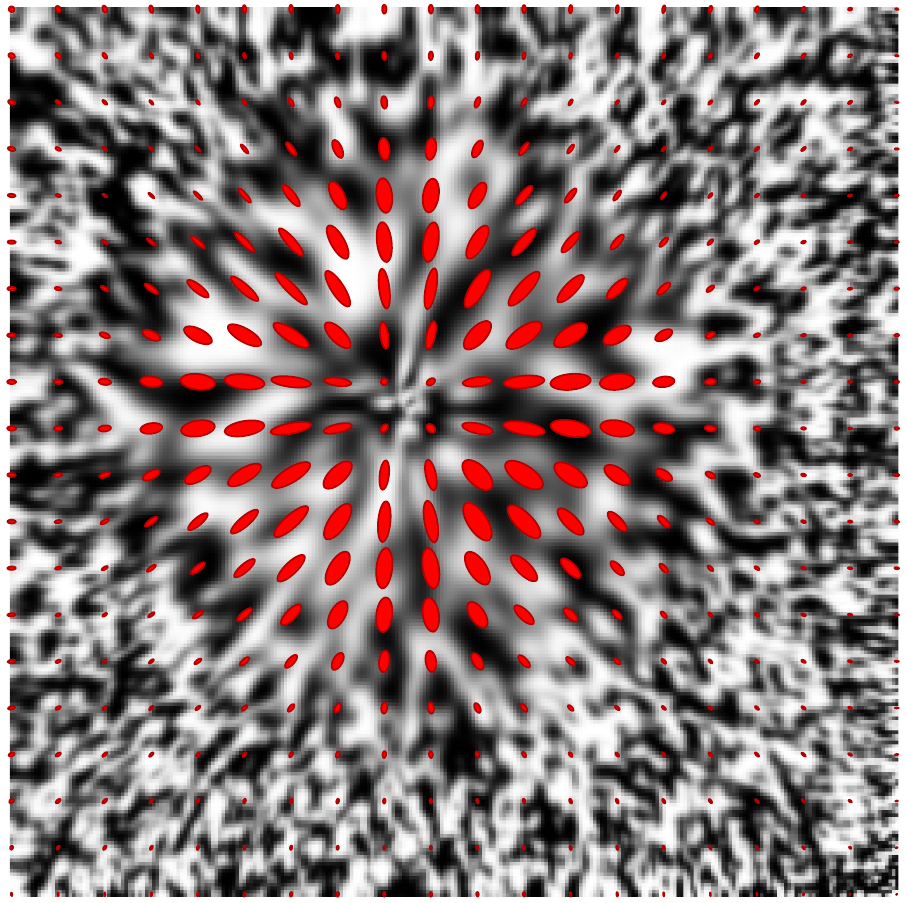}} 
    \hspace*{3pt}
    \subfloat{\includegraphics[width = 0.11\textwidth, height = 0.11\textwidth]{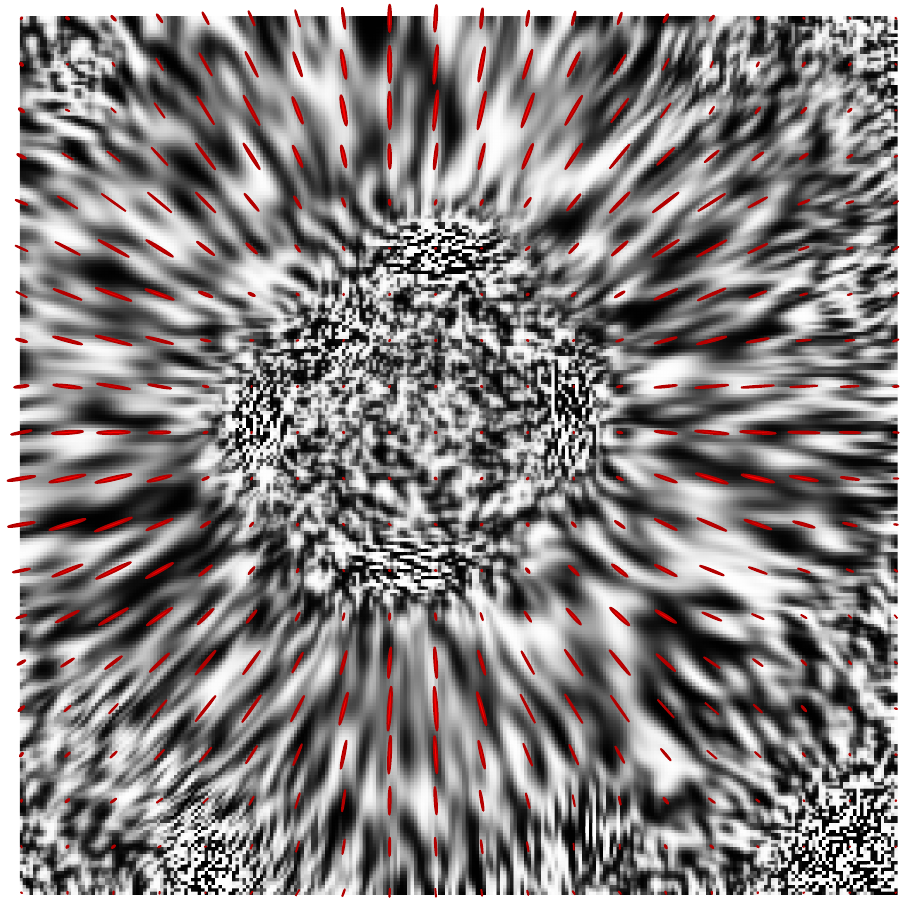}}
    \hspace*{3pt}
    \subfloat{\includegraphics[width = 0.11\textwidth, height = 0.11\textwidth]{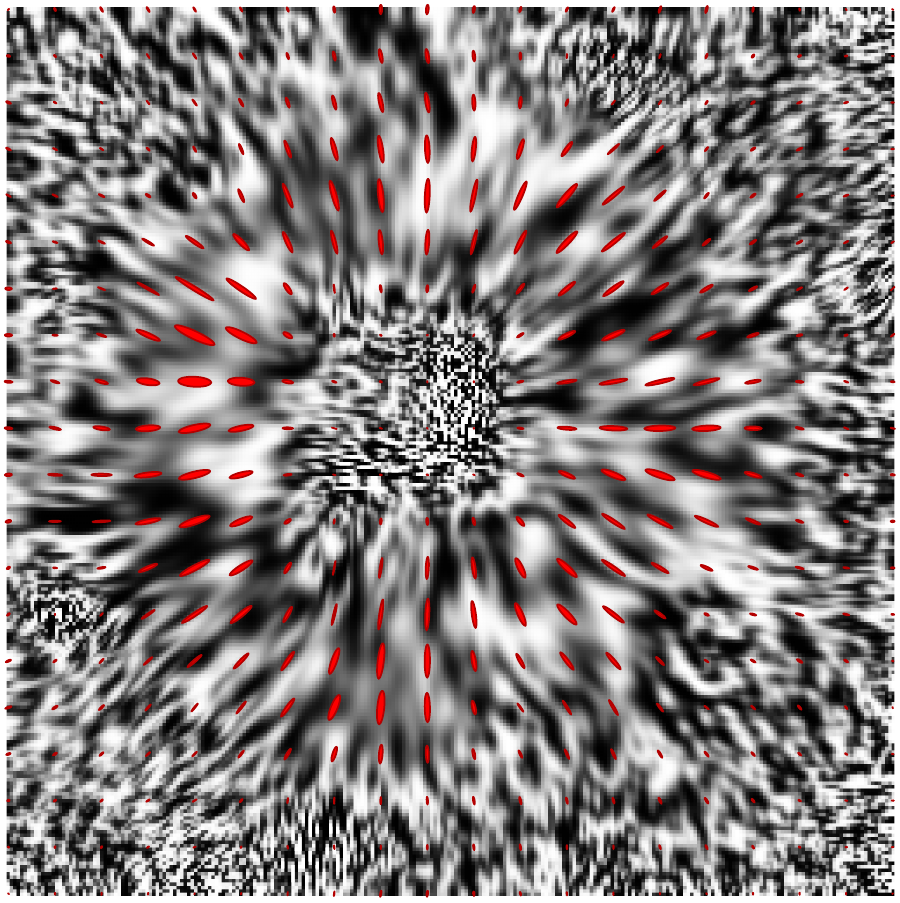}}
    \hspace*{3pt}
    \subfloat{\includegraphics[width = 0.11\textwidth, height = 0.11\textwidth]{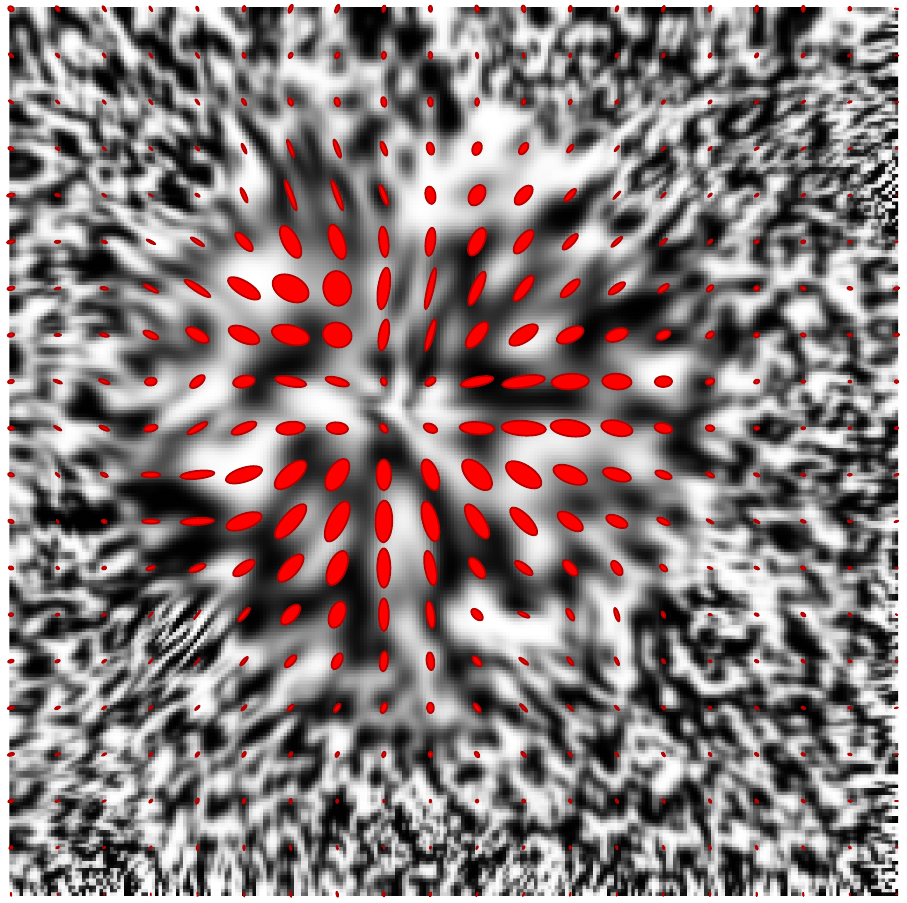}} \\[5pt]
    
    \stackunder[5pt]{\includegraphics[width=0.11\textwidth, height = 0.11\textwidth]{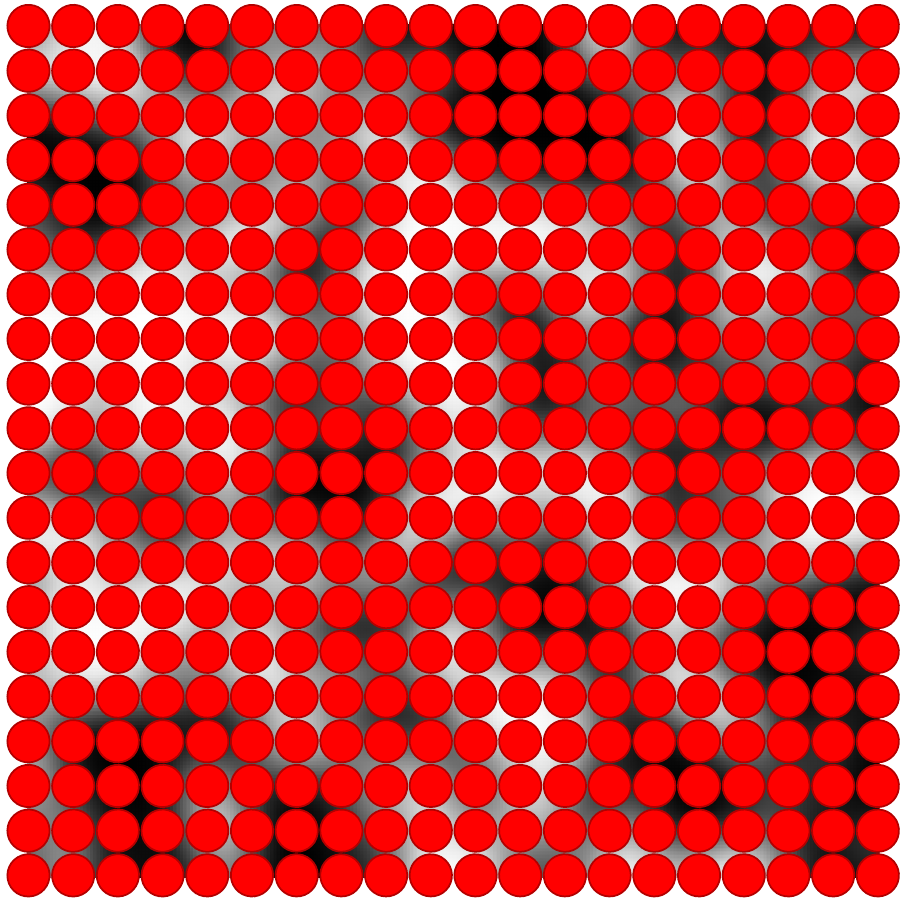}}{\footnotesize Input-1}
    \hspace*{0.7pt}
    \stackunder[5pt]{\includegraphics[width=0.11\textwidth, height = 0.11\textwidth]{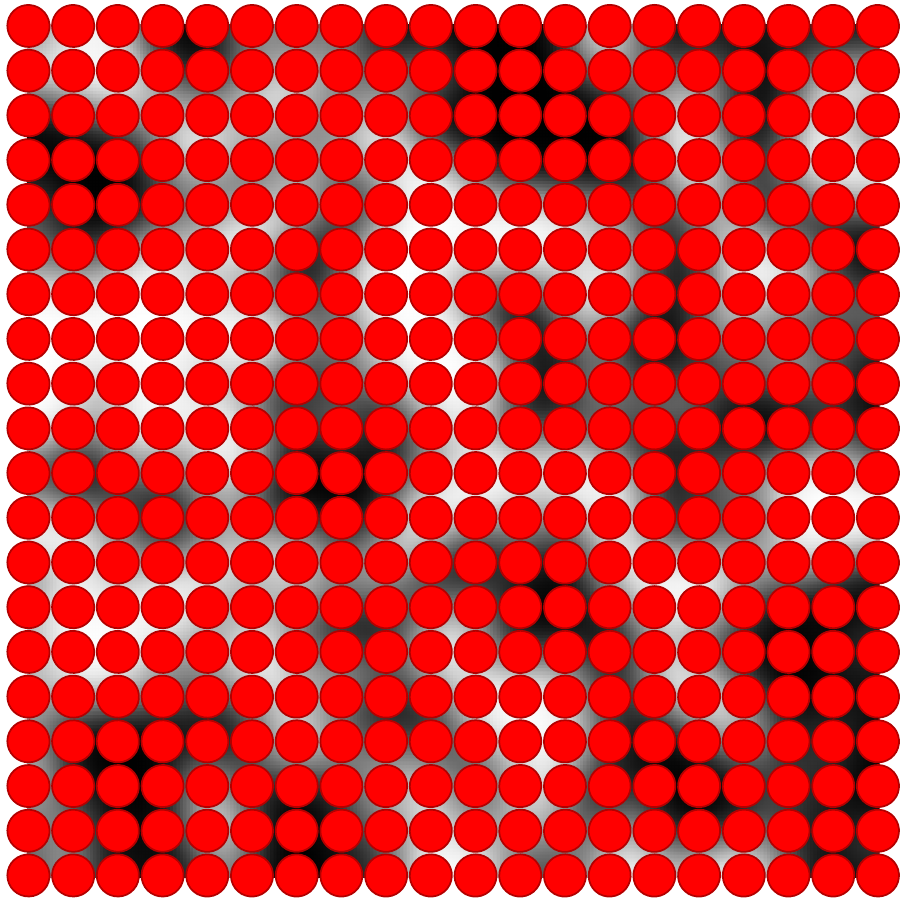}}{\footnotesize Input-5}
    \hspace*{0.7pt}
    \stackunder[5pt]{\includegraphics[width=0.11\textwidth, height = 0.11\textwidth]{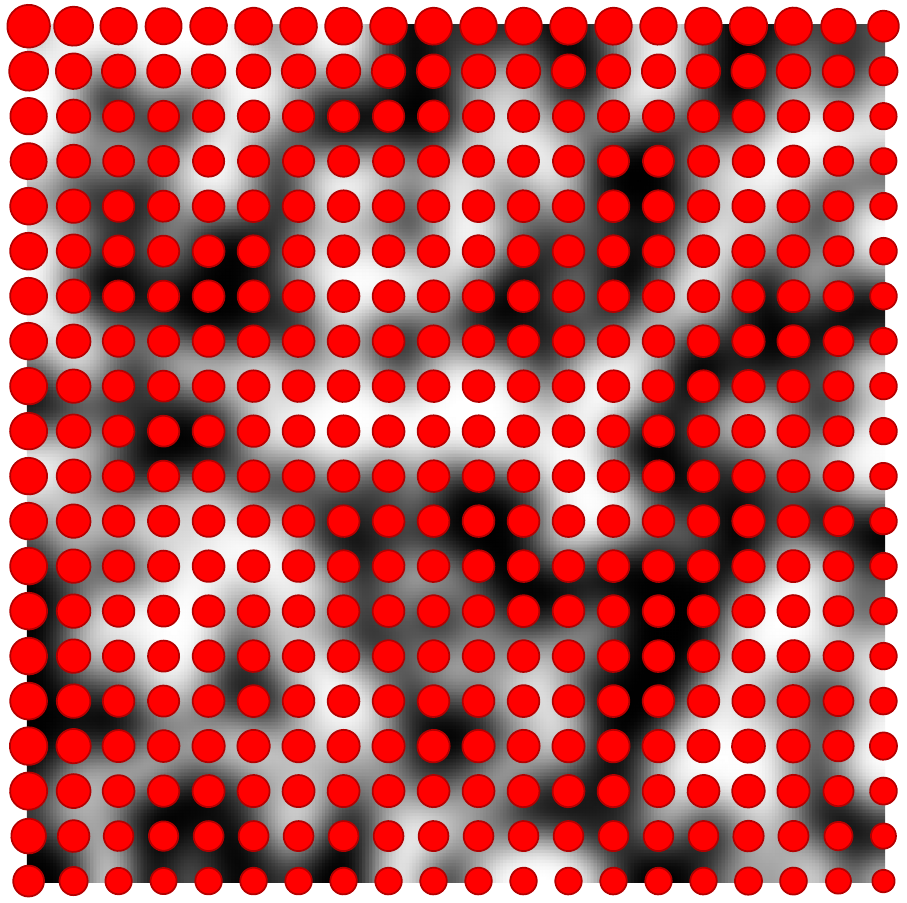}}{\footnotesize QOT-2}
    \hspace*{0.7pt}
    \stackunder[5pt]{\includegraphics[width=0.11\textwidth, height = 0.11\textwidth]{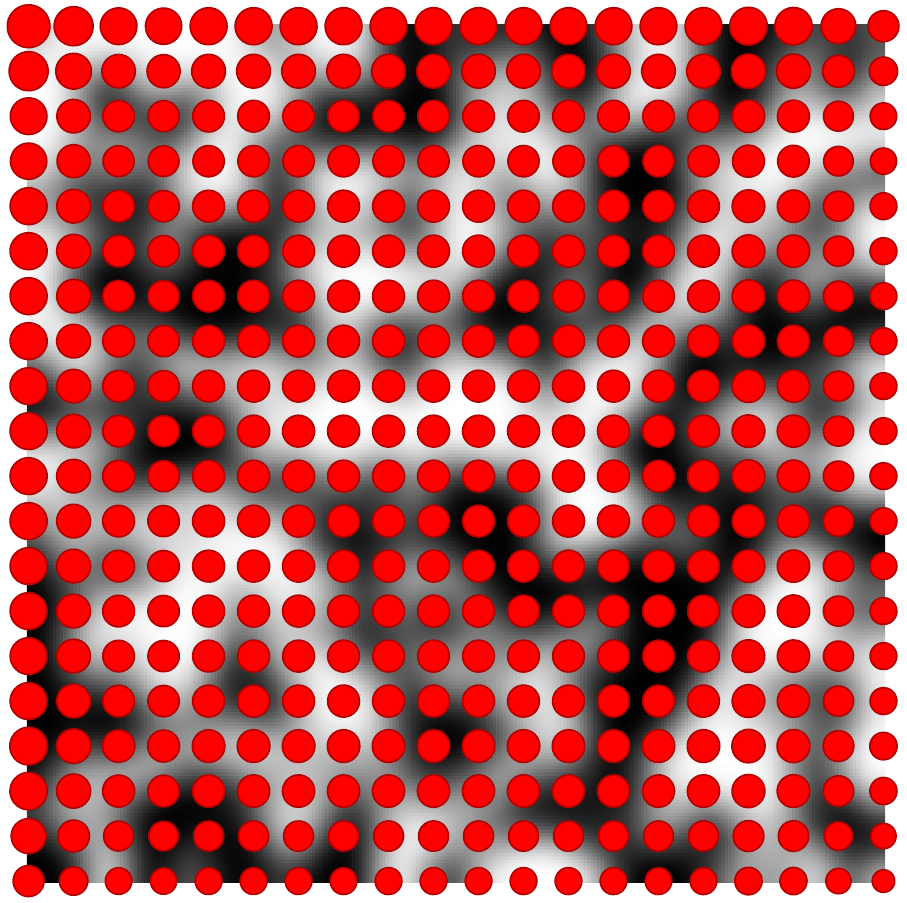}}{\footnotesize QOT-3}
    \hspace*{0.7pt}
    \stackunder[5pt]{\includegraphics[width=0.11\textwidth, height = 0.11\textwidth]{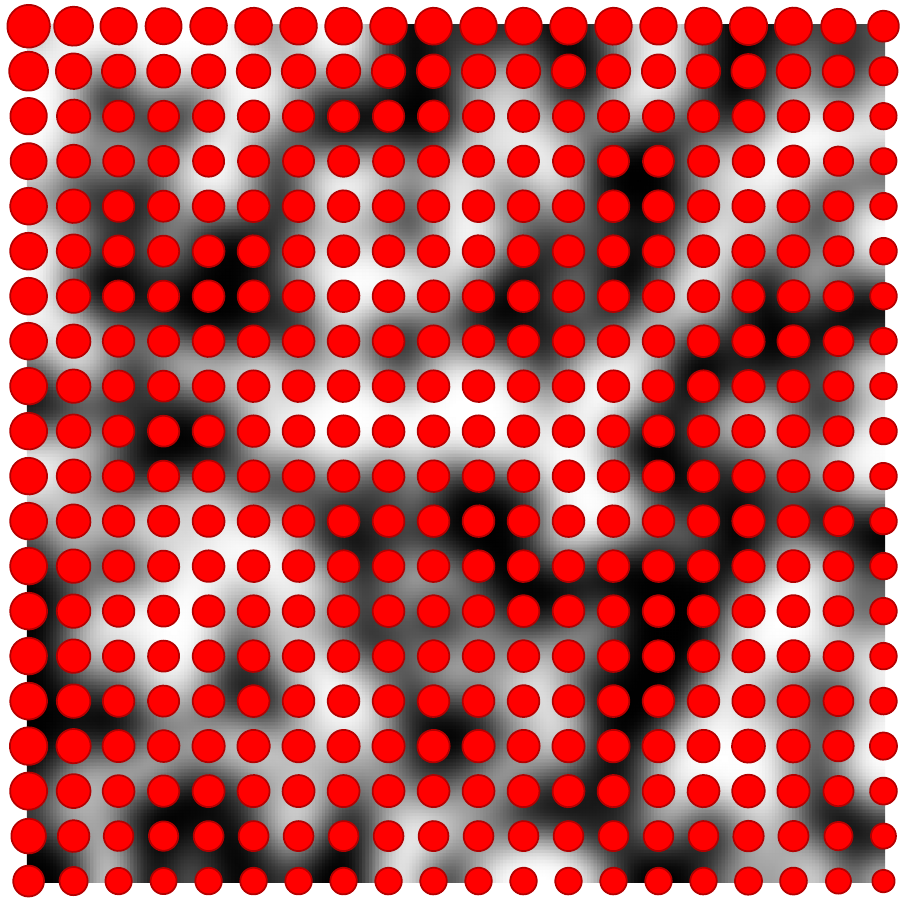}}{\footnotesize QOT-4}
    \hspace*{0.7pt}
    \stackunder[5pt]{\includegraphics[width=0.11\textwidth, height = 0.11\textwidth]{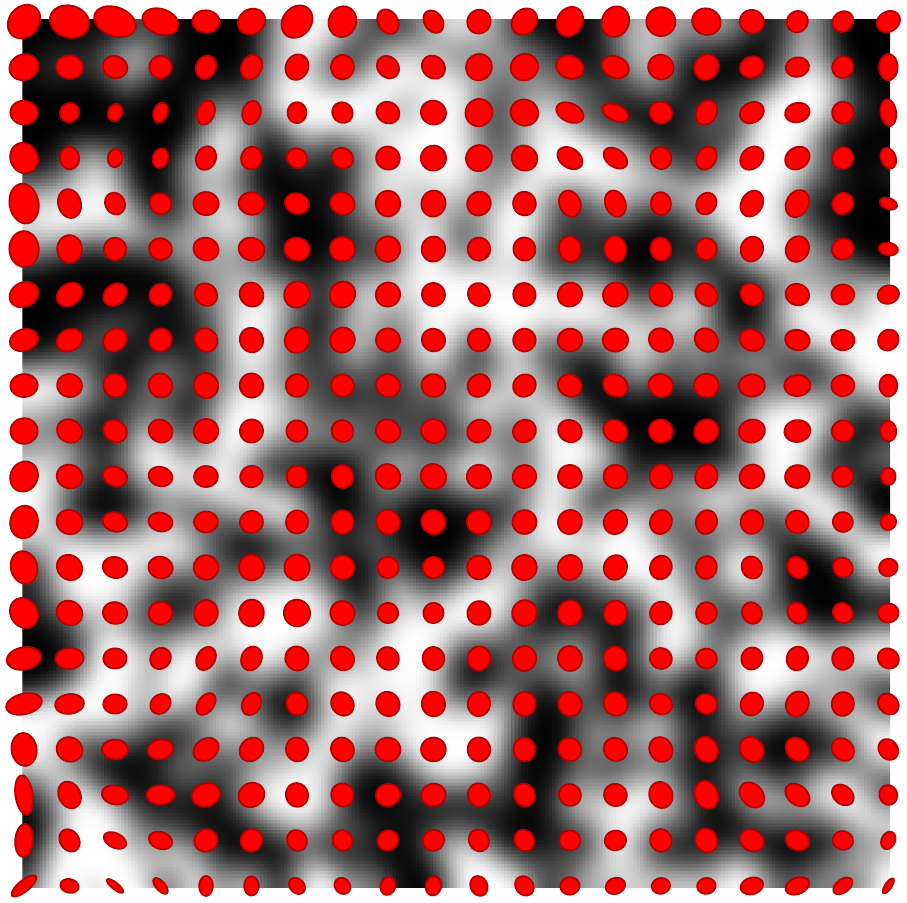}}{\footnotesize {\algname}-2} 
    \hspace*{0.7pt}
    \stackunder[5pt]{\includegraphics[width=0.11\textwidth, height = 0.11\textwidth]{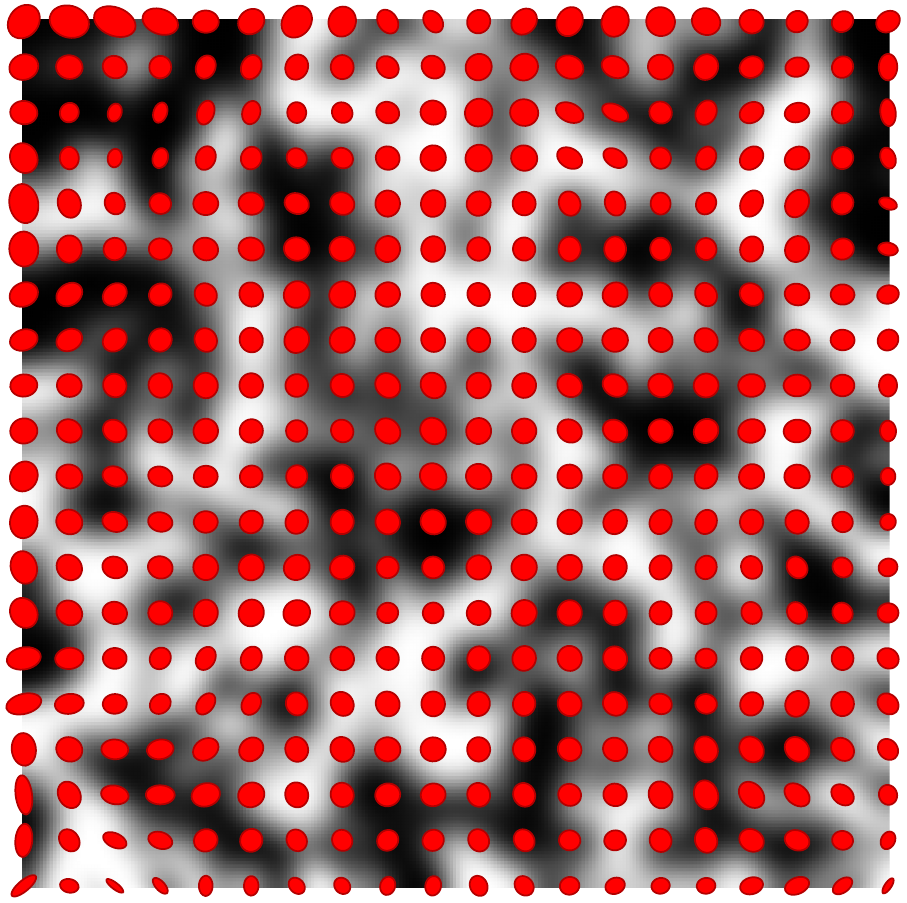}}{\footnotesize {\algname}-3} 
    \hspace*{0.7pt}
    \stackunder[5pt]{\includegraphics[width=0.11\textwidth, height = 0.11\textwidth]{Figures/2d-smooth-rand/BlockMOT-interpol-ellispses-3}}{\footnotesize {\algname}-4} 
    \caption{\small 2-$d$ tensor fields mass interpolation.} 
    \label{tensor_field_mass_2d_appendix}
\end{figure*}

\subsection{Tensor field Wasserstein barycenter}

We first show how both linear interpolation and QOT solutions are not optimal. We initialize our Riemannian optimizers for $\bar{\bP}$ from the linear interpolation and (normalized) QOT. We also include uniform initialization as a benchmark.

We compare the objective value of $\sum_\ell \omega_\ell {\rm MW}_\epsilon(\bar{\bP}, \bP^\ell)$ against the optimal objective value obtained from the CVX toolbox \cite{grant2014cvx}. This allows to compute the optimality gap.

In Figure \ref{obj_compare_barycenter}, we see that the optimality gap keeps reducing with iterations even after properly normalizing the barycenter from linear interpolation and (normalized) QOT. This shows that linear interpolation and (normalized) QOT solutions are not optimal. Also, the performance of RMOT with uniform initialization is competitive to that initialized with linear interpolation and (normalized) QOT, implying that RMOT is a competitive solver in itself and obtains better solutions.

Additionally, we show the barycenter results for $n = 16$ along with convergence of RMOT in Figure \ref{2d_barycenter_n4_main} and \ref{2d_barycenter_n4}. \revision{From Figure \ref{2d_barycenter_n4_main}, we see visually no difference in the solutions obtained by QOT and RMOT, which suggests the solution by QOT (with normalization) is close to optimal. This observation is further validated in Figure \ref{2d_barycenter_n4} where we see the objective value is already quite small when initialized from the QOT solution.}

\begin{figure*}[t]
    \centering
    \subfloat[$t =0.25$]{\includegraphics[width = 0.24\textwidth, height = 0.2\textwidth]{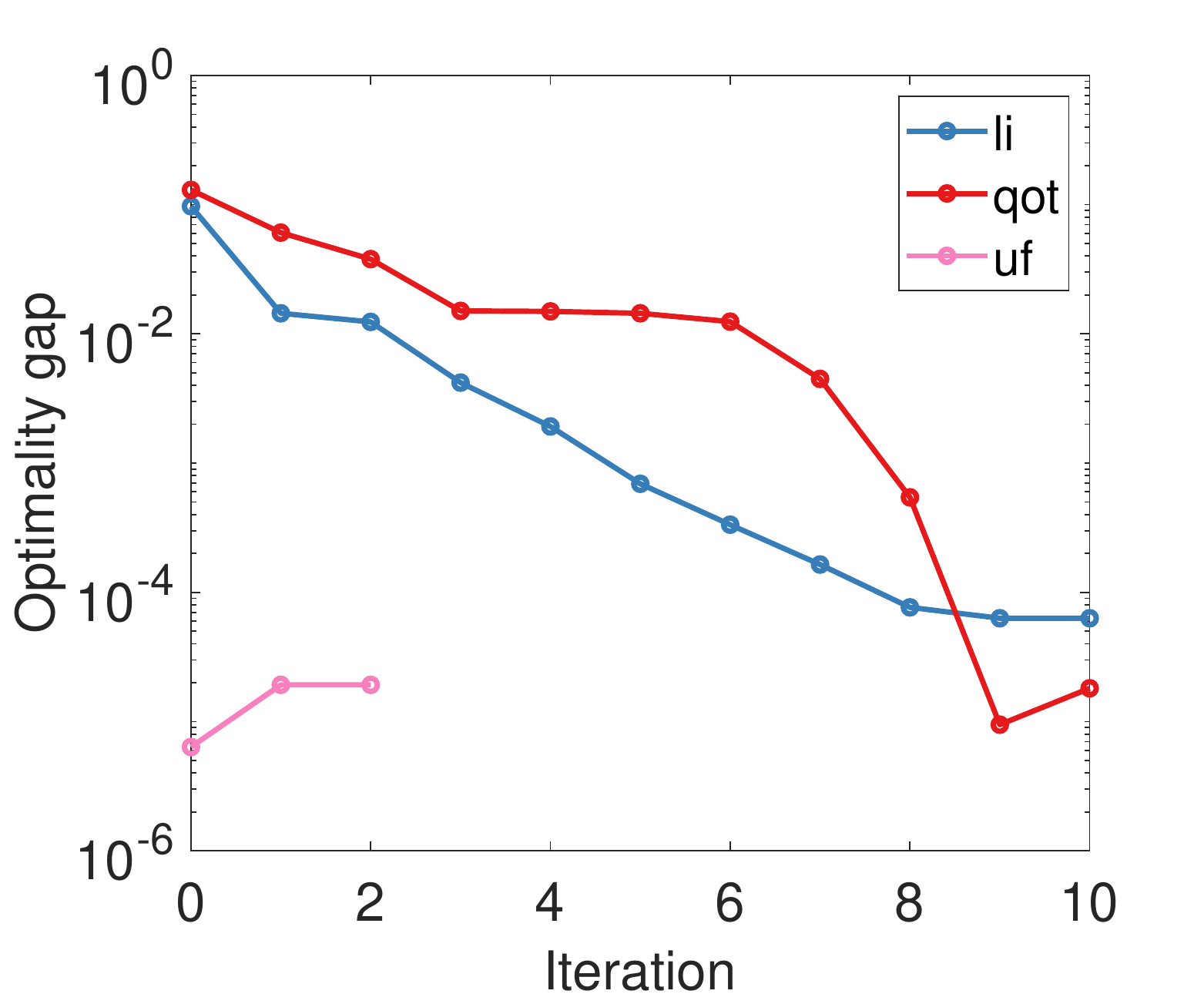}} 
    \hspace*{5pt}
    \subfloat[$t=0.5$]{\includegraphics[width = 0.24\textwidth, height = 0.2\textwidth]{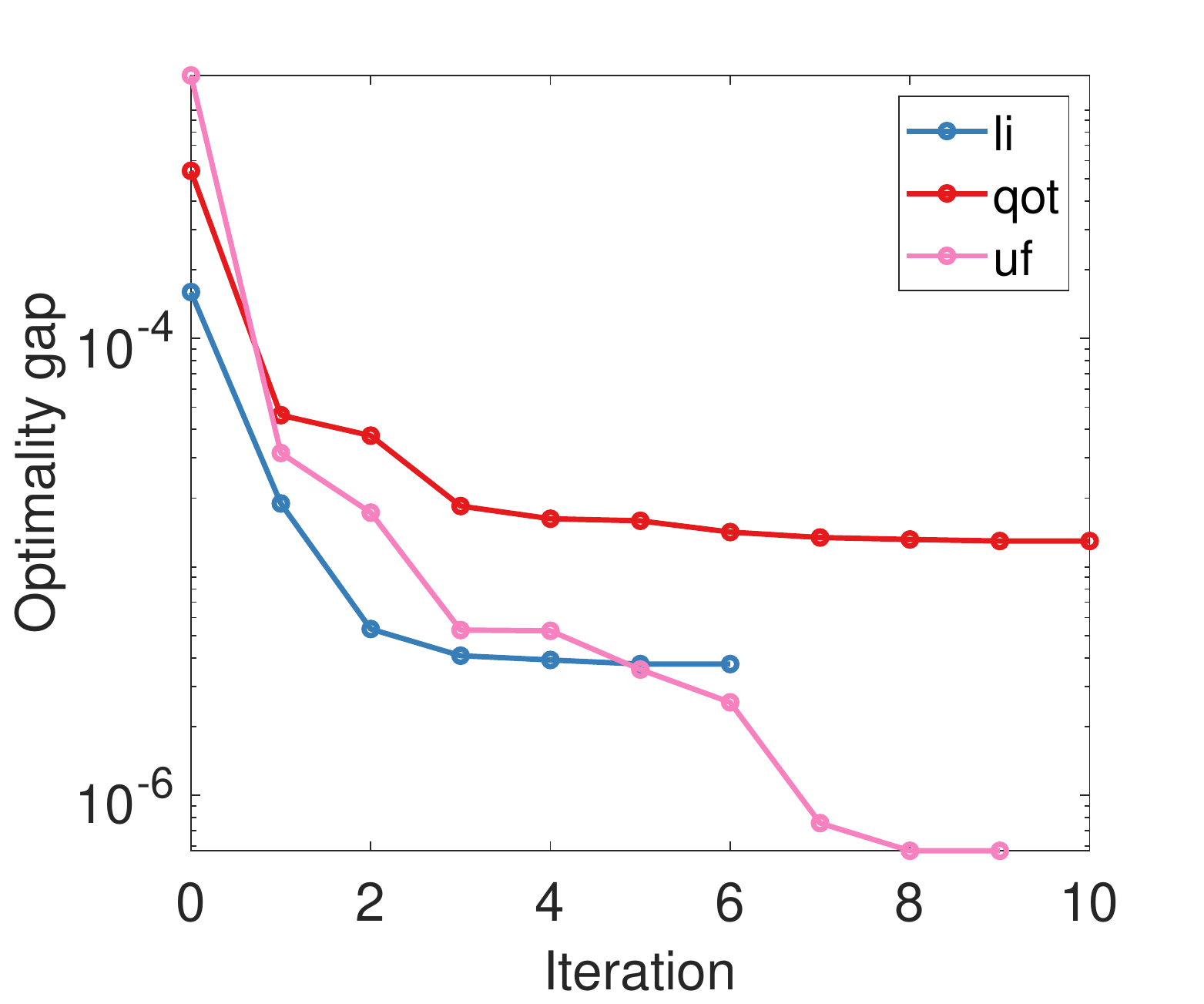}}
    \hspace*{5pt}
    \subfloat[$t = 0.75$]{\includegraphics[width = 0.24\textwidth, height = 0.2\textwidth]{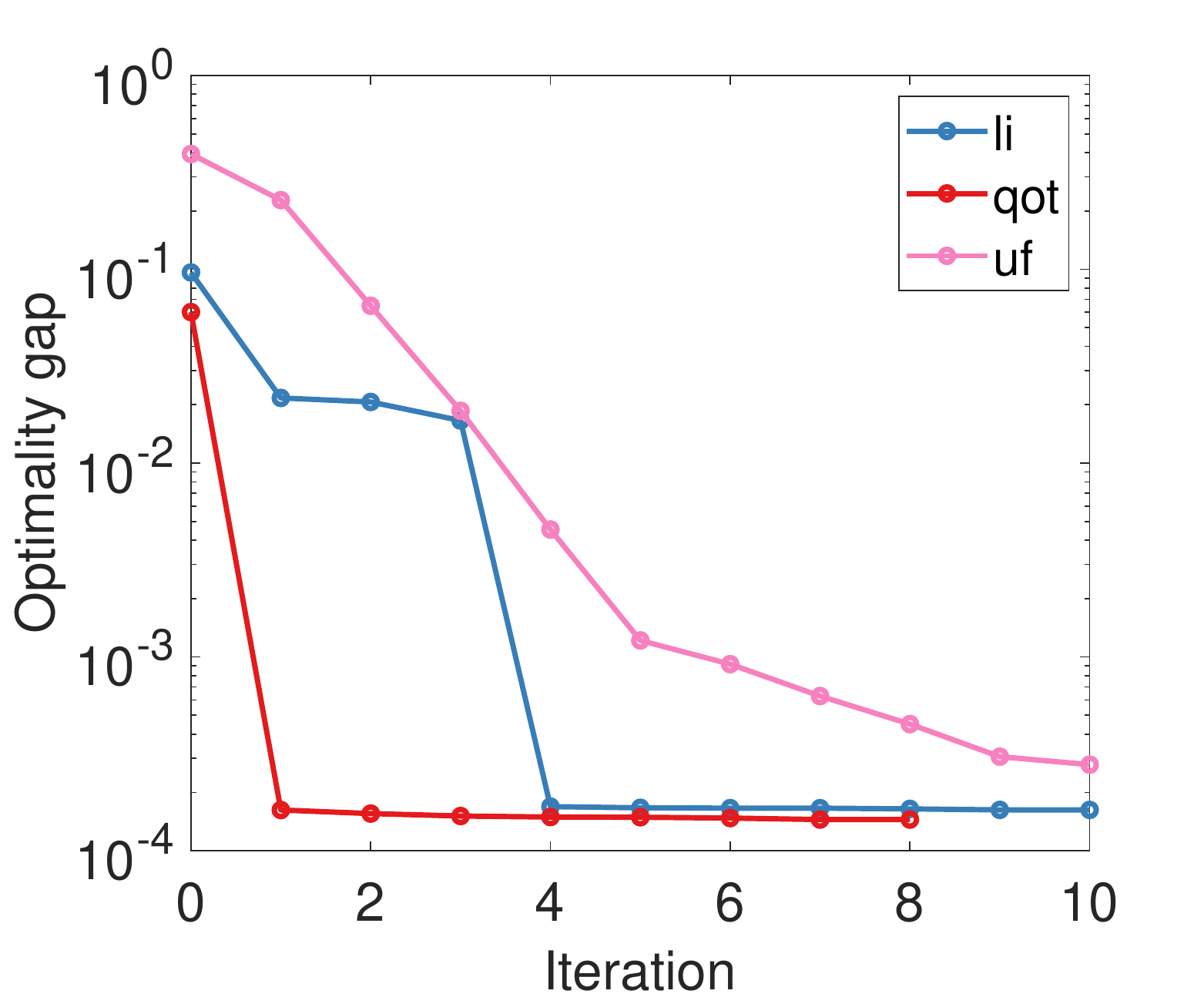}}
    \caption{\small Convergence of ($n=4$) barycenter update initialized from linear interpolation (li), QOT (qot), and uniform identity (uf). Irrespective of the initialization, RMOT continue to achieve better optimality gap (to the CVX optimal solution) with iterations. As the initial optimality gap is high for all the cases, it shows that the linear interpolation (li) and QOT (qot) solutions are not optimal.} 
    \label{obj_compare_barycenter}
\end{figure*}

\begin{figure}[!th]
\centering
  \subfloat{
	\begin{minipage}[c][0.15\textwidth]{
	   0.15\textwidth}
	   \centering
	   \includegraphics[width=0.5\textwidth]{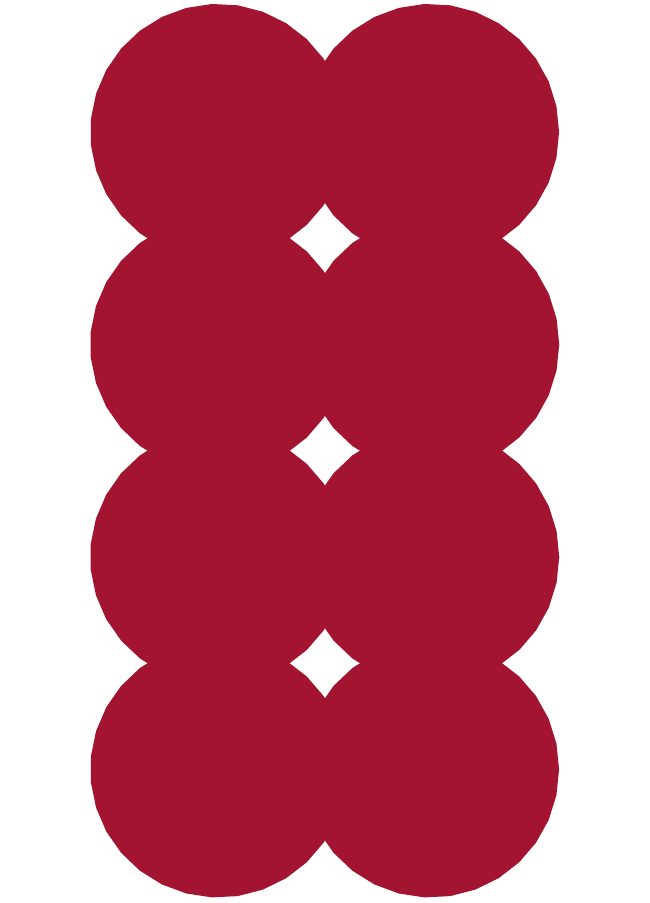}
	\end{minipage}}
  \subfloat{
	\begin{minipage}[c][0.15\textwidth]{
	   0.15\textwidth}
	   \centering
	   \includegraphics[width=0.5\textwidth]{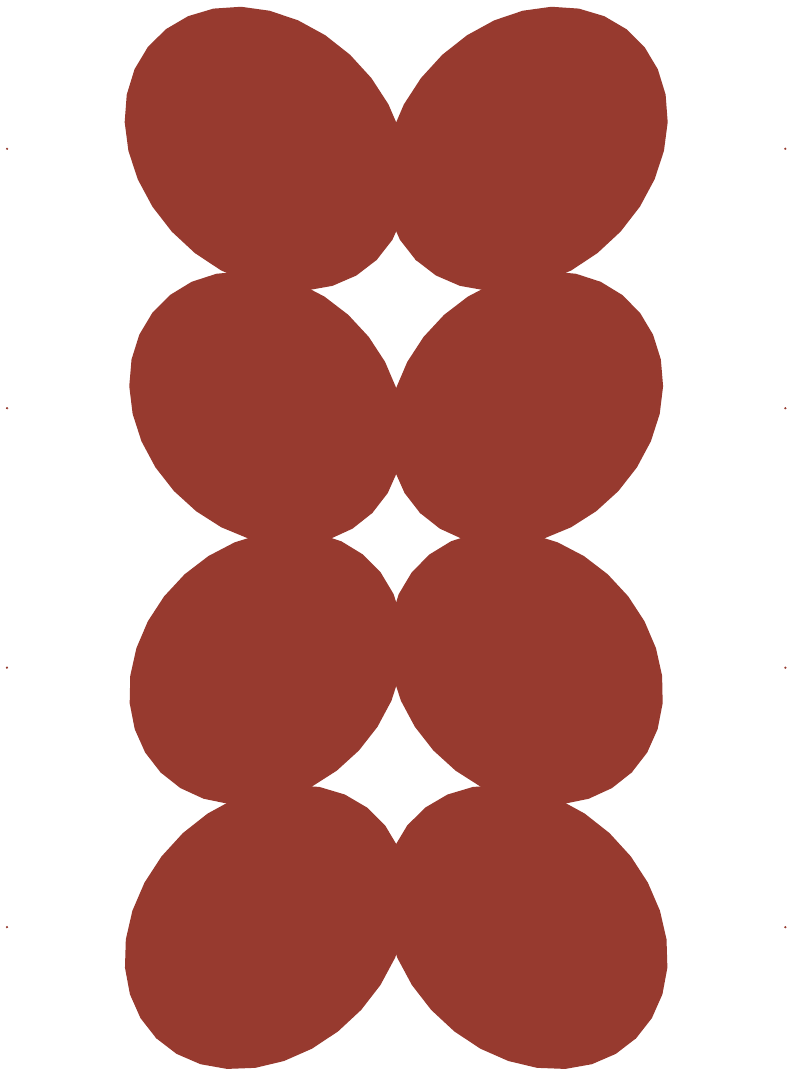} \\[8pt]
	   \includegraphics[width=0.5\textwidth]{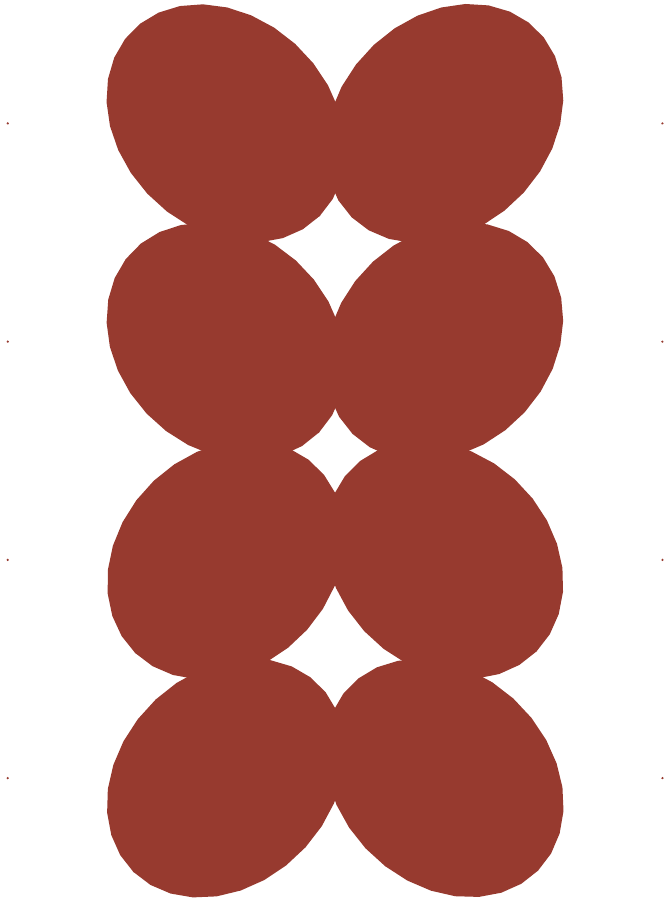}
	\end{minipage}}
	\subfloat{
	\begin{minipage}[c][0.15\textwidth]{
	   0.15\textwidth}
	   \centering
	   \includegraphics[width=0.5\textwidth]{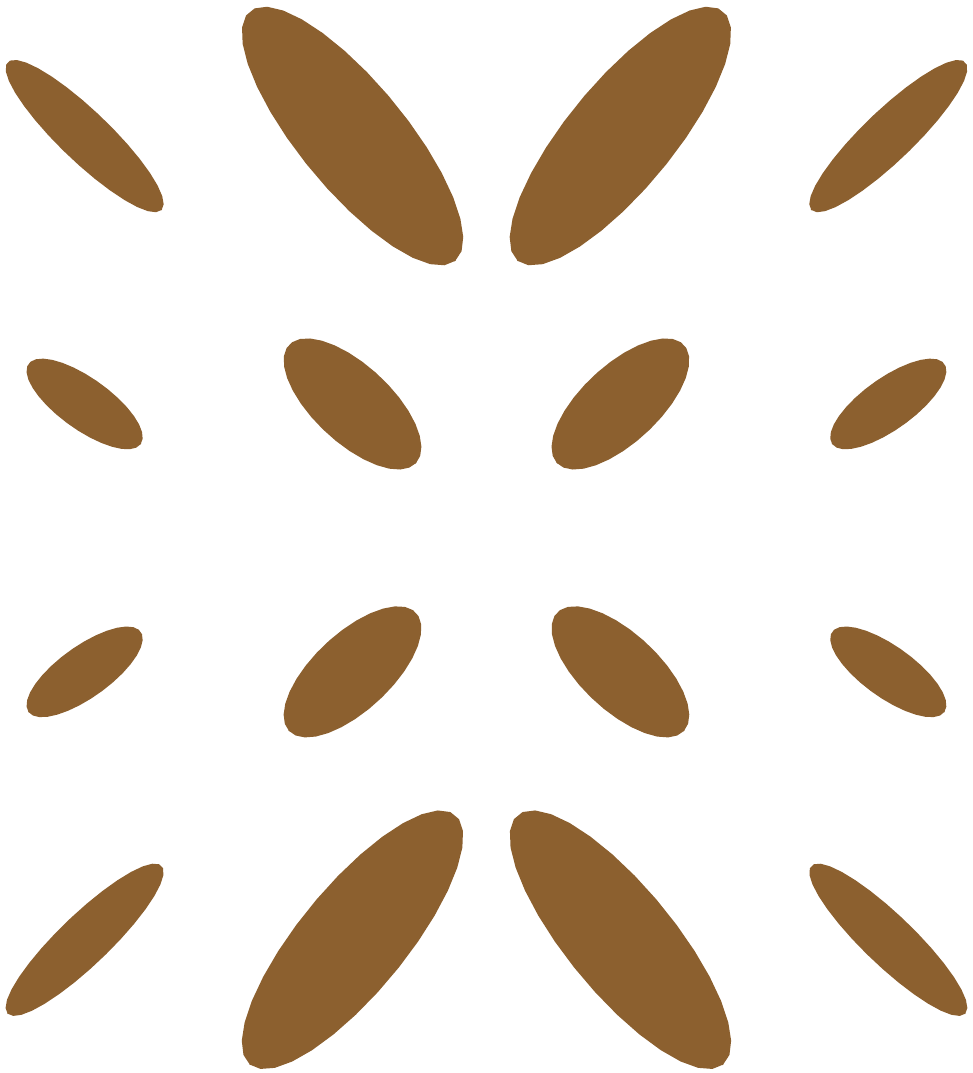} \\[8pt]
	   \includegraphics[width=0.5\textwidth]{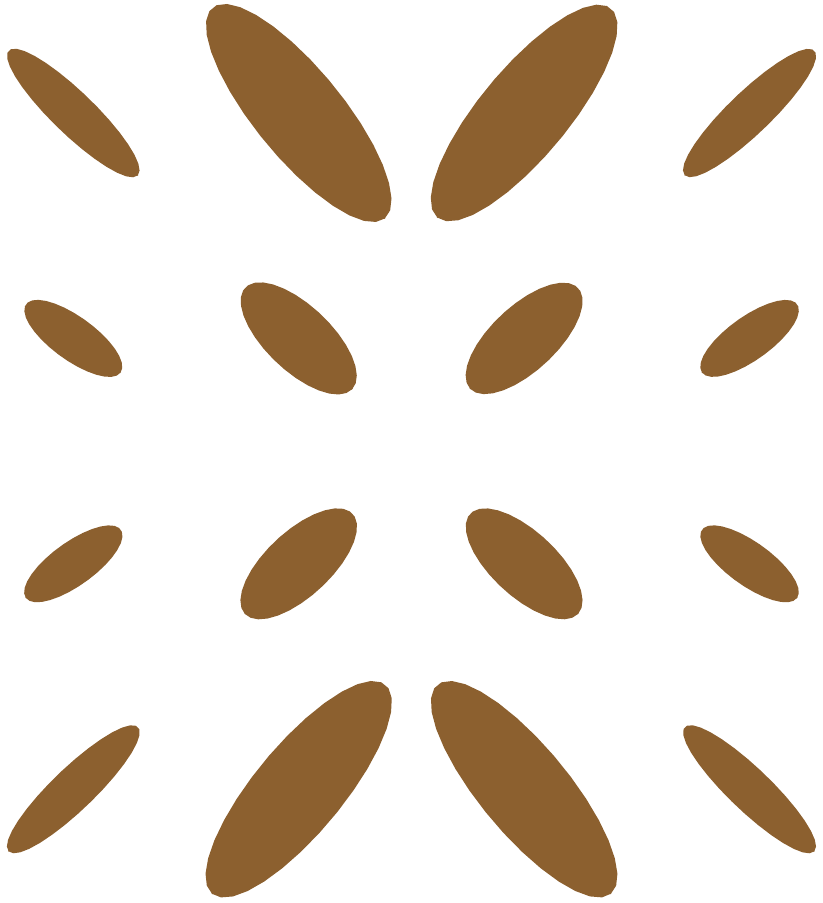}
	\end{minipage}}
	\subfloat{
	\begin{minipage}[c][0.15\textwidth]{
	   0.15\textwidth}
	   \centering
	   \includegraphics[width=0.5\textwidth]{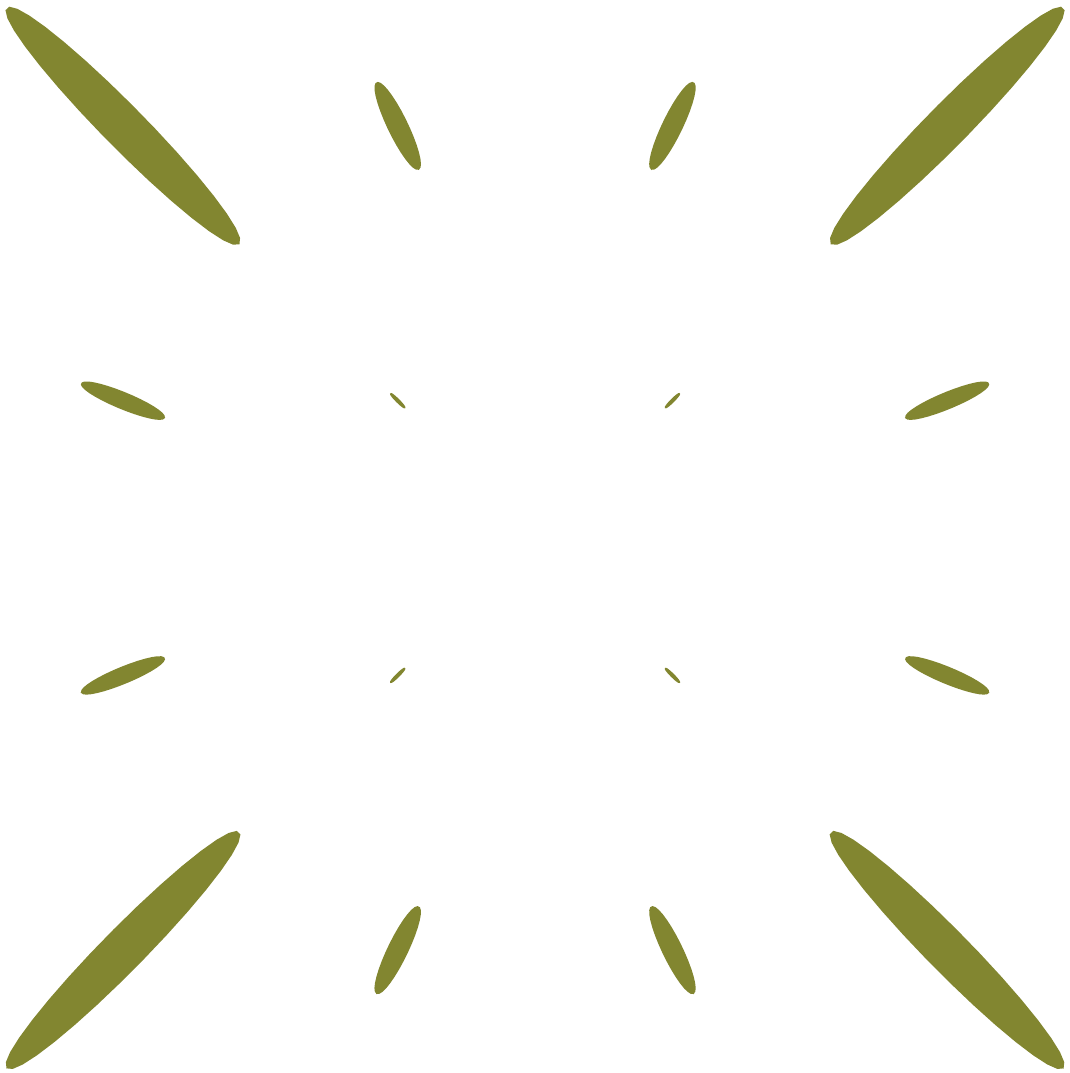} \\[8pt]
	   \includegraphics[width=0.5\textwidth]{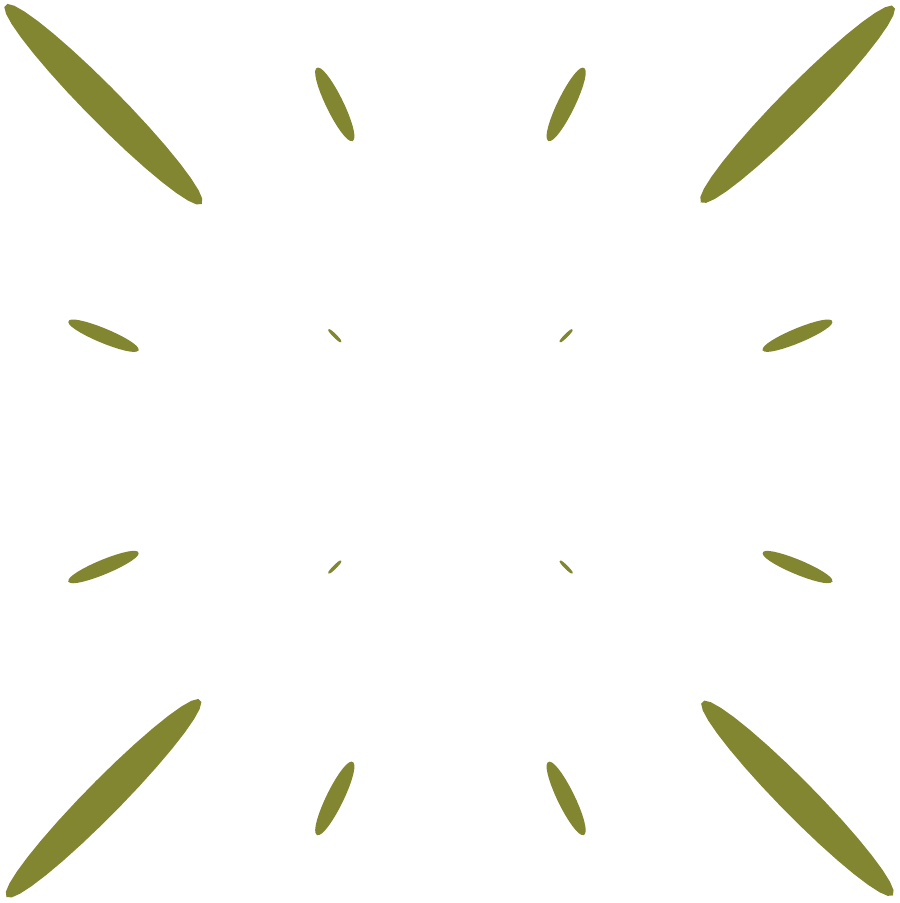}
	\end{minipage}}
  \subfloat{
	\begin{minipage}[c][0.15\textwidth]{
	   0.15\textwidth}
	   \centering
	   \includegraphics[width=0.5\textwidth]{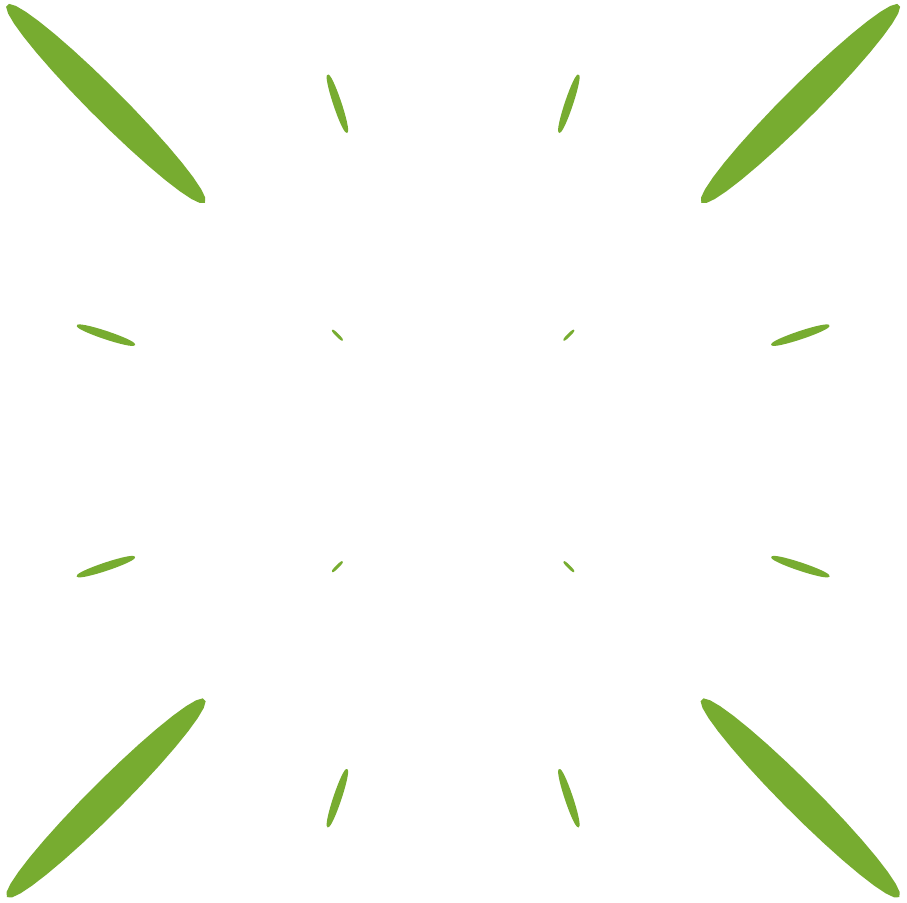}
	\end{minipage}}
	\vspace*{10pt}
\caption{\small \revision{Tensor field Wasserstein barycenter ($n=16$). The barycetners are shown in (a) to (e). From left to right $t = 0$ (input), $t= 0.25, 0.5, 0.75$ (barycenters), $t=1$ (input). The top row is QOT and the bottom is {\algname}.}}
\label{2d_barycenter_n4_main}
\end{figure}

\begin{figure*}[h]
    \centering
    \subfloat[Convergence $t=0.25$]{\includegraphics[width = 0.24\textwidth, height = 0.2\textwidth]{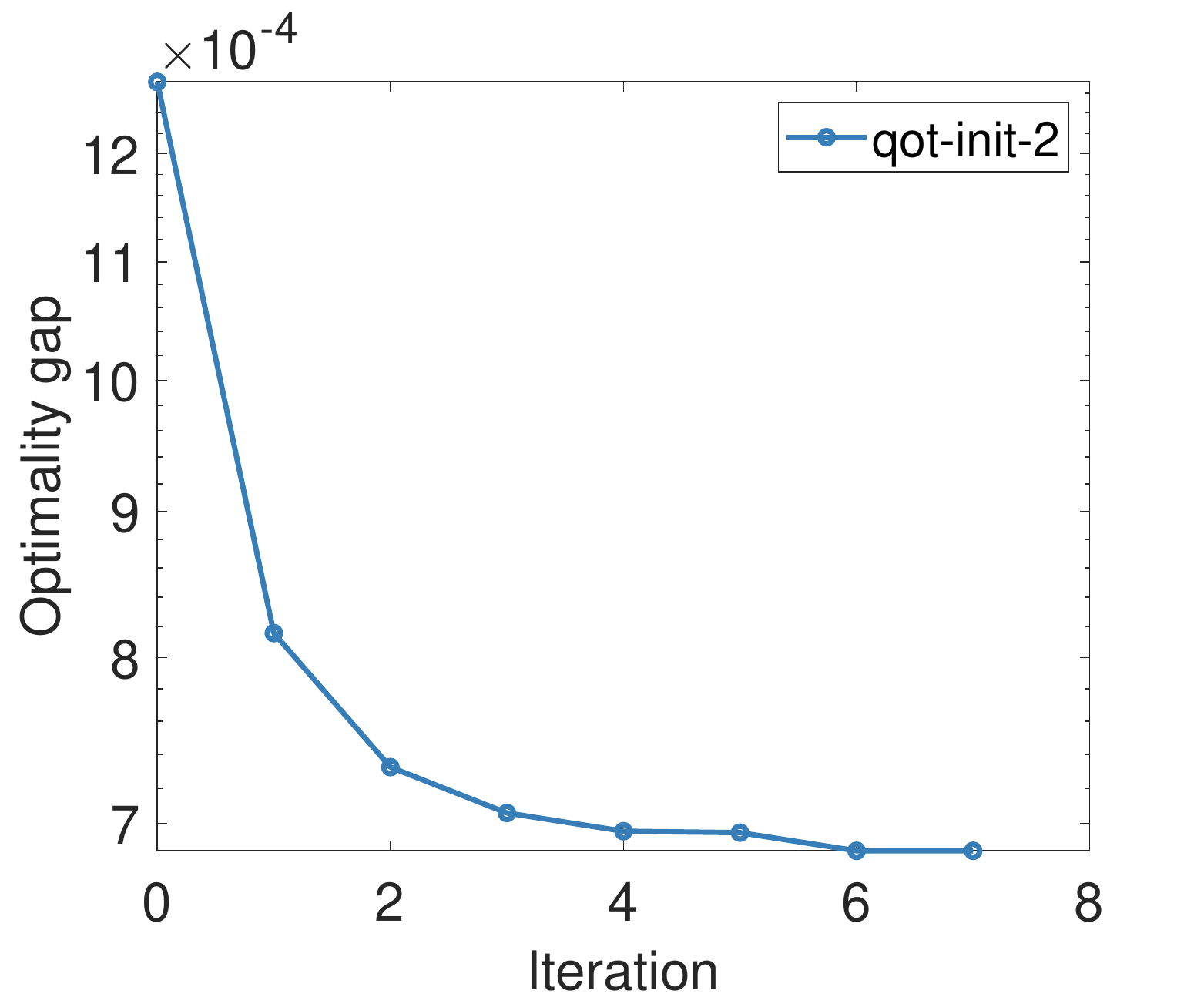}}
    \hspace*{5pt}
    \subfloat[Convergence $t=0.5$]{\includegraphics[width = 0.24\textwidth, height = 0.2\textwidth]{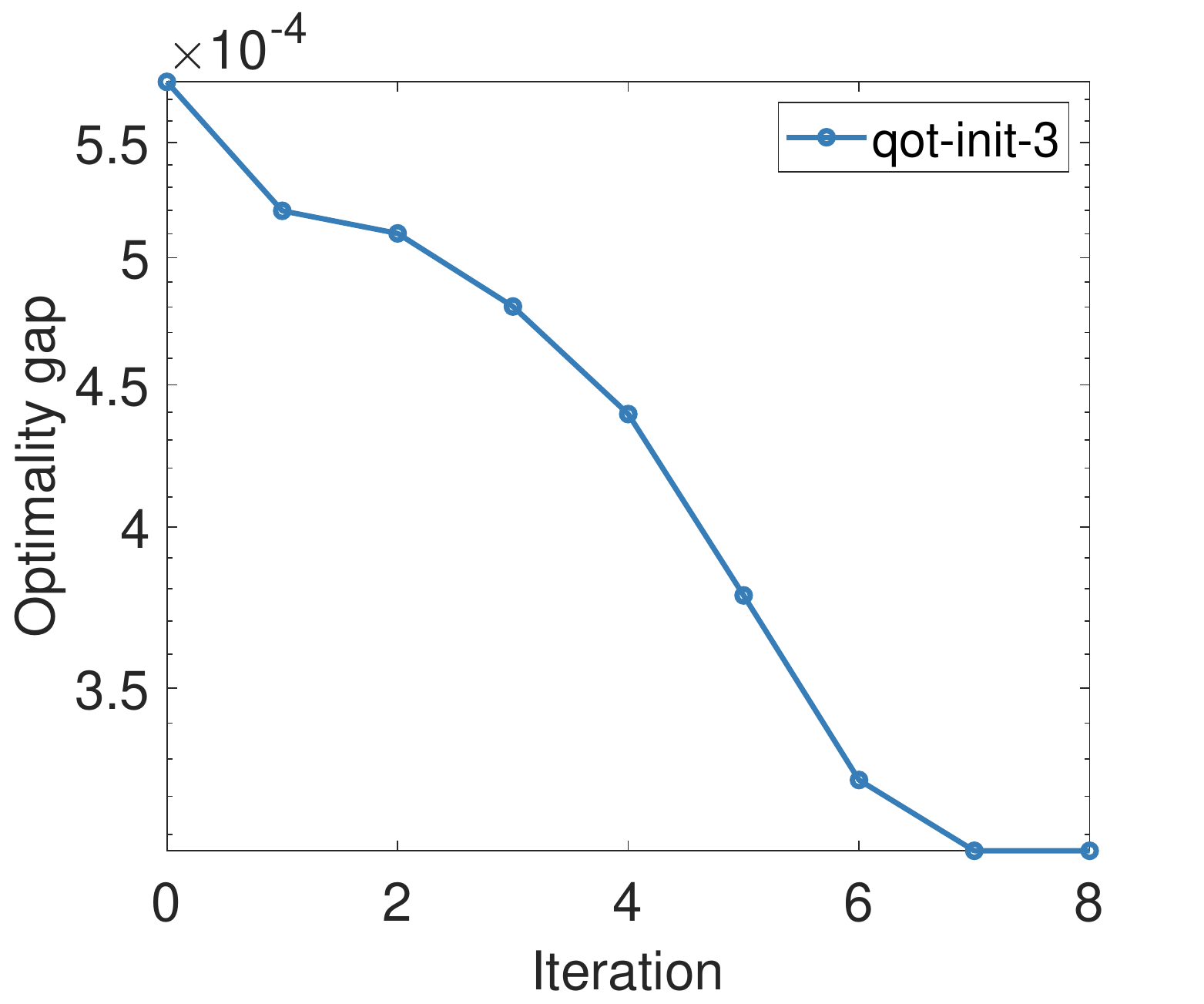}}
    \hspace*{5pt}
    \subfloat[Convergence $t = 0.75$]{\includegraphics[width = 0.24\textwidth, height = 0.2\textwidth]{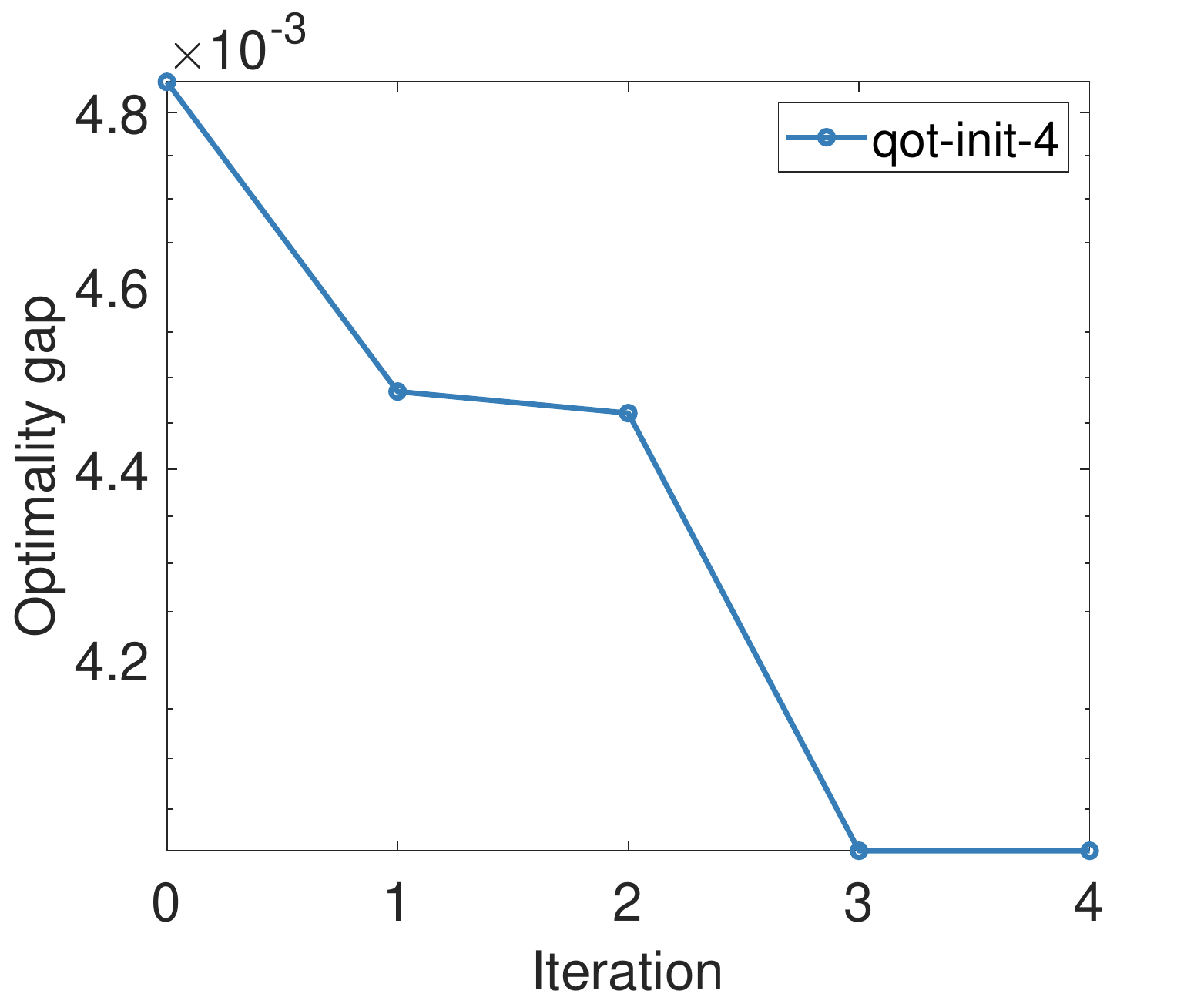}}
    \caption{\small The convergence plots for RMOT initialized from (normalized) QOT solution. For this setting, the solution from QOT is close to optima.} 
    \label{2d_barycenter_n4}
\end{figure*}

\revision{
\subsection{Additional experiments on domain adaptation}

Here, we perform the experiments of domain adaptation on more challenging tasks, including video based face recognition with YouTube Celebrities (YTC) dataset \cite{kim2008face} and texture classification via Dynamic Texture (DynTex) \cite{ghanem2010maximum} dataset, where covariance representation learning has shown great promise \cite{huang2015log,harandi2014manifold}. 


{\bf Datasets and experimental setup.}
YTC \cite{kim2008face} comprises of 1910 low-resolution videos of 47 celebrities from YouTube. Here we only select $9$ persons with video size larger than $15$.
Following standard preprocessing techniques \cite{huang2015log}, we first crop the frames of each video to the detected face regions and resize into $10 \times 10$ intensity images. Then we construct the covariance representation for each video, which is a $100 \times 100$ SPD matrix. We then apply the geometry-aware principal component analysis for SPD manifold \cite{horev2016geometry} via the Bures-Wasserstein Riemannian metric \cite{bhatia2019bures,han2021riemannian,han2021generalized} to reduce the dimensionality to $d = 5$. Finally, we obtain a collection of $194$ SPD covariance matrices of size $5 \times 5$, each representing one video. Given the relatively small sample size, we select $8$ videos per class as the test data and the rest are treated as the training data. Different to the settings in Section \ref{domain_adapt_sect}, we skew the selected class by sub-selecting a ratio $\alpha$ of the samples in the training set, where $\alpha = 0.2, 0.4, 0.6, 0.8, 1.0$. This is again due to the small data size. To further test the robustness of the algorithms, we then randomly truncate the training size to $100$. This results in a training set of $100$ videos against a test set of $72$ videos. Such randomization of process is repeated $5$ times. 

DynTex \cite{ghanem2010maximum} collects video sequences of $36$ moving scenes, such as sea waves, fire, clouds. For our experiment, we choose $10$ classes, each with $20$ videos. The subsequent processing steps are the same as for YTC dataset. 

Finally, we also test on Cifar10 \cite{krizhevsky2009learning} under the same settings as in Section \ref{domain_adapt_sect} in the main text. However, because when $d = 5$, much information is lost for this complex dataset, we choose $d = 17$, which captures $70\%$ of the variance in the samples.

{\bf Results.}
The final results are shown in Figure \ref{domain_adaptation_figure_appendix} where we observe consistent good performance of the proposed RMOT compared to both sOT and SPDOT. This strengthens the findings that matrix-valued OT is able to explore more variations in the dataset compared to scalar-valued OT. 

}

\begin{figure}[!th]
    \centering
    \subfloat[YTC \label{ytc_fig}]{\includegraphics[width = 0.33\textwidth, height = 0.25\textwidth]{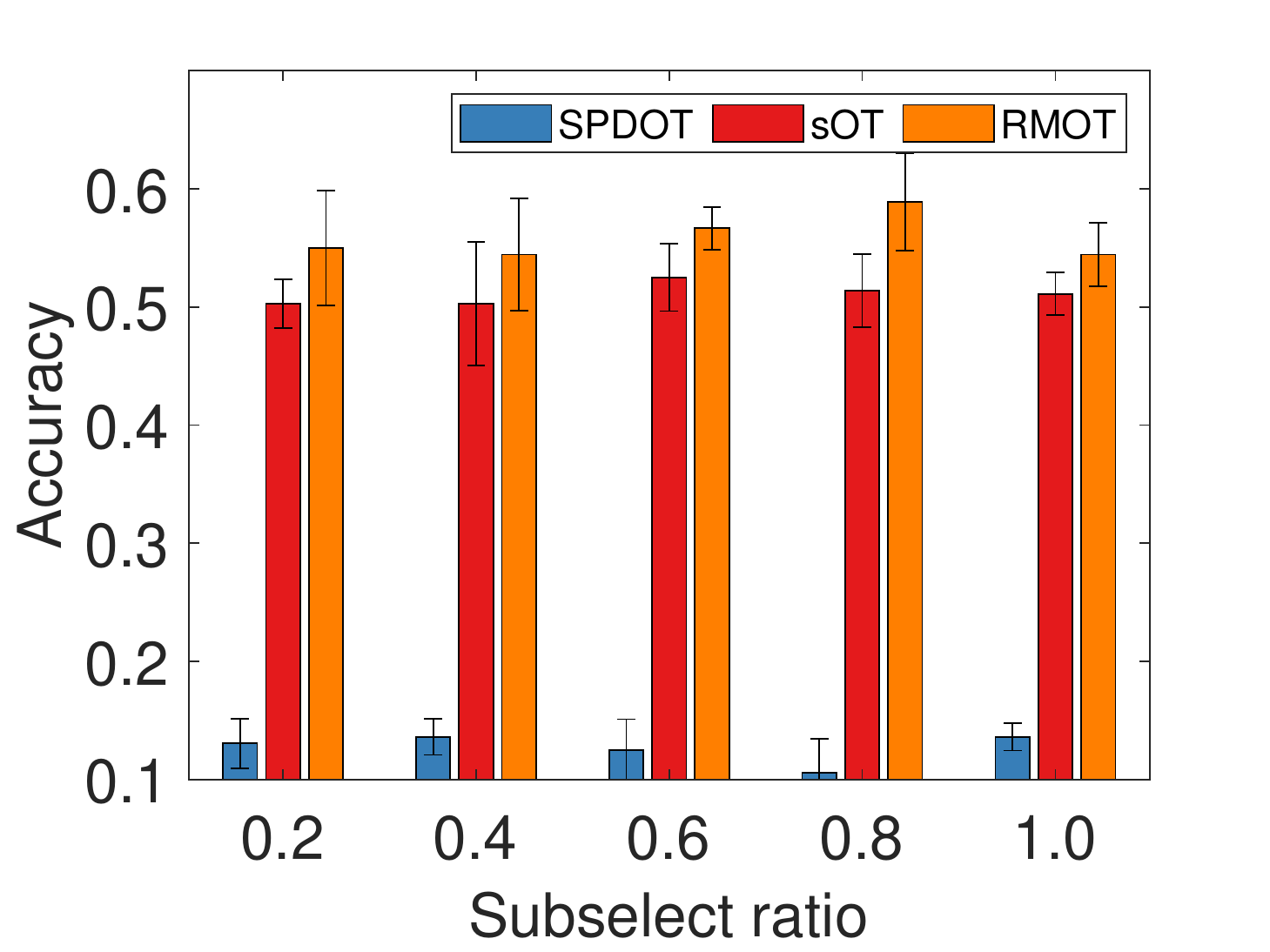}} 
    \subfloat[DynTex \label{dyntex_fig}]{\includegraphics[width = 0.33\textwidth, height = 0.25\textwidth]{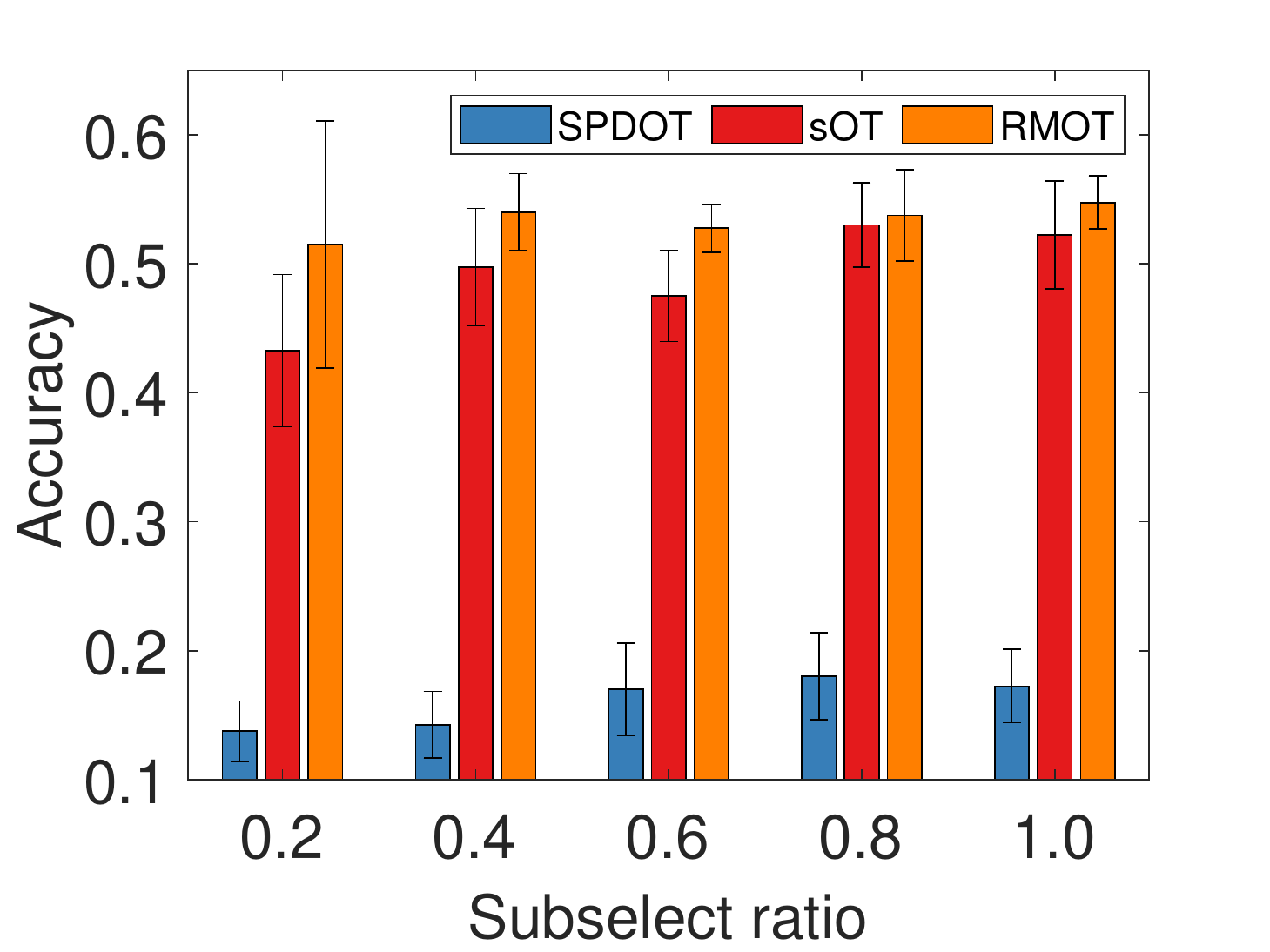}} 
    \subfloat[Cifar10 \label{cifar10_fig}]{\includegraphics[width = 0.33\textwidth, height = 0.25\textwidth]{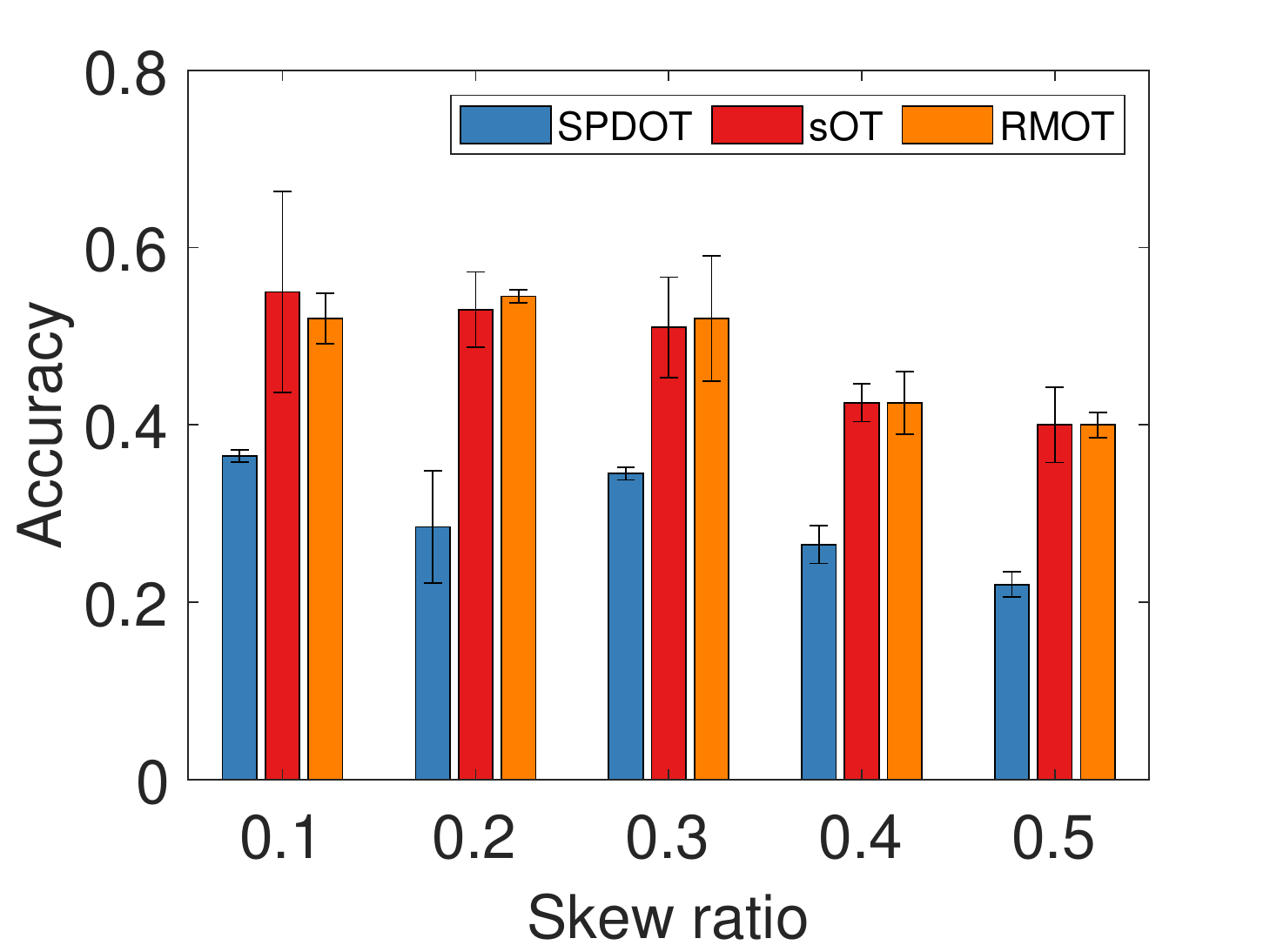}} 
    \caption{\small \revision{Additional results on domain adaptation. On these datastes as well, we see a good performance of the propose MOT modeling.}}
    \label{domain_adaptation_figure_appendix}
\end{figure}



\section{Proofs}
\label{proof_sect}

\begin{proof}[Proof of Proposition \ref{mw_distance_prop}]
For simplicity, we assume $\bp, \bq > 0$. Otherwise, we can follow \cite{peyre2019computational} to define $\tilde{p}_j = p_j$ if $p_j > 0$ and $1$ otherwise.

We note that $\bP$ and $\bQ$ are defined as $\bP\coloneqq\{[\bP_i]_{m\times 1}:\bP_i = p_i \bI\}$ and $\bQ\coloneqq\{[\bQ_j]_{n\times 1}:\bQ_j = q_i \bI\}$, where $\bI$ is the $d\times d$ identity matrix. With a slight abuse of notation and for simplicity, we define ${\rm MW}(\bp, \bq)\coloneqq {\rm MW}(\bP, \bQ)$.

First, it is easy to verify the symmetry property, i.e., ${\rm MW}(\bp, \bq) = {\rm MW}(\bq, \bp)$. For the definiteness, when $\bp = \bq$, we have $\bC_{i,i} = \bzero$ and $\bC_{i,j} \succ \bzero$ for $i \neq j$. Hence the optimal coupling is  a block diagonal matrix with $\bGamma_{i,i} = p_i \bI$. Hence ${\rm MW}(\bp, \bq) = 0$. For the opposite direction, if ${\rm MW}(\bp, \bq) = 0$, we always need to have $\bGamma_{i,j} = \bzero$, for $i \neq j$ because $\trace(\bC_{i,j} \bGamma_{i,j}) > 0$ for any $\bC_{i,j} \succ \bzero$ and $i \neq j$. Thus, $\bGamma_{i,i} \neq \bzero$, which gives $\bC_{i,i} = \bzero$ and $\bp = \bq$. 

Finally, for triangle inequality, given $\ba, \bb, \bc \in \Sigma_n$, and optimal matrix coupling $\bGamma, \bDelta$ between $(\ba, \bb)$ and $(\bb, \bc)$, respectively. That is, $\sum_j \bGamma_{i,j} = a_i \bI, \sum_i \bGamma_{i,j} = b_j \bI$ and similarly $\sum_j \bDelta_{i,j} = b_i \bI, \sum_i \bDelta_{i,j} = c_j \bI$. We now follow the same strategy by gluing the coupling $\bGamma, \bDelta$ in \cite{peyre2019computational,villani2021topics}. That is, we define a coupling $\bT$ as
\begin{equation*}
    \bT_{i,j} = \sum_k \frac{1}{2b_k} (\bGamma_{i,k} \bDelta_{k,j} + \bDelta_{k,j} \bGamma_{i,k}), \quad  \forall i,j.
\end{equation*}
We can verify $\bT_{i,j} \in \sS_{+}^d$, given $\bGamma_{i,j}, \bDelta_{i,j} \in \sS_{+}^d$. Furthermore, we have $\forall i,j$,
\begin{align*}
    \sum_{j} \bT_{i,j} &= \sum_k \frac{1}{2b_k} (\bGamma_{i,k} \sum_j \bDelta_{k,j} + \sum_j \bDelta_{k,j} \bGamma_{i,k}) = \sum_k \bGamma_{i,k} = a_i \bI, \\
    \sum_i \bT_{i,j} &= \sum_k \frac{1}{2b_k} (\sum_i\bGamma_{i,k} \bDelta_{k,j} +  \bDelta_{k,j} \sum_i \bGamma_{i,k}) = \sum_k \bDelta_{k,j} = c_j \bI.
\end{align*}
Hence, $[\bT_{i,j}]_{m \times n}$ is a valid coupling between $(\ba, \bc)$. Let $\bP_i = a_i \bI, \bQ_j = c_j \bI$ and the corresponding samples as $\bX, \bY, \bZ$ for measures $\ba, \bb, \bc$ respectively. Then,
\begin{align*}
    {\rm MW}(\ba, \bc) &= \left( \min_{\bA \in \bPi(n,n,d,\bP, \bQ) } \sum_{i,j} \trace \big( \bC_{i,j} \bA_{i,j} \big)   \right)^{1/2} \leq \left( \sum_{i,j} \trace \big( \bC_{i,j} \bT_{i,j} \big) \right)^{1/2} \\
    &= \left( \sum_{i,j, k} \frac{1}{2 b_k} \trace\Big( \bC_{i,j}  (\bGamma_{i,k} \bDelta_{k,j} + \bDelta_{k,j} \bGamma_{i,k}) \Big) \right)^{1/2} \\
    &\leq \Big( \sum_{i,j, k} \frac{1}{2 b_k} \Big(\sqrt{ \trace \big( \bC_{i,k}  (\bGamma_{i,k} \bDelta_{k,j} + \bDelta_{k,j} \bGamma_{i,k})  \big)} + \sqrt{ \trace\big( \bC_{k,j}  (\bGamma_{i,k} \bDelta_{k,j} + \bDelta_{k,j} \bGamma_{i,k}) \big) } \Big)^2 \Big)^{1/2} \\
    &\leq \Big( \sum_{i,j,k} \frac{1}{2b_k} \trace \big( \bC_{i,k} (\bGamma_{i,k} \bDelta_{k,j} + \bDelta_{k,j} \bGamma_{i,k})  \big) \Big)^{1/2} + \Big( \sum_{i,j,k} \frac{1}{2b_k} \trace \big( \bC_{k,j}  (\bGamma_{i,k} \bDelta_{k,j} + \bDelta_{k,j} \bGamma_{i,k}) \big) \Big)^{1/2}  \\
    &= \Big( \sum_{i,k} \trace\Big( \bC_{i,k} \bGamma_{i,k} \Big) \Big)^{1/2}+ \Big( \sum_{k,j} \trace\Big( \bC_{k,j} \bDelta_{k,j} \Big) \Big)^{1/2}\\
    &= {\rm MW}(\ba, \bb) + {\rm MW}(\bb, \bc),
\end{align*}
where the second inequality is by assumption (iii) of the proposition and the third inequality is due to the Minkowski inequality. This completes the proof. 
\end{proof}

\begin{proof}[Proof of Proposition \ref{smooth_set_prop}]
For a given feasible element $\bGamma \in \M^d_{m,n}(\bP ,\bQ)$, we can construct a family of feasible elements. For example, choose $0 \leq \zeta < \min_{i,j}\{ \lambda_{\min} (\bGamma_{i,j}) \}$. Then, we can add/subtract the equal number of $\zeta \bI$ and the result is still feasible. In other words, the set is smooth in a ball around the element $\bGamma$ of radius $\zeta$.
\end{proof}

\begin{proof}[Proof of Proposition \ref{orthogonal_proj}]
Following \cite{mishra2016riemannian}, the projection is derived orthogonal to the Riemannian metric \eqref{Riem_metric} as 
\begin{equation}
    {\rm P}_{\bGamma}(\bS) = \argmin_{\bU \in T_\bGamma \M_{m,n}^d} f(\bU) = - g_\bGamma(\bU, \bS) + \frac{1}{2}g_{\bGamma}(\bU, \bU). \label{argmin_projection}  
\end{equation}
The Lagrangian of problem \eqref{argmin_projection} is 
\begin{equation}
    f(\bU) - \trace(\bLambda_i \sum_j \bU_{i,j}) - \trace( \bTheta_j \sum_i \bU_{i,j}), \label{lagrangian}
\end{equation}
where $\bLambda_i$, $\bTheta_j$ are dual variables for $i \in [m], j \in [n]$. The orthogonal projection follows from the stationary conditions of \eqref{lagrangian}.
\end{proof}

\begin{proof}[Proof of Proposition \ref{Riem_grad_hess}]
Given the manifold $\M_{m,n}^d$ is a submanifold of $\bigtimes_{m, n} \sS_{++}^d$ with affine-invariant (AI) Riemannian metric, the Riemannian gradient is given by
\begin{equation*}
    \grad F(\bGamma) = {\rm P}_{\bGamma}([\grad_{\rm ai} F(\bGamma_{i,j})]_{m \times n}) = {\rm P}_{\bGamma}([ \bGamma_{i,j} \{ \nabla F(\bGamma_{i,j}) \}_{\rm S} \bGamma_{i,j}]_{m \times n}),
\end{equation*}
where $\grad_{\rm ai} F(\bX)$ is the Riemannian gradient of $\bX \in \sS_{++}^d$ with AI metric. Similarly,  the Riemannian Hessian $\hess F(\bGamma)[\bU] = \nabla_{\bU} \grad F(\bGamma)$ where $\nabla$ denotes the Riemannian connection. For submanifolds, the connection $\nabla_\bU \grad F(\bGamma) = {\rm P}_{\bGamma}([\tilde{\nabla}_{\bU_{i,j}}(\grad F(\bGamma_{i,j}))]_{m \times n})$, where $\tilde{\nabla}$ represents the connection of $\sS_{++}^d$. From \revision{\cite{sra2015conic}}, $\tilde{\nabla}_{\bU_{i,j}} \grad F(\bGamma_{i,j}) = \D \grad F(\bGamma_{i,j})[\bU_{i,j}] - \{ \bU_{i,j} \bGamma_{i,j}^{-1} \grad F(\bGamma_{i,j}) \}_{\rm S}$. Hence, the proof is complete.
\end{proof}

\begin{proof}[Proof of Theorem \ref{Prop_Riem_opt_optimal}]
We first write the Lagrange dual function as
\begin{equation*}
\begin{array}{lll}
    g(\bLambda, \bTheta, \bPsi)
    = &\min_{\bGamma = [\bGamma_{i,j}]_{m \times n} } F(\bGamma) + \sum_i \trace \Big( \bLambda_i \big( \sum_j \bGamma_{i,j} - \bP_i \big) \Big) \\
   \qquad \qquad \qquad \qquad & + \sum_j \trace\Big(  \bTheta_j \big( \sum_{i} \bGamma_{i,j} - \bQ_j \big) \Big)  - \sum_{i,j} \trace \Big( \bPsi_{i,j} \bGamma_{i,j} \Big),
    \end{array}
\end{equation*}
where we relax the SPD constraint on $\bGamma_{i,j}$ to the semidefinite constraint, i.e. $\bGamma_{i,j} \succeq \bzero$, for some dual variable $\bLambda_i, \bTheta_j \in \sS^d$ and $\bPsi_{i,j} \succeq \bzero$. Given the function $F$ is convex with non-empty constraint set, by Slater's condition, strong duality holds and the primal and dual variables should jointly satisfy the KKT conditions. 

First, we notice by complementary slackness, $\trace (\bPsi_{i,j}^* \bGamma_{i,j}^*) = 0$ for $\bGamma_{i,j}^* \succ \bzero$. This implies that $\bPsi_{i,j}^* =\bzero$ since $\bPsi_{i,j}^* \succeq \bzero$. Note that in some cases $\bGamma_{i,j}^*$ may be rank-deficient (i.e., some eigenvalues are close to zero), which gives rise to non-zero $\bPsi_{i,j}^*$. Regardless, from the optimality condition, it always satisfies for optimal $\bGamma_{i,j}^*$, $\bLambda_i^*$, $\bTheta_j^*$,
\begin{equation}
    \bGamma_{i,j}^* (\nabla F(\bGamma_{i,j}^*) + \bLambda_i^* + \bTheta_j^*) \bGamma_{i,j}^* = \bzero, \label{main_kkt_cond}
\end{equation}
due to that $\bGamma_{i,j}^*$ is orthogonal to $\bPsi_{i,j}^*$. $\nabla F(\bGamma_{i,j}^*)$ denotes the block partial derivative of $F$ with respect to $\bGamma_{i,j}$ at optimality. On the other hand, to perform Riemannian optimization, the Riemannian gradient is first computed for the primal objective $F$ as  
\begin{align*}
    {\rm grad} F(\bGamma) = {\rm P}_\bGamma( [\bGamma_{i,j} \{ \nabla F(\bGamma_{i,j}) \}_{\rm S} \bGamma_{i,j} ]_{m\times n}), 
\end{align*}
which from the definition of orthogonal projection, gives 
\begin{equation*}
    {\rm grad} F(\bGamma_{i,j}) = \bGamma_{i,j} \Big( \nabla F(\bGamma_{i,j}) + \tilde{\bLambda}_i + \tilde{\bTheta}_j  \Big) \bGamma_{i,j},
\end{equation*}
where ${\rm grad} F(\bGamma_{i,j})$ represents the Riemannian partial derivative and $\tilde{\bLambda}_i, \tilde{\bTheta}_j \in \sS^d$ are computed such that 
\begin{align}
\begin{cases}
    \sum_{i} \bGamma_{i,j} ( \bC_{i,j} + \epsilon \nabla \Omega(\revision{\bGamma_{i,j}}) + \tilde{\bLambda}_i + \tilde{\bTheta}_j) \bGamma_{i,j} = \bzero, & \forall j \\
    \sum_{j} \bGamma_{i,j} (\bC_{i,j} + \epsilon \nabla \Omega(\revision{\bGamma_{i,j}}) + \tilde{\bLambda}_i + \tilde{\bTheta}_j) \bGamma_{i,j} = \bzero, & \forall i.
\end{cases}
\label{gradient_proj_cond}
\end{align}
Comparing \eqref{gradient_proj_cond} to \eqref{main_kkt_cond}, we see that at optimality, there exists $\bLambda_i^*, \bTheta_j^*$ such that for all $i, j$, the conditions \eqref{gradient_proj_cond} are satisfied, with $\tilde{\bLambda}_i =  \bLambda_i^* + \Delta, \tilde{\bTheta}_j = \bTheta_j^* - \Delta$, for any symmetric matrix $\Delta$, i.e., the Riemannian gradient ${\rm grad} F(\bGamma_{i,j}^*) = \bzero$, thus completing the proof. 
\end{proof}

\begin{proof}[Proof of Proposition \ref{dual_var_solution_prop}]
For each regularized OT problem, we consider the Lagrange dual problem of $\min_{\bGamma \in \M_{m,n}^d} {\rm MW}_\epsilon(\bar{\bP}, \bP^\ell)$, which is given as 
\begin{equation}
    \begin{array}{lll}
    \L_{{\rm MW}_\epsilon} = & \max\limits_{\substack{\bLambda_i^\ell, \, \bTheta_j^\ell \in \sS^d \\ \bPsi_{i,j}^\ell \succeq \bzero}}  \min\limits_{\bGamma_{i,j}^\ell} \sum\limits_{i,j} \big( \trace(\bC_{i,j}^\ell \bGamma_{i,j}^\ell) + \epsilon \, \Omega(\bGamma_{i,j}^\ell) \big)  + \sum\limits_i \trace\big( \bLambda_i^\ell \big(\sum\limits_j \bGamma_{i,j}^\ell - \bar{\bP}_i \big) \big) \\
    & \qquad \qquad \qquad + \sum\limits_j \trace \big( \bTheta_j^\ell \big( \sum\limits_{i} \bGamma_{i,j}^\ell - \bP^\ell_j \big) \big) - \sum\limits_{i,j} \trace \big( \bPsi_{i,j}^\ell \bGamma_{i,j}^\ell \big)  . \label{dual_problem_dual_var}
    \end{array}
\end{equation}
From the Lagrangian (\ref{dual_problem_dual_var}), it is easy to see the Euclidean gradient of the barycenter problem with respect to $\bar{\bP}_i$ is $-\sum_\ell \bLambda_i^\ell$ with the dual optimal $\bLambda_i^\ell$ for problem \eqref{dual_problem_dual_var}. The proof is complete by substituting the objective $F(\bGamma) = \sum_{i,j} \big( \trace(\bC_{i,j} \bGamma_{i,j} + \epsilon \Omega(\bGamma_{i,j})) \big)$ as in Theorem~\ref{Prop_Riem_opt_optimal}.
\end{proof}

\begin{proof}[Proof of Proposition \ref{average_distance_update}]
First we rewrite SPD matrix-valued GW discrepancy as 
\begin{equation*}
\begin{array}{lll}
    &{\rm MGW}\big((\bar{\bD}, \bar{\bP}), (\bD^\ell, \bP^\ell) \big) \\
    &= \sum_{i,i',j,j'} \Big(  f_1(\bar{D}_{i,i'}) + f_2(D^\ell_{j,j'}) -h_1(\bar{D}_{i,i'}) h_2({D}^\ell_{j,j'}) \Big) \trace(\bGamma^\ell_{i,j} \bGamma^\ell_{i',j'}) \\
    &= \sum_{i,j} \trace \Big( \bGamma_{i,j} \Big( \sum_{i'} f_1(\bar{D}_{i,i'}) \bar{\bP}_{i'}  + \sum_{j'} f_2(D^\ell_{j,j'}) \bP^\ell_{j'} - \sum_{i'} h_1(\bar{D}_{i,i'}) \sum_{j'} h_2(D^\ell_{j,j'}) \bGamma^\ell_{i',j'} \Big)  \Big),
\end{array}    
\end{equation*}
where we use the fact that $\bGamma^\ell_{i,j}$ are optimal and satisfy the constraints $\sum_j \bGamma^\ell_{i,j} = \bar{\bP}_i$ and $\sum_i \bGamma^\ell_{i,j} = \bP^\ell_j$. By the first order condition, we have 
\begin{equation*}
    \trace \Big( \bar{\bP}_i  f_1'(\bar{D}_{i,i'}) \bar{\bP}_{i'} - h_1'(\bar{D}_{i,i'}) \sum_j \bGamma_{i,j}  \big( \sum_{j'} h_2(D^\ell_{j,j'}) \bGamma^\ell_{i',j'} \big) \Big) = 0, \quad \forall i,i' \in [m],
\end{equation*}
which gives the desired result.
\end{proof}

\section{Riemannian geometry for block SPD Wasserstein barycenter}
\label{geometry_spd_simplex_appendix}

{\bf Riemannian geometry of $\Delta_n (\sS_{++}^d)$.} In \cite{mishra2019riemannian}, the authors endow a Riemannian manifold structure for the set $\Delta_n (\sS_{++}^d) \coloneqq \{ \bP = [\bP_i]_{n \times 1} : \sum_i \bP_i = \bI \}$. 
Its tangent space is given by 
$
    T_\bP \Delta_n (\sS_{++}^d) = \{ (\bU_1, ..., \bU_n) : \bU_i \in \sS^d, \sum_i \bU_i = \bzero \}.
$
By introducing the affine-invariant metric $\langle \bU, \bV \rangle_\bP = \sum_i \trace(\bP_i^{-1} \bU_i \bP_i^{-1} \bV_i)$,  $\Delta_n (\sS_{++}^d)$ has a submanifold structure. The retraction from the tangent space to the manifold is derived as 
\begin{equation*}
    R_\bP(\bU) = (\hat{\bP}^{-1/2}_{\rm sum} \hat{\bP}_1 \hat{\bP}^{-1/2}_{\rm sum}, ..., \hat{\bP}^{-1/2}_{\rm sum} \hat{\bP}_n \hat{\bP}^{-1/2}_{\rm sum}),
\end{equation*}
where $\hat{\bP}_i = \bP_i (\exp(\bP_i^{-1} \bU_i) )$ and $\hat{\bP}_{\rm sum} = \sum_i \hat{\bP}_i$. 

The Riemannian gradient of a function $F:\Delta_n (\sS_{++}^d) \to \mathbb{R}$ is computed as 
\begin{equation*}
    \grad F(\bP) = {\rm Proj}_{\bP} \Big( \big(\bP_1 \{ \nabla F(\bP_1) \}_{\rm S} \bP_1, ..., \bP_n \{ \nabla F(\bP_n) \}_{\rm S} \bP_n \big) \Big),
\end{equation*}
where the orthogonal projection $\rm{P}_{\bP}$ of a of $\bS = (\bS_1,  \bS_2, ..., \bS_n)$ such that $\bS_i \in \sS^d$ is
\begin{equation*}
    {\rm Proj}_\bP(\bS) = (\bS_1 + \bP_1 \bLambda \bP_1, ..., \bS_n + \bP_n \bLambda \bP_n),
\end{equation*}
where $\bLambda \in \sS^d$ is the solution to the linear equation $\sum_i \bP_i \bLambda \bP_i = - \sum_i \bS_i$.

{\bf Optimization for Wasserstein barycenter.} With the Riemannian geometry defined for the simplex of SPD matrices, we can update the barycenter by Riemannian optimization as shown in Algorithm \ref{wb_algorithm}.

\begin{algorithm}[t]
 \caption{\small Block-SPD Wasserstein barycenter}
 \label{wb_algorithm}
 \begin{algorithmic}[1]
  \STATE \textbf{Input:} block-SPD marginals $\{\bP^\ell\}_{\ell = 1}^K$, cost matrices $\{ \bC^\ell\}_{\ell=1}^K$.
  \STATE Initialize $\bar{\bP}^{0}$ as uniform distribution.
  \FOR{$t = 1,...,T$}
  \FOR{$\ell = 1,...,K$}
  \STATE Compute $(\bLambda^\ell)^*$ as in Proposition \ref{dual_var_solution_prop}.
  \ENDFOR
  \STATE Compute the Riemannian gradient of \eqref{w_barycenter} by orthogonally projecting $\sum_\ell \omega_\ell (\bLambda^\ell)^*$ onto the tangent space of $\Delta_n(\sS_{++}^d)$.
  \STATE Update $\bar{\bP}^t$ by retracting on $\Delta_n(\sS_{++}^d)$ with a step size.
  \ENDFOR
  \STATE \textbf{Output:} Barycenter $\bar{\bP}^T$.
 \end{algorithmic} 
\end{algorithm}

\end{document}